\documentclass[12pt]{amsart}

\usepackage{amsmath,amsthm,amsfonts,amscd,amssymb,eucal,latexsym,mathrsfs, appendix, yhmath, fancyhdr, setspace}
\usepackage[numbers,sort&compress]{natbib}
\usepackage[all,cmtip]{xy}

\usepackage[colorlinks,linkcolor=red,citecolor=blue,urlcolor=blue,hypertexnames=true]{hyperref}

\newtheorem{theorem}{Theorem}[section]

\newtheorem{lemma}[theorem]{Lemma}
\newtheorem{proposition}[theorem]{Proposition}

\newtheorem*{theorem*}{Theorem}

\theoremstyle{definition}
\newtheorem{definition}[theorem]{Definition}
\newtheorem{remark}[theorem]{Remark}
\newtheorem{notation}[theorem]{Notation}
\newtheorem{example}[theorem]{Example}
\newtheorem{question}[theorem]{Question}

\numberwithin{equation}{section}

\newcommand{\htopol}{{\text{\rm h}}_{\text{\rm top}}}

\newcommand{\im}{{\rm im}}

\newcommand{\id}{{\rm id}}

\newcommand{\cB}{{\mathcal B}}

\newcommand{\cW}{{\mathcal W}}
\newcommand{\cD}{{\mathcal D}}
\newcommand{\cE}{{\mathcal E}}
\newcommand{\cC}{{\mathcal C}}
\newcommand{\cU}{{\mathcal U}}
\newcommand{\cM}{{\mathcal M}}
\newcommand{\cN}{{\mathcal N}}

\newcommand{\cO}{{\mathcal O}}

\newcommand{\cF}{{\mathcal F}}

\newcommand{\cX}{{\mathcal X}}
\newcommand{\cY}{{\mathcal Y}}
\newcommand{\Cb}{{\mathbb C}}

\newcommand{\Zb}{{\mathbb Z}}
\newcommand{\Tb}{{\mathbb T}}
\newcommand{\Qb}{{\mathbb Q}}
\newcommand{\Rb}{{\mathbb R}}
\newcommand{\Nb}{{\mathbb N}}

\newcommand{\sD}{{\mathscr D}}
\newcommand{\sE}{{\mathscr E}}

\newcommand{\tr}{{\rm tr}}

\newcommand{\GL}{{\rm GL}}

\newcommand{\rh}{{\rm h}}

\newcommand{\ddet}{{\rm det}}

\newcommand{\odd}{{\rm odd}}
\newcommand{\even}{{\rm even}}

\allowdisplaybreaks

\begin{document}

\title{Entropy, Determinants, and $L^2$-Torsion}

\author{Hanfeng Li}
\author{Andreas Thom}
\address{\hskip-\parindent
H.L., Department of Mathematics, Chongqing University,
Chongqing 401331, China.
Department of Mathematics, SUNY at Buffalo,
Buffalo, NY 14260-2900, U.S.A.}
\email{hfli@math.buffalo.edu}

\address{\hskip-\parindent
A.T., 	Mathematisches Institut,
Univ. Leipzig,
PF 100920,
04009 Leipzig, Germany.}
\email{thom@math.uni-leipzig.de}

\date{March 5, 2013.
}

\begin{abstract}
We show that for any amenable group $\Gamma$ and any $\Zb \Gamma$-module $\cM$ of type FL with vanishing Euler characteristic, the entropy of the natural $\Gamma$-action on the Pontryagin dual of ${\cM}$ is equal to the $L^2$-torsion of $\cM$. As a particular case, the entropy of the principal algebraic action associated with the module $\Zb \Gamma / \Zb \Gamma f$ is equal to the logarithm of the Fuglede-Kadison determinant of $f$ whenever $f$ is a non-zero-divisor in $\Zb \Gamma$. This confirms a conjecture of Deninger.
As a key step in the proof we provide a general Szeg\H{o}-type approximation theorem for the Fuglede-Kadison determinant on the group von Neumann algebra of an amenable group.

As a consequence of the equality between $L^2$-torsion and entropy, we show that the $L^2$-torsion of a non-trivial amenable group with finite classifying space vanishes. This was conjectured by L\"uck. Finally, we establish a Milnor-Turaev formula for the $L^2$-torsion of a finite $\Delta$-acyclic chain complex.
\end{abstract}

\maketitle

\tableofcontents

\section{Introduction} \label{S-introduction}

This article is concerned with the three interacting topics: \emph{entropy, determinants}, and \emph{$L^2$-torsion}.
Let $\Gamma$ be a countable discrete amenable group (Section~\ref{SS-amenable group}) and let
$$\ddet_{\cN \Gamma} \colon \cN \Gamma \to \Rb_{\geq 0}, \quad \ddet_{\cN \Gamma}f = \exp \left( \int_{\Rb} \log(t) d \mu_{|f|}(t) \right)$$
be the Fuglede-Kadison determinant (Section~\ref{SS-determinant}), defined on the group von Neumann algebra $\cN \Gamma$.
Let $f$ be a non-zero-divisor in the integral group ring $\Zb \Gamma$ of $\Gamma$ and denote the left ideal generated by $f$ in $\Zb \Gamma$ by $\Zb \Gamma f$. The Pontryagin dual $X_f$ of the quotient $\Zb \Gamma / \Zb \Gamma f$ is a compact abelian group and admits a natural continuous $\Gamma$-action. We call such an action a \emph{principal algebraic action}. In 2006, Deninger \cite{Deninger06} started a program to compute the entropy of principal algebraic actions of a countable discrete amenable group in terms of the Fuglede-Kadison determinant. Classical results by Yuzvinski\u\i \ \cite{yup} for $\Gamma = \Zb$ and by Lind-Schmidt-Ward \cite{LSW, Schmidt} for $\Gamma=\Zb^d$ showed that the entropy of the $\Gamma$-action on $X_f$, which we denote by $\rh(X_f)$, is equal to $\log\ddet_{\cN \Gamma}f$, which can be identified with the logarithm of the Mahler measure of $f$. Deninger conjectured that this equality extends to all amenable groups.
In order to prove this for general amenable groups, Deninger \cite{Deninger06} developed important new techniques and confirmed this equality assuming that $\Gamma$ is finitely generated and of polynomial growth, and that $f$ is positive in $\cN \Gamma$ and invertible in $\ell^1(\Gamma)$. Later Deninger-Schmidt \cite{DS} showed the equality in the case $\Gamma$ is amenable and residually finite and $f$ is invertible in $\ell^1 (\Gamma)$.
The first author has weakened the assumptions and proved the equality between the logarithm of the Fuglede-Kadison determinant and the entropy of the associated principal algebraic action for all $f \in \Zb\Gamma$ invertible in $\cN\Gamma$ \cite{Li}.

More generally, one may want to study the action of $\Gamma$ on the Pontryagin dual $\widehat{\cM}$ of an arbitrary $\Zb \Gamma $-module $\cM$. We call such an action an \emph{algebraic action}. In the case $\Gamma=\Zb^d$, the results of Lind-Schmidt-Ward were enough to determine the entropy of \emph{any} algebraic $\Zb^d$-action in terms of determinants \cite[\S 4]{LSW}. This was essentially due to the fact that $\Zb[\Zb^d]$-modules have a tractable structure theory and decompose nicely. For general amenable groups, this fails to be true. Following Serre, we say that a left $\Zb \Gamma$-module $\cM$ is of type FL if it admits a finite free resolution:
\begin{align*}
0 \to (\Zb\Gamma)^{d_k} \rightarrow \cdots \rightarrow (\Zb\Gamma)^{d_1} \rightarrow (\Zb\Gamma)^{d_0}  \rightarrow \cM\rightarrow 0.
\end{align*}
Note that modules of the form $\cM=\Zb \Gamma/\Zb \Gamma f$, for $f$ being a non-zero-divisor in $\Zb \Gamma$, are of type FL since the sequence
$$0 \to \Zb \Gamma \stackrel{f}{\to} \Zb \Gamma \to \cM \to 0$$
is exact. If the Euler characteristic (see Section~\ref{SS-euler}) $\chi(\cM) = \sum_i (-1)^i d_i$ vanishes, then the $L^2$-torsion $\rho^{(2)}(\cM)$ of $\cM$ -- a real number -- can be defined, see Section~\ref{SS-torsion}. The $L^2$-torsion is a natural secondary invariant for modules of type FL and is defined in terms of the Fuglede-Kadison determinants of the Laplace operators associated with a resolution as above.
The relation between determinants and torsion is classical and has found many nice applications in topology and algebra. Reidemeister torsion and Whitehead torsion are indispensable tools in algebraic topology \cite{Cohen, Milnor66, Turaev, turaevbook}. $L^2$-torsion was first introduced in \cite{carmat, LR} and has further enlarged the range of applications, see \cite[Chapter 3]{Luck} for an overview and a detailed description. Through its relationship with the analytic Ray-Singer torsion it is related to interesting problems in analysis and geometry.
In Section~\ref{S-torsion} we give a self-contained account on $L^2$-torsion, and show that it can be viewed as a completely classical torsion theory -- much in the spirit of classical Reidemeister torsion.
The discussion is based on the ring $\cN \Gamma^{\Delta}$ introduced by Haagerup-Schultz \cite{HS}.

For any $\Zb\Gamma$-module $\cM$ of type FL with $\chi(\cM)=0$, we establish the equality of the entropy of the action of $\Gamma$ on $\widehat \cM$ and the $L^2$-torsion of $\cM$. This is our main result. More precisely, we show the following result; see the definitions in Section~\ref{S-preliminary}.

\begin{theorem} \label{T-torsion}
Let $\Gamma$ be a countable discrete amenable group and let $\cM$ be a left $\Zb\Gamma$-module of type FL$_k$ for some $k\in \Nb$ with $\dim_{\cN\Gamma}(\cN\Gamma\otimes_{\Zb\Gamma}\cM)=0$. Let $\cC_* \to \cM$ be a partial resolution of $\cM$ by based finitely generated free left $\Zb\Gamma$-modules as in \eqref{E-resolution 0}. Then, we have
\begin{align} \label{E-entropy torsion1}
 (-1)^k\rh(\widehat{\cM})\ge (-1)^k \rho^{(2)}(\cC_*).
\end{align}
If furthermore $\cM$ is of type FL with $\chi(\cM)=0$ and $\cC_*\rightarrow \cM$ is a resolution of $\cM$ by finitely generated free left $\Zb\Gamma$-modules as in \eqref{E-resolution 2}, then
\begin{align} \label{E-entropy torsion2}
\rh(\widehat \cM) = \rho^{(2)}(\cM).
\end{align}
\end{theorem}

We expect that Equality \eqref{E-entropy torsion2} in Theorem~\ref{T-torsion} will have numerous applications in both directions. Even though the aim of Deninger's program was to provide tools to compute the entropy of particular actions, it turns out that the final outcome is just as useful to compute the $L^2$-torsion of particular $\Zb\Gamma$-modules. Note that the different sides of the equation between entropy and $L^2$-torsion are of completely different nature. There is for example a priori no reason to think that the $L^2$-torsion of a $\Zb \Gamma$-module is always non-negative; for the entropy, this is obvious from the definition. On the other hand, it is known that the right side of the equation in the case $\Gamma=\Zb^d$ is often given by polylogarithms and special values of $L$-functions, and is related to regulators from algebraic geometry \cite{deningerregulator}. This arithmetic property of the possible values of the $L^2$-torsion is well-studied for $\Gamma = \Zb^d$ \cite{mixedmot}, but remains largely unexplored for non-commutative $\Gamma$. Neither in the commutative nor in the non-commutative case, this property is expected a priori for the possible values of the entropy.

A confirmation of the conjectured equality of the entropy of a principal algebraic action and the logarithm of the Fuglede-Kadison determinant arises as a special case of Theorem~\ref{T-torsion}. For $f \in M_{d'\times d}(\Zb \Gamma)$, we denote by $\ker f \subseteq (\ell^2(\Gamma))^{d}$ the kernel of left-multiplication with $f$, see Section~\ref{SS-group rings}.

\begin{theorem} \label{T-main} Let $\Gamma$ be a countable discrete amenable group and
let $f\in M_{d'\times d}(\Zb \Gamma)$ with $\ker f=\{0\}$. Let $X_f$ be the Pontryagin dual of $(\Zb \Gamma)^{1 \times d} / (\Zb \Gamma)^{1 \times d'} f$ with its natural $\Gamma$-action. Then
$$\rh(X_f)\le \log \ddet_{\cN\Gamma}f.$$
If furthermore $d'=d$, then
$$\rh(X_f)=\rh(X_{f^*})=\log \ddet_{\cN\Gamma}f.$$
\end{theorem}

Note that the case $\ker f \neq \{0\}$ is more pathological.
Indeed, for $f\in M_{d'\times d}(\Zb\Gamma)$ with $\ker f\neq \{0\}$, one has $\rh(X_f)=\infty$ \cite[Theorem 4.11]{CL}. Similarly, in the context of Theorem~\ref{T-torsion}, we have that $\dim_{\cN \Gamma} (\cN \Gamma \otimes_{\Zb \Gamma} \cM) \neq 0$ or $\chi(\cM) \neq 0$ implies $\rh(\widehat{\cM})= \infty$, see Remark~\ref{R-vanishdim}.

One can prove Theorem~\ref{T-torsion} directly, but we choose to prove Theorem~\ref{T-main} first, because all the technical difficulties arise already in this simpler special case. We then use Theorem~\ref{T-main} to establish Theorem~\ref{T-torsion}.

\vspace{0.2cm}

Theorem~\ref{T-torsion} turns out to be very useful for the computation of the $L^2$-torsion for specific modules. Throughout, $B\Gamma$ denotes a CW-complex, whose homotopy groups but the first one are all trivial and $\pi_1(B\Gamma)=\Gamma$. Such a space exists and is unique up to homotopy equivalence, see \cite[Section I.4]{Brown} where the notation $K(\Gamma,1)$ is used to emphasize that $B\Gamma$ is an Eilenberg-MacLane space. We say that there is \emph{a finite model for $B\Gamma$} if we can choose it to be a CW-complex with finitely many cells. If the group $\Gamma$ has a finite model for $B \Gamma$, then the trivial $\Zb \Gamma$-module $\Zb$ is of type FL \cite[Proposition VIII.6.3]{Brown}. Indeed, the cellular chain complex of the universal covering space of $B \Gamma$ is easily seen to be a finite free resolution of  $\Zb$. If $\Gamma$ is infinite and the trivial $\Zb \Gamma$-module $\Zb$ of type FL, then its $L^2$-torsion can be defined. It is called the \emph{$L^2$-torsion} of $\Gamma$ and denoted by $\rho^{(2)}(\Gamma)$.
Theorem~\ref{T-torsion} implies the following result:

\begin{theorem} \label{T-Luck conjecture}
If $\Gamma$ is a non-trivial amenable group such that the trivial $\Zb \Gamma$-module $\Zb$ is of type FL, then $\rho^{(2)}(\Gamma)=0.$
\end{theorem}
This confirms a conjecture of L\"uck, see Conjecture 9.24 in \cite{Luck02}, Conjecture 11.3 in \cite{Luck} and the remark after Corollary 1.11 in \cite{lsweg}.
Conjecture 9.24 in \cite{Luck02} and Conjecture 11.3 in \cite{Luck} talk about the $L^2$-torsion $\rho^{(2)}(\tilde{X})$ of the universal covering space $\tilde{X}$ of an aspherical connected closed manifold $X$ or more generally an aspherical connected finite CW-complex $X$ whose fundamental group $\Gamma$ contains a non-trivial normal amenable subgroup. In such case $\Gamma$ must be infinite and one can take $X$ as $B\Gamma$. It follows that $\rho^{(2)}(\tilde{X})=\rho^{(2)}(\Gamma)$.

Wegner \cite{weg} proved Theorem~\ref{T-Luck conjecture} for elementary amenable groups using the structure theory of this class of groups. Our method is completely different and gives this result as part of a much larger picture.

The distinction between having a finite model for $B\Gamma$ and having the trivial $\Zb\Gamma$-module $\Zb$ of type FL is rather subtle.
If $\Gamma$ has a finite model for $B\Gamma$, then $\Gamma$ is finitely presented \cite[Corollary 3.1.17]{Geoghegan} and the trivial left $\Zb\Gamma$-module $\Zb$ is of type FL \cite[Proposition VIII.6.3]{Brown}. Conversely, Eilenberg-Ganea \cite{EG} and Wall \cite{Wall65, Wall66} proved that if
$\Gamma$ is finitely presented and the trivial left $\Zb\Gamma$-module $\Zb$ is of  type FL, then $\Gamma$ must have a finite model for $B\Gamma$ \cite[Theorem VIII.7.1]{Brown}. Note that Bestvina and Brady \cite{BB} have constructed examples of groups $\Gamma$ for which the trivial $\Zb\Gamma$-module $\Zb$ is of type FL but which are not finitely presented. These examples are not amenable. At the same time, \cite[Corollary 1.2]{krop} says that every elementary amenable group $\Gamma$ for which the trivial $\Zb\Gamma$-module $\Zb$ is of type FL has a finite model for $B\Gamma$.

\vspace{0.2cm}

In order to prove Theorem~\ref{T-main}, we show a Szeg\H{o}-type approximation theorem for the Fuglede-Kadison determinant on the group von Neumann algebra of an amenable group.
We prove a general approximation theorem for the Fuglede-Kadison determinant by determinants of finite-dimensional matrices, arising from a F\o lner approximation of $\Gamma$. The most classical such theorem was proved by Szeg\H{o} \cite{Szego} for Toeplitz matrices and generalized as follows \cite[Theorem 2.7.14]{Simon}.

\begin{theorem*}
Let $f$ be an essentially bounded $\Rb_{\ge 0}$-valued measurable function on the unit circle $S^1$. Then,
$$\exp\left( \int_{S^1} \log f(z) d\mu(z) \right) = \lim_{n \to \infty} (\det(D_n))^{1/n},$$
where $\mu$ denotes the Haar probability measure on $S^1$ and $D_n$ denotes the $n \times n$-matrix with entries $(D_n)_{i,j} = \int_{S^1} f(z) z^{i-j} d\mu(z)$.
\end{theorem*}

Following Deninger \cite[Section 2]{Deninger09}, we interpret this result as an approximation of the Fuglede-Kadison determinant on the group von Neumann algebra of the group $\Zb$ by determinants associated with F{\o}lner sets $\{1,\dots,n\} \subseteq \Zb$, see Example~\ref{Ex-determinant}. Our generalization holds for every amenable group $\Gamma$, every F{\o}lner sequence, and every positive element in the group von Neumann algebra of $\Gamma$. We refer to Section~\ref{SS-amenable group} for the necessary definitions.

\begin{theorem} \label{T-approximate}
Let $\Gamma$ be a countable discrete amenable group. Let $g\in M_d(\cN\Gamma)$ be positive.
Then
\begin{align} \label{E-approximate}
 \ddet_{\cN\Gamma} g=\inf_{F\in \cF(\Gamma)}(\det(g_F))^{\frac{1}{|F|}}=\lim_F (\det(g_F))^{\frac{1}{|F|}}.
\end{align}
\end{theorem}

For $d=1$, this is a positive answer to a question of Deninger \cite[Question 6]{Deninger09}.
To the best of our knowledge, this result is even new for $\Gamma=\Zb^2$ and $d=1$. A Szeg\H{o}-type result for essentially bounded, measurable matrix-valued functions on $S^1$ was known \cite[Theorem 2.13.5]{Simon}.

\vspace{0.2cm}

We can use Theorem~\ref{T-torsion} to define the \emph{torsion} of a countable $\Zb \Gamma$-module $\cM$ to be the entropy of its Pontryagin dual; we denote the torsion of $\cM$ by $\rho(\cM)$. If $\cM$ is finitely presented, this number is finite only if $\dim_{\cN \Gamma}(\cN \Gamma \otimes_{\Zb \Gamma} \cM)=0$ (see Remark \ref{R-vanishdim}) and thus can been understood as a natural secondary invariant for $\Zb \Gamma$-modules.
We now can study the $L^2$-torsion of a $\Delta$-acyclic chain complex  $\cC_*$ (see Section~\ref{S-torsion}) and prove a Milnor-Turaev formula for $L^2$-torsion. Since the $L^2$-torsion turns out to be a homotopy invariant of the chain complex $\cC_*$, it is natural to expect that it can be expressed in terms of the homology $H_*(\cC_*)$ of the chain complex. We show:

\begin{theorem} \label{T-milnorturaev} Let $\Gamma$ be a countable discrete amenable group.
Let $\cC_*$ be a chain complex of finitely generated free left $\Zb \Gamma$-modules of finite length. Assume that $\cC_*$ is $\Delta$-acyclic (defined in Section~\ref{SS-newview}) or, equivalently, that $\ell^2 (\Gamma) \otimes_{\Zb \Gamma} \cC_*$ is weakly acyclic (see Section~\ref{SS-torsion}).
Then $\rho(H_i(\cC_*))< \infty$ for all $i \in \Zb$ and
$$\rho^{(2)}(\cC_*) = \sum_{i \in \Zb} (-1)^i \rho(H_i(\cC_*)).$$
\end{theorem}

We expect that Theorem~\ref{T-milnorturaev} will have interesting applications in algebraic topology. It shows that $L^2$-torsion can be thought of as a generalized Euler characteristic where the r\^{o}le of the ordinary ($L^2$-)Betti numbers is played by the torsion of the homology groups.

\vspace{0.2cm}

Recently, the entropy theory has been extended to actions of countable sofic groups \cite{Bowen10, KL11}, which include all countable amenable groups and countable residually finite groups. The analogue of Theorem~\ref{T-main} for countable residually finite groups has been established for some special cases in \cite{Bowen11, BL, KL11}, though the general case is still open.

\vspace{0.2cm}

The paper is organized as follows:
Section~\ref{S-preliminary} deals with preliminaries on notation and gives a brief introduction to
group rings, the Fuglede-Kadison determinant, amenable groups, entropy, the Euler characteristic, and $L^2$-torsion.

Section~\ref{S-approximation} contains a brief history of the approximation results of the Fuglede-Kadison determinant and our first main result: Theorem~\ref{T-approximate}, the approximation theorem for the Fuglede-Kadison determinant. This section also contains a uniform estimate of the spectral measure near zero in case of a non-vanishing determinant, Proposition~\ref{P-measure weak convergence}.

Section~\ref{S-entaction} contains our second main result: Theorem~\ref{T-main}, the computation of the entropy of principal algebraic actions in terms of the Fuglede-Kadison determinant. This section is the most technical one. We give a new formula for the entropy of a finitely generated algebraic action, this is Theorem~\ref{T-approximate solution formula for entropy}. In Section~\ref{SS-positive}, we give a proof of Theorem~\ref{T-main} in the positive case. Finally, we prove the general case of Theorem~\ref{T-main} on the basis of a formula for entropy that was obtained by Peters, Theorem~\ref{T-Peters}.

The first part of Section~\ref{S-entropy-torsion} is concerned with the proof of Theorem~\ref{T-torsion}. In Section~\ref{SS-application to torsion} we give applications of Theorem~\ref{T-torsion}. We prove Theorem~\ref{T-Luck conjecture} and show in Theorem~\ref{T-infinite index} that the $L^2$-torsion of every $\Zb \Gamma$-module of type FL, which is finitely generated as an abelian group, vanishes, if $\Gamma$ contains $\Zb$ as a subgroup of infinite index. Section~\ref{SS-milnorturaev} contains the definition of torsion for general $\Zb\Gamma$-modules. There, we prove the Milnor-Turaev formula; this is Theorem~\ref{T-milnorturaev}.

Section~\ref{S-torsion} is a self-contained introduction to $L^2$-torsion, assuming only the classical work of Milnor~\cite{Milnor66}. We review the definition of Whitehead torsion. After introducing the Haagerup-Schultz algebra we define $L^2$-torsion and show its main properties.
The paper ends with acknowledgments.

\section{Preliminaries} \label{S-preliminary}

Throughout this paper $\Gamma$ will be a countable discrete group with the identity element $e$.
For any set $\cX$ and $d\in \Nb$, we write $\cX^{d\times 1}$ (resp. $\cX^{1\times d}$) for the elements of $\cX^d$ written as column (resp.~row) vectors.
For any set $\cX$, we denote by $\ell^2(\cX)$  the Hilbert space of all complex-valued  square-summable functions on $\cX$.

For a Hilbert space $H$, we denote by $B(H)$ the algebra of bounded linear operators on $H$, and
by $\|T\|$ the operator norm of $T$ for each $T\in B(H)$.

\subsection{Group rings} \label{SS-group rings}

For a unital ring $R$, the {\it group ring} $R\Gamma$ is the set of finitely supported functions
$f: \Gamma\rightarrow R$, written as $f=\sum_{s\in \Gamma}f_ss$, with addition and multiplication defined by
$$ \sum_{s\in \Gamma}f_ss+\sum_{s\in \Gamma}g_ss=\sum_{s\in \Gamma}(f_s+g_s)s \quad \mbox{and} \quad (\sum_{s\in \Gamma}f_ss)(\sum_{t\in \Gamma}g_tt)=\sum_{s, t\in \Gamma}f_sg_tst.$$

The group $\Gamma$ has two commuting unitary representations $l$ and $r$ on $\ell^2(\Gamma)$, called the \emph{left regular representation} and the \emph{right regular representation} respectively and defined by
$$ (l_sx)_t=x_{s^{-1}t} \quad \mbox{and} \quad (r_sx)_t=x_{ts}$$
for $s, t\in \Gamma$ and $x\in \ell^2(\Gamma)$. The {\it group von Neumann algebra} $\cN\Gamma$ is defined as the sub-$*$-algebra of $B(\ell^2(\Gamma))$ consisting of elements commuting with the image of $r$. See \cite[Section V.7]{Takesaki} for detail. Via the left regular representation $l$, we identify $\Cb\Gamma$ as a sub-$*$-algebra of $\cN\Gamma$.
Consider the anti-linear isometric involution $x\mapsto x^*$ on $\ell^2(\Gamma)$ defined by
$$(x^*)_s=\overline{x_{s^{-1}}}$$
for all $s\in \Gamma$ and $x\in \ell^2(\Gamma)$. Then, $\ell^2 (\Gamma)$ is also a right $\cN\Gamma$-module with
$$xf:=(f^*x^*)^*$$
for all $x\in \ell^2(\Gamma)$ and $f\in \cN\Gamma$. This allows an identification of $\cN \Gamma$ with the von Neumann algebra generated by the right regular representation, however, we do not need this fact.

For $d', d\in \Nb$, we think of elements of $M_{d'\times d}(\cN\Gamma)$ as bounded linear operators from $(\ell^2(\Gamma))^{d\times 1}$ to $(\ell^2(\Gamma))^{d'\times 1}$. There is a canonical trace $\tr_{\cN\Gamma}$ on $M_d(\cN\Gamma)$ defined
by
$$\tr_{\cN\Gamma}f=\sum_{j=1}^d\left<f_{j,j}e, e\right>$$
for $f=(f_{j,k})_{1\le j, k\le d}\in M_d(\cN\Gamma)$. One has $\tr_{\cN\Gamma}(fg)=\tr_{\cN\Gamma}(gf)$ for all $f, g\in M_d(\cN\Gamma)$.
Furthermore, $\tr_{\cN\Gamma}$ is {\it faithful} in the sense that $\tr_{\cN\Gamma}(f^*f)>0$ for every nonzero $f\in M_d(\cN\Gamma)$.

For a finitely generated projective left $\cN\Gamma$-module $\cM$, take $q\in M_d(\cN\Gamma)$ for some $d\in \Nb$ such that $q^2=q$
and $(\cN\Gamma)^{1 \times d}q$ is isomorphic to $\cM$ as a left $\cN\Gamma$-module. Then the {\it dimension}
$\dim_{\cN\Gamma}\cM$ of $\cM$ is defined as
$$\dim_{\cN\Gamma}\cM=\tr_{\cN\Gamma}q,$$
and is independent of the choice of $q$. For a general left $\cN\Gamma$-module $\cM$, its dimension $\dim_{\cN\Gamma}\cM$ is defined as
the supremum of $\dim_{\cN\Gamma}\cM'$ for $\cM'$ ranging over finitely generated projective submodules of $\cM$ \cite[Section 6.1]{Luck}. This generalized dimension was introduced by L\"uck. It has found numerous applications in the computation of $L^2$-invariants. We refer to \cite[Section 6.1]{Luck} for its basic properties.

\subsection{Fuglede-Kadison determinant} \label{SS-determinant}

For a Hilbert space $H$, we say $T\in B(H)$ is {\it positive}, written as
$T\ge 0$, if $\left<Tx, y\right>=\left<x, Ty\right>$ and $\left<Tx, x\right>\ge 0$ for all $x, y\in H$.

Let $f\in M_{d'\times d}(\cN\Gamma)$. Then $\ker f$ is a closed linear subspace of $(\ell^2(\Gamma))^{d\times 1}$ invariant under the right regular representation of $\Gamma$. Thus the orthogonal projection $q_f$ from $(\ell^2(\Gamma))^{d\times 1}$ onto $\ker f$ lies in $M_d(\cN\Gamma)$.

Let $f\in M_d(\cN\Gamma)$ be positive. Then there exists a unique Borel measure $\mu_f$, called the {\it spectral measure of $f$}, on the interval $[0, \|f\|]$
satisfying
\begin{align} \label{E-spectral measure}
 \int_0^{\|f\|}p(t)\, d\mu_f(t)=\tr_{\cN\Gamma}(p(f))
\end{align}
for every one-variable real-coefficients polynomial $p$. In particular, $\mu_f([0, \|f\|])=d$. From Theorems 5.2.2 and 5.2.8 of \cite{KR1} one has
$\mu_f(\{0\})=\tr_{\cN\Gamma}(q_f)$. Since $\tr_{\cN\Gamma}$ is faithful,
$\mu_f(\{0\})>0$ if and only if $\ker f\neq \{0\}$.

Let $f\in M_{d'\times d}(\cN\Gamma)$. Set $|f|=(f^*f)^{1/2}\in M_d(\cN\Gamma)$ \cite[page 248]{KR1}. Then $f$ and $|f|$ have the same operator norm, and $|f|$ is positive. The {\it Fuglede-Kadison determinant} $\det_{\cN\Gamma}f$
of $f$ \cite{FK} \cite[Section 3.2]{Luck} is defined as
\begin{align} \label{E-definition of determinant}
 \ddet_{\cN\Gamma}f=\exp\left(\int_{0}^{\|f\|}\log t\, d\mu_{|f|}(t)\right)\in \Rb_{\ge 0}.
\end{align}

One may describe $\det_{\cN\Gamma}f$ without referring to $|f|$. When $f\in M_d(\cN\Gamma)$ is positive, one has
$|f|=f$, and hence
\begin{align} \label{E-determinant}
 \ddet_{\cN\Gamma}f=\exp\left(\int_0^{\|f\|}\log t\, d\mu_f(t)\right).
\end{align}
For general $f\in M_{d'\times d}(\cN\Gamma)$, one has
\begin{align} \label{E-determinant square root}
 \ddet_{\cN\Gamma}f=\left(\ddet_{\cN\Gamma}(f^*f)\right)^{1/2}.
\end{align}

We remark that the Fuglede-Kadison determinant used in \cite{Luck} is a modified one, excluding the point $0$ in the integral of \eqref{E-definition of determinant}, i.e.  $\exp\left(\int_{0+}^{\|f\|}\log t\, d\mu_{|f|}(t)\right)$. Similar to \eqref{E-determinant square root}, this modified determinant is equal to $\left(\ddet_{\cN\Gamma}(f^*f+q_f)\right)^{1/2}$.

Among the properties of  $\ddet_{\cN\Gamma}$ established by Fuglede and Kadison in \cite[Section 5]{FK}, we mainly need the following

\begin{theorem}\label{T-FK}
Let $d\in \Nb$ and $f, g\in M_d(\cN\Gamma)$. The following hold:
\begin{enumerate}
\item  $\ddet_{\cN\Gamma}(fg)=\ddet_{\cN\Gamma} f\cdot \ddet_{\cN\Gamma}g$.

\item $\ddet_{\cN\Gamma}f=\ddet_{\cN\Gamma}(f^*)$.

\item If $f\ge 0$, then $\ddet_{\cN\Gamma}f=\inf_{\varepsilon>0}\ddet_{\cN\Gamma}(f+\varepsilon)$.

\item If $\ker f\neq \{0\}$, then $\mu_{f^*f}(\{0\})>0$, and hence $\ddet_{\cN\Gamma}f=0$.
\end{enumerate}
\end{theorem}

We also need the following result
of L\"{u}ck \cite[Theorem 3.14.(3)]{Luck}
\begin{align} \label{E-symmetry}
\ddet_{\cN\Gamma}(f^*f+q_f)=\ddet_{\cN\Gamma}(ff^*+q_{f^*})
\end{align}
for every $f\in M_{d'\times d}(\cN\Gamma)$.

\begin{example} \label{Ex-determinant}
It is instructing to consider the case $\Gamma=\Zb^d$ for $d\in \Nb$. Denote by $\lambda$ the Haar probability measure on the $d$-dimensional torus
$\Tb^d$. Fourier transform makes everything transparent: $\ell^2(\Zb^d) = L^2(\Tb^d,\lambda)$, $\cN(\Zb^d) = L^{\infty}(\Tb^d,\lambda)$, and
$$ \tr_{\cN(\Zb^d)} f=\int_{\Tb^d}f(z)\, d\lambda(z), \quad \ddet_{\cN(\Zb^d)}f=\exp \left(\int_{\Tb^d}\log |f(z)|\, d\lambda(z) \right)$$
for all $f\in \cN(\Zb^d)$.
\end{example}

\subsection{Amenable groups} \label{SS-amenable group}

Denote by $\cF(\Gamma)$ the set of all nonempty finite subsets of $\Gamma$.
For $K\in \cF(\Gamma)$ and $\delta > 0$ write $\cB (K,\delta )$ for the collection
of all $F\in \cF(\Gamma)$ such that $|\{t\in F: Kt\subseteq F \}| \ge(1-\delta) |F|$.
The group $\Gamma$ is called {\it amenable} if $\cB(K, \delta)$ is nonempty for every $(K, \delta)$ \cite[Section 4.9]{CC}.
Let $\Gamma$ be a countable discrete amenable group.

The pairs $(K,\delta )$ form a net $\Lambda$ where
$(K',\delta' ) \succ (K,\delta )$ means that $K' \supseteq K$ and $\delta' \leq \delta$.
For a $\Rb$-valued function $\varphi$ defined on $\cF(\Gamma)$, we say
$\varphi(F)$ converges to $c\in \Rb$ as the nonempty finite subset $F$ of $\Gamma$ becomes more and more left invariant, denoted by $\lim_F \varphi(F)=c$,  if for any
$\varepsilon>0$, there exist $K\in \cF(\Gamma)$ and $\delta>0$ such that
$$ |\varphi(F)-c|<\varepsilon$$
for all $F\in \cB(K, \delta)$. Similarly, we define $\varphi(F)$ converges to $-\infty$ or $+\infty$ as the nonempty finite subset $F$ of $\Gamma$ becomes more and more left invariant. In general, we define
\[ \limsup_{F} \varphi(F) := \lim_{(K,\delta )\in\Lambda} \sup_{F\in \cB (K,\delta )} \varphi (F). \]

Let $d\in \Nb$. For $F\in \cF(\Gamma)$, we denote by $\iota_F$ the embedding
$(\ell^2(F))^{d\times 1}\rightarrow (\ell^2(\Gamma))^{d\times 1}$ and by $p_F$ the orthogonal projection
$(\ell^2(\Gamma))^{d\times 1}\rightarrow (\ell^2(F))^{d\times 1}$. For $g\in M_d(\cN\Gamma)$, we set
\begin{align} \label{E-g_F}
 g_F:=p_F\circ g\circ \iota_F\in B((\ell^2(F))^{d\times 1}).
\end{align}

For many purposes, properties of $g \in M_d(\cN \Gamma)$ can be captured by properties of $g_F$ for $F \in \cF(\Gamma)$, as $F$ becomes more and more invariant. The following striking result was proved by Elek \cite{Elek03}.

\begin{lemma} \label{L-Elek}
Let $g\in M_d(\Cb\Gamma) \subseteq B((\ell^2(\Gamma))^{d \times 1})$. Then
$$\tr_{\cN\Gamma} (q_g) = \lim_{F} \frac{ \dim_{\Cb} \ker(g_F) }{|F|} = \lim_{F} \frac{ \dim_{\Cb} \left( \ker g \cap (\ell^2(F))^{d \times 1}  \right)}{|F|},$$
where $q_g$ denotes the orthogonal projection from $(\ell^2(\Gamma))^{d\times 1}$ onto $\ker g$.
In particular, if $\ker g\neq \{0\}$, then $\ker g \cap (\Cb\Gamma)^{d\times 1}\neq \{0\}$.
\end{lemma}

\subsection{Entropy} \label{SS-entropy}
We recall briefly the definition of entropy for actions of amenable groups. For more detail, see \cite{JMO, OW, Walters}.
Let $\Gamma$ a countable discrete amenable group.

Consider a continuous action of $\Gamma$ on a compact metrizable space $X$. For each finite open cover $\cU$ of $X$
and $F\in \cF(\Gamma)$, denote by $N(\cU)$ the minimal cardinality of subcovers of $\cU$ and by $\cU^F$ the cover of $X$ consisting of $\bigcap_{s\in F}s^{-1}U_s$ for all maps $F\rightarrow \cU$ sending $s$ to $U_s$. By the Ornstein-Weiss lemma \cite[Theorem 6.1]{LW}, the limit $\lim_F\frac{\log N(\cU^F)}{|F|}$ exists for every finite open cover $\cU$ of $X$. The {\it topological entropy} of the action $\Gamma\curvearrowright X$, denoted by $\htopol(X)$, is defined as
$$\htopol(X) = \sup_{\cU}\lim_F\frac{\log N(\cU^F)}{|F|}$$ for $\cU$ ranging over finite open covers of $X$.

For any measurable and measure-preserving action of $\Gamma$ on a probability measure space $(X, \cB, \mu)$, one can also define the measure-theoretic entropy, denoted by $\rh_{\mu}(X)$. We omit the definition, and just mention that the variational principle says that for any continuous action of
$\Gamma$ on a compact metrizable space $X$, one has $\htopol(X)=\sup_\mu \rh_\mu(X)$ for $\mu$ ranging over all the $\Gamma$-invariant Borel probability measures on $X$.

Let $\Gamma$ act on a compact metrizable group $X$ by (continuous) automorphisms. It is a theorem of Deninger that the topological entropy $\htopol(X)$ and the measure-theoretic entropy $\rh_\nu(X)$ for the Haar probability measure $\nu$ of $X$ coincide \cite[Theorem 2.2]{Deninger06}.  We call this common value {\it the entropy of this action}, and denote it by $\rh(X)$.

\subsection{Euler characteristic} \label{SS-euler}

Let $R$ be a unital ring and $k \in \Nb$. A left $R$-module $\cM$ is said to be of {\it type FL${}_k$} \cite[page 193]{Brown} if there exists a partial resolution $\cC_* \to \cM$ by finitely generated free left  $R$-modules of the form:
\begin{align} \label{E-resolution 0}
 \cC_k\overset{\partial_k}\rightarrow \cdots \overset{\partial_2}\rightarrow \cC_1\overset{\partial_1}\rightarrow \cC_0\rightarrow \cM\rightarrow 0,
\end{align}
i.e. this is an exact sequence of left $R$-modules and each $\cC_j$ for $0\le j\le k$ is a finitely generated free left $R$-module.

We say that a left $R$-module is of {\it type FL} \cite[page 199]{Brown} if, for some $k$, it admits a resolution  $\cC_*\rightarrow \cM$ by finitely generated free left $R$-modules of the form
\begin{align}  \label{E-resolution 2}
0 \to \cC_k \overset{\partial_k}\rightarrow \cdots \overset{\partial_2}\rightarrow \cC_1 \overset{\partial_1}\rightarrow \cC_0 \rightarrow \cM\rightarrow 0.
\end{align}

Now we assume that
\begin{align} \label{E-rank}
\mbox{ the free left } R\mbox{-modules } R^k \mbox{ and } R^l \mbox{ are non-isomorphic for distinct $k,l \in \Nb$}.
\end{align}
For any left $R$-module of type FL, its {\it Euler characteristic} $\chi(\cM)$ is defined as $\sum_{j=0}^k(-1)^jd_j$ for any resolution $\cC_*\rightarrow \cM$ by finitely generated free left $R$-modules as in \eqref{E-resolution 2}, where $d_j$ is the rank of $\cC_j$. The assumption \eqref{E-rank} and the generalized Schanuel's lemma \cite[Lemma VIII.4.4]{Brown} imply that $\chi(\cM)$ does not depend on the choice of the resolution.

Note that every field satisfies the condition \eqref{E-rank}. For any discrete group $\Gamma$, using the unital ring homomorphism $\Zb\Gamma\rightarrow \Qb$ sending $\sum_{s\in \Gamma}f_ss$ to $\sum_{s\in \Gamma}f_s$
one concludes that $\Zb\Gamma$ also satisfies the condition \eqref{E-rank}.

\subsection{$L^2$-torsion} \label{SS-torsion}

The definition of $L^2$-torsion is due to Carey-Mathai \cite{carmat, carmat2} and L\"uck-Rothenberg \cite{LR}. All technical ingredients and properties are developed in great detail in \cite[Chapter 3]{Luck}, see also \cite{BCFM,CFM}. We concentrate on what is called \emph{cellular} or \emph{combinatorial} $L^2$-torsion. Even in the combinatorial setup, the study of $L^2$-torsion involves complicated functional analysis.

For a finitely generated free left $\Zb\Gamma$-module $\cC$ with rank $d$, by choosing an ordered basis of $\cC$, we may identify $\ell^2(\Gamma)\otimes_{\Zb\Gamma}\cC$ with
the Hilbert space $(\ell^2(\Gamma))^{d\times 1}$. Though the inner product depends on the choice of the ordered basis of $\cC$, the resulting topology
on $\ell^2(\Gamma)\otimes_{\Zb\Gamma}\cC$ is independent of the choice of the ordered basis. For a (not necessarily exact) chain complex $\cC_*$ of finitely generated free left $\Zb\Gamma$-modules of the form
$$ \cC_k\overset{\partial_k}\rightarrow \cdots \overset{\partial_2}\rightarrow \cC_1\overset{\partial_1}\rightarrow \cC_0\overset{\partial_0}\rightarrow  0,$$
we say that the chain complex $\ell^2(\Gamma)\otimes_{\Zb\Gamma}\cC_*$, i.e.
\begin{align} \label{E-resolution 3}
\ell^2(\Gamma)\otimes_{\Zb\Gamma}\cC_k\overset{1\otimes \partial_k}\rightarrow \cdots \overset{1\otimes \partial_2}\rightarrow \ell^2(\Gamma)\otimes_{\Zb\Gamma}\cC_1\overset{1\otimes \partial_1}\rightarrow \ell^2(\Gamma)\otimes_{\Zb\Gamma}\cC_0\overset{1\otimes \partial_0}\rightarrow  0,
\end{align}
is {\it weakly acyclic} \cite[Definition 3.29]{Luck} if $\ker(1\otimes \partial_j)$ is equal to the closure of $\im(1\otimes \partial_{j+1})$ for all $0\le j< k$. In such case, we choose an ordered basis for each $\cC_j$, and identify $\cC_j$ with $(\Zb\Gamma)^{1\times d_j}$.
For each $1\le j\le k$, let $f_j\in  M_{d_j\times d_{j-1}}(\Zb\Gamma) $ such that $\partial_j(x)=xf_j$ for all $x\in (\Zb\Gamma)^{1\times d_j} $.
The {\it $L^2$-torsion} of $\cC_*$ is defined as \cite[Definition 3.29]{Luck}
\begin{align} \label{E-torsion}
\rho^{(2)}(\cC_*):=\frac12 \sum_{j=1}^k(-1)^{j+1}\log \ddet_{\cN\Gamma}(f_j^*f_j+q_{f_j}),
\end{align}
where $q_{f_j}$ denotes the orthogonal projection from $(\ell^2(\Gamma))^{d_{j-1}\times 1}$ onto $\ker f_j$. For a  chain complex $\cC_*$ of finitely generated free left $\Zb\Gamma$-modules of the form
$$ 0\rightarrow \cC_k\overset{\partial_k}\rightarrow \cdots \overset{\partial_2}\rightarrow \cC_1\overset{\partial_1}\rightarrow \cC_0\overset{\partial_0}\rightarrow  0,$$
the weak acyclicity
of the chain complex $\ell^2(\Gamma)\otimes_{\Zb\Gamma}\cC_*$ also requires that $1\otimes \partial_k$ is injective.

If a left $\Zb\Gamma$-module $\cM$ is of type FL and some resolution $\cC_*\rightarrow \cM$ of $\cM$ by finitely generated free left $\Zb\Gamma$-modules as in \eqref{E-resolution 2} is weakly acyclic, we define the {\it $L^2$-torsion} of $\cM$ to be $\rho^{(2)}(\cM):=\rho^{(2)}(\cC_*)$.
The results in \cite[Chapter 3]{Luck} imply
that when $\Gamma$ satisfies the determinant condition (Definition~\ref{D-determinant condition}), $\rho^{(2)}(\cM)$ does not depend on the choice of
the resolution $\cC_*$ and the ordered basis for each $\cC_j$.
In Section~\ref{S-torsion} we shall give a more algebraic proof of this fact.
In the case $\Gamma$ is amenable, this follows directly from Theorem~\ref{T-torsion}.

\section{Approximation of the determinant} \label{S-approximation}

\subsection{Review of known results} \label{SS-review}

The approximation properties of the Fuglede-Kadison determinant (and its ancestors) by determinants of associated matrices have attracted a lot of attention over the last century.

The most interesting applications arise if for $f \in \Zb \Gamma$ with $\Gamma$ residually finite, $\det_{\cN \Gamma}f$ can be approximated by determinants of the image of $f$ in the group ring of finite quotient groups of $\Gamma$ \cite[Question 13.52]{Luck}. Unfortunately, positive results are rare and only known in the simplest cases, where they are already non-trivial to prove. Denote by $\pi_n \colon \Cb[\Zb] \to \Cb[\Zb/n\Zb]$ the natural map induced by reduction modulo $n$. Schmidt showed that, in terms of the notation of Section~\ref{SS-determinant} and Example~\ref{Ex-determinant},
$$\lim_{n \to \infty} \ddet_{\cN(\Zb/n\Zb)} \left(\pi_n(f) + q_{\pi_n(f)}\right)= \ddet_{\cN \Zb} f$$
for elements in the group ring $\bar\Qb[\Zb]$ of $\Zb$ with coefficients in the field of algebraic numbers \cite[Lemma 21.8]{Schmidt}, see also
\cite[Lemma 13.53]{Luck}.
The proof relies on a theorem of Gelfond in number theory. Later, L\"uck gave an example of some element in $\Cb[\Zb]$ for which the corresponding approximation fails \cite[Example 13.69]{Luck}. The corresponding number theoretic results for $\Gamma=\Zb^d$ have been
identified and are open, see \cite[Section 9]{LSV10}. For the most recent progress in the case $\Gamma=\Zb^d$, see \cite{LSV11}. Given L\"uck's example, the approximation of the Fuglede-Kadison determinant by determinants of matrices was considered to be difficult in general. General approximation results with respect to finite quotients exist -- and are easy to prove -- if one assumes invertibility in the universal group $C^{\ast}$-algebra \cite[Theorem 7.1]{KL11}.
For $\Gamma$ being a finitely generated residually finite group and $f\in \Zb\Gamma$ being a generalized Laplace operator, the approximation result was proved by Lyons \cite{Lyons05, Lyons10}.

The approximation of the Fuglede-Kadison determinant by restrictions on F\o lner sets has an even longer history and dates -- in the case $\Gamma=\Zb$ -- back to Szeg\H{o} \cite{Szego}. Szeg\H{o}'s original result assumed invertibility of the positive function $f \in L^{\infty}(S^1)$, which makes computations much easier. However, Szeg\H{o}'s Theorem was extended \cite[Theorem 2.7.14]{Simon} to \emph{all} non-negative essentially bounded functions on $S^1$. We refer to Simon's book \cite[Chapter 3]{Simon} for the complete history. This opened up the possibility and expectation that the approximation by restriction on F\o lner sets could behave much better than one would expect at first. Theorem~\ref{T-approximate} confirms this expectation and generalizes this result from the case $\Gamma = \Zb$ to all countable discrete amenable groups.

\subsection{Approximation using F\o lner sequences}

In the rest of this section $\Gamma$ will be a countable discrete amenable group.
We prove four lemmas in order to prepare for the proof of Theorem~\ref{T-approximate}. First of all, we need the following
simple observation of Deninger \cite[Lemma 3]{Deninger09}.

\begin{lemma} \label{L-invertible}
Let $g\in M_d(\cN\Gamma)$ be positive.
Suppose that $\ker g\cap (\Cb\Gamma)^{d\times 1}=\{0\}$.
Then, for every $F\in \cF(\Gamma)$, $g_F \colon (\ell^2(F))^{d\times 1} \to (\ell^2(F))^{d\times 1}$ is  invertible.
\end{lemma}

The key observation for us is the following classical result of Gantmacher and Kre\u{\i}n \cite[page 96]{GK},
see also \cite{KK, JB}. Sometimes this result is attributed to Hadamard-Fischer-Koteljanskii. For convenience, we include a proof of this lemma following \cite{KK}.

\begin{lemma} \label{L-determinant is strong submultiplicative}
Let $\cX$ and $\cY$ be finite sets. Let $g\in B(\ell^2(\cX\cup \cY))$ be positive and invertible.
For any nonempty finite subset $\cE$ of $\cX\cup \cY$, define $g_\cE=p_\cE \circ g\circ \iota_\cE\in B(\ell^2(\cE))$, where $p_\cE$ denotes the orthogonal projection $\ell^2(\cX\cup \cY)\rightarrow \ell^2(\cE)$
and $\iota_\cE$ denotes the embedding $\ell^2(\cE)\rightarrow \ell^2(\cX\cup \cY)$.
Set $\det(g_{\emptyset})=1$.
Then
$$ \det (g_{\cX\cup \cY})\cdot \det (g_{\cX\cap \cY})\le \det (g_\cX)\cdot \det (g_\cY).$$
\end{lemma}
\begin{proof} First of all, we note that
for a positive definite matrix
$$t= \left(\begin{matrix}a & b \\ b^* & c \end{matrix} \right) \in M_{n+m}(\Cb)$$
we have that $a \in M_n (\Cb)$ and $c \in M_m(\Cb)$ are invertible. Indeed, if $t \geq \varepsilon$ for some constant $\varepsilon>0$, then $a,c \geq \varepsilon$. Moreover, we have:
\begin{equation} \label{E-detineq}
\det(t) = \det(a) \det(c - b^*a^{-1}b) \leq \det(a) \det(c).
\end{equation}
Indeed, the equality in \eqref{E-detineq} follows from
$$\left(\begin{matrix}1 & 0 \\ -b^*a^{-1} & 1 \end{matrix} \right)\left(\begin{matrix}a & b \\ b^* & c \end{matrix} \right) \left(\begin{matrix}1 & -a^{-1}b \\ 0 & 1 \end{matrix} \right) = \left(\begin{matrix}a & 0 \\ 0 & c - b^*a^{-1}b \end{matrix} \right)
$$ and the inequality in \eqref{E-detineq} follows since $c \geq c-b^*a^{-1}b \geq 0$.
Consider the positive definite matrix
$$g_{\cX \cup \cY} = \left(\begin{matrix}a & b & c\\ b^* & d & e \\ c^* & e^* & f \end{matrix} \right)$$
with
$$g_{\cX\cap \cY} = a, \quad g_{\cX} = \left(\begin{matrix}a & b \\ b^* & d \end{matrix} \right), \mbox{ and} \quad g_{\cY} = \left(\begin{matrix}a & c\\  c^* & f \end{matrix} \right).$$
Now, the equality {\small
$$\left(\begin{matrix}1 & 0 & 0\\ -b^*a^{-1} & 1 & 0 \\ -c^*a^{-1} & 0 & 1 \end{matrix} \right) \left(\begin{matrix}a & b & c\\ b^* & d & e \\ c^* & e^* & f \end{matrix} \right) \left(\begin{matrix}1 & -a^{-1 }b & -a^{-1}c\\ 0 & 1 & 0 \\ 0 & 0 & 1 \end{matrix} \right)= \left(\begin{matrix}a & 0 & 0\\ 0 & d-b^*a^{-1}b & e -b^*a^{-1} c\\ 0 & e^* - c^*a^{-1}b & f -c^*a^{-1}c \end{matrix} \right)$$}
shows that
$$\det(g_{\cX \cup \cY}) = \det(g_{\cX \cap \cY}) \cdot \det\left(\begin{matrix}d-b^*a^{-1}b & e -b^*a^{-1} c\\  e^* - c^*a^{-1}b & f -c^*a^{-1}c \end{matrix} \right).$$
But, using \eqref{E-detineq} three times, we see
\begin{eqnarray*}
\det\left(\begin{matrix}d-b^*a^{-1}b & e -b^*a^{-1} c\\  e^* - c^*a^{-1}b & f -c^*a^{-1}c \end{matrix} \right) &\leq& \det(d-b^*a^{-1}b) \cdot \det (f -c^*a^{-1}c) \\
&=&  \frac{\det(g_{\cX})}{\det(g_{\cX \cap \cY})} \cdot \frac{\det(g_{\cY})}{\det(g_{\cX \cap \cY})}.
\end{eqnarray*}
This finishes the proof.
\end{proof}

We also need the following result from \cite{JMO}, see Definitions 2.2.10 and 3.1.5, Remark 3.1.7, and Proposition 3.1.9 of \cite{JMO}.
For convenience, we include a proof following \cite{Danilenko}.

\begin{lemma} \label{L-strong subadditive to limit}
Let $\varphi$ be a $\Rb$-valued function defined on $\cF(\Gamma)\cup \{\emptyset\}$ such that
\begin{enumerate}
\item $\varphi(\emptyset)=0$;

\item $\varphi(Fs)=\varphi(F)$ for all $F\in \cF(\Gamma)$ and $s\in \Gamma$;

\item $\varphi(F_1\cup F_2)+\varphi(F_1\cap F_2)\le \varphi(F_1)+\varphi(F_2)$ for all $F_1, F_2\in \cF(\Gamma)$.
\end{enumerate}
Then
$$\lim_F\frac{\varphi(F)}{|F|}=\inf_{F\in \cF(\Gamma)}\frac{\varphi(F)}{|F|}.$$
\end{lemma}
\begin{proof}
We show first that if $F, F_1, \dots, F_m\in \cF(\Gamma)$ and $\lambda_1, \dots, \lambda_m\in (0, 1]$ with $1_F=\sum_{j=1}^m\lambda_j1_{F_j}$, then
\begin{align} \label{E-limit inf}
 \varphi(F)\le \sum_{j=1}^m\lambda_j\varphi(F_j).
 \end{align}
Consider the partition of $F$ generated by $F_1, \dots, F_m$. Take
$$\emptyset= K_0\varsubsetneq K_1\varsubsetneq\cdots \varsubsetneq K_n=F$$
such that $K_i\setminus K_{i-1}$ is an atom of this partition for each $1\le i\le n$.
For each $1\le i\le n$ take $s_i\in K_i\setminus K_{i-1}$.
For any $1\le i\le n$ and $1\le j\le m$, either $F_j\cap (K_i\setminus K_{i-1})=\emptyset$
or $K_{i-1}\cup(K_i\cap F_j)=K_i$, and hence from the condition (3) one always has
\begin{align} \label{E-limit inf 1}
1_{F_j}(s_i)(\varphi(K_i)-\varphi(K_{i-1}))\le 1_{F_j}(s_i)(\varphi(K_i\cap F_j)-\varphi(K_{i-1}\cap F_j)).
\end{align}
Also note that for any $1\le i\le n$ and $1\le j\le m$, $s_i\not \in F_j$ if and only if $K_i\cap F_j=K_{i-1}\cap F_j$.
Thus, for any $1\le j\le m$, listing all the $1\le i\le n$ with $s_i\in F_j$ in increasing order as $i_1<i_2<\dots <i_k$ for some $1\le k\le n$ and setting $i_0=0$, we have
\begin{align} \label{E-limit inf 3}
\lefteqn{\sum_{i=1}^n1_{F_j}(s_i)(\varphi(K_i\cap F_j)-\varphi(K_{i-1}\cap F_j))} \hspace*{10mm}\\
\nonumber \hspace*{10mm} &=\sum_{l=1}^k(\varphi(K_{i_l}\cap F_j)-\varphi(K_{i_l-1}\cap F_j))\\
\nonumber \hspace*{10mm}&= \sum_{l=1}^k(\varphi(K_{i_l}\cap F_j)-\varphi(K_{i_{l-1}}\cap F_j))\\
\nonumber \hspace*{10mm}&= \varphi(K_{i_k}\cap F_j)-\varphi(K_0\cap F_j)\\
\nonumber \hspace*{10mm}&=\varphi(K_n\cap F_j)-\varphi(K_0\cap F_j)\\
\nonumber \hspace*{10mm}&=\varphi(F_j)-\varphi(\emptyset)=\varphi(F_j).
\end{align}
Therefore
\begin{eqnarray*}
\varphi(F)&=&\sum_{i=1}^n(\varphi(K_i)-\varphi(K_{i-1}))\\
&=&\sum_{i=1}^n(\sum_{j=1}^m\lambda_j1_{F_j}(s_i))(\varphi(K_i)-\varphi(K_{i-1}))\\
&=&\sum_{j=1}^m\lambda_j\sum_{i=1}^n1_{F_j}(s_i)(\varphi(K_i)-\varphi(K_{i-1}))\\
&\overset{\eqref{E-limit inf 1}}\le& \sum_{j=1}^m\lambda_j\sum_{i=1}^n1_{F_j}(s_i)(\varphi(K_i\cap F_j)-\varphi(K_{i-1}\cap F_j))\\
&\overset{\eqref{E-limit inf 3}}=&\sum_{j=1}^m\lambda_j\varphi(F_j)
\end{eqnarray*}
as desired.

To prove the lemma, it suffices to show
\begin{align} \label{E-limit inf 2}
\limsup_F\frac{\varphi(F)}{|F|}\le \frac{\varphi(K)}{|K|}
\end{align}
for every $K\in \cF(\Gamma)$. Let $K\in \cF(\Gamma)$. By the condition (2) we  may assume that $e\in K$.
Denote by $C$ the maximum of $\varphi(K')$ for $K'$ ranging through the nonempty subsets of $K$.
For any $F\in \cF(\Gamma)$ we have
$$1_F=\frac{1}{|K|}\sum_{Ks\cap F\neq \emptyset}1_{Ks\cap F},$$
and hence by \eqref{E-limit inf} we get
\begin{align*}
\varphi(F)&\le \frac{1}{|K|}\sum_{Ks\cap F\neq \emptyset}\varphi(Ks\cap F)\\
&=\frac{1}{|K|}\sum_{Ks\subseteq F}\varphi(Ks\cap F)+\frac{1}{|K|}\sum_{Ks\varsubsetneq F, Ks\cap F\neq \emptyset}\varphi(Ks\cap F)\\
&\le \frac{\varphi(K)}{|K|}\cdot |\{ s \in \Gamma \mid Ks \subseteq F\}| +\frac{C}{|K|}\cdot |\{s\in \Gamma: Ks\varsubsetneq F, Ks\cap F\neq \emptyset\}|.
\end{align*}
When $F\in \cF(\Gamma)$ becomes more and more left invariant, we have $\frac{1}{|F|}|\{s\in \Gamma: Ks \subseteq F\}|\to 1$ and $\frac{1}{|F|}|\{s\in \Gamma: Ks\varsubsetneq F, Ks\cap F\neq \emptyset\}|\to 0.$
It follows that \eqref{E-limit inf 2} holds.
\end{proof}

The following result of Deninger \cite[Theorem 3.2]{Deninger06} is also needed for the proof. Though he proved it only for the case $d=1$, his argument works for general $d\in \Nb$. For $\Gamma = \Zb^n$ and $d=1$, the result goes back to work by Linnik \cite{linnik}, see also the work of Helson
and Lowdenslager \cite{hellow}. For a proof see Remark~\ref{R-lemma invertible}.

\begin{lemma} \label{L-approximate invertible}
Let $g\in M_d(\cN\Gamma)$ be positive and invertible.
Then
$$ \ddet_{\cN\Gamma}g=\lim_F (\det(g_F))^{\frac{1}{|F|}}.$$
\end{lemma}

We are ready to prove Theorem~\ref{T-approximate}.

\begin{proof}[Proof of Theorem~\ref{T-approximate}]
Let $h\in M_d(\cN\Gamma)$ be positive such that $\ker h\cap (\Cb\Gamma)^{d\times 1}=\{0\}$.
Define $\varphi: \cF(\Gamma)\cup \{\emptyset\}\rightarrow \Rb$ by
$\varphi(F)=\log \det (h_F)$, where we set $\det (h_{\emptyset})=1$. Then $\varphi(\emptyset)=0$ and $\varphi(Fs)=\varphi(F)$ for
all $F\in \cF(\Gamma)$ and $s\in \Gamma$. Denote $\{1, \dots, d\}$ by $\Delta_d$.
By Lemmas~\ref{L-invertible} $h_F\in B((\ell^2(F))^{d\times 1})=B(\ell^2(F\times \Delta_d))$ is positive and invertible for every $F\in \cF(\Gamma)$.
For any $F_1, F_2\in \cF(\Gamma)$, taking $\tilde{h}=h_{F_1\cup F_2}\in B(\ell^2((F_1\cup F_2)\times \Delta_d)$,
in terms of the notation in Lemma~\ref{L-determinant is strong submultiplicative} we have $\tilde{h}_{F_j\times \Delta_d}=h_{F_j}$ for $j=1, 2$ and
$\tilde{h}_{(F_1\cap F_2)\times \Delta_d}=h_{F_1\cap F_2}$.
Thus by Lemma~\ref{L-determinant is strong submultiplicative} we have
$\varphi(F_1\cup F_2)+\varphi(F_1\cap F_2)\le \varphi(F_1)+\varphi(F_2)$ for all $F_1, F_2\in \cF(\Gamma)$. Then by Lemma~\ref{L-strong subadditive to limit} we have
\begin{align*}
 \lim_F \frac{\log \det(h_F)}{|F|}=\inf_{F\in \cF(\Gamma)}\frac{\log \det(h_F)}{|F|},
\end{align*}
equivalently,
\begin{align} \label{E-inf=limit}
 \lim_F (\det(h_F))^{\frac{1}{|F|}}=\inf_{F\in \cF(\Gamma)}(\det(h_F))^{\frac{1}{|F|}}.
\end{align}

For any $\varepsilon>0$, since $g+\varepsilon$ is invertible in $\cN\Gamma$, by Lemma~\ref{L-approximate invertible} we have
\begin{align} \label{E-limit invertible}
\ddet_{\cN\Gamma} (g+\varepsilon)= \lim_F (\det((g+\varepsilon)_F))^{\frac{1}{|F|}}.
\end{align}
From Theorem~\ref{T-FK} we have
\begin{eqnarray*}
\ddet_{\cN\Gamma} g&=&\inf_{\varepsilon>0}\ddet_{\cN\Gamma} (g+\varepsilon) \\
&\overset{\eqref{E-limit invertible}}=& \inf_{\varepsilon>0}\lim_F (\det((g+\varepsilon)_F))^{\frac{1}{|F|}} \\
&\overset{\eqref{E-inf=limit}}=& \inf_{\varepsilon>0}\inf_{F\in \cF(\Gamma)} (\det((g+\varepsilon)_F))^{\frac{1}{|F|}} \\
&=& \inf_{F\in \cF(\Gamma)} \inf_{\varepsilon>0} (\det((g+\varepsilon)_F))^{\frac{1}{|F|}} \\
&=& \inf_{F\in \cF(\Gamma)} (\det(g_F))^{\frac{1}{|F|}},
\end{eqnarray*}
establishing the first equality in \eqref{E-approximate}.

If $\ker g\cap (\Cb\Gamma)^{d\times 1}=\{0\}$, then taking $h=g$ in \eqref{E-inf=limit} we get the second equality in \eqref{E-approximate}.
Thus assume $gx=0$ for some nonzero $x\in (\Cb\Gamma)^{d\times 1}$. Since $\ker g\neq \{0\}$, by Theorem~\ref{T-FK} we have $\ddet_{\cN\Gamma}(g)=0$. Denote by $K$ the support
of $x$ as a $\Cb^{d\times 1}$-valued function on $\Gamma$. Then $K\in \cF(\Gamma)$. Let $F\in \cB(K, 1/2)$. Then we can find some $s\in F$ such that $Ks\subseteq F$.
Now $xs$ is a nonzero element of $(\ell^2(F))^{d\times 1}$, and $g(xs)=(gx)s=0$. It follows that $g_F(xs)=p_F(g(xs))=0$. Thus $g_F$ is not injective, and hence
$\det(g_F)=0$. In particular, $\lim_F(\det(g_F))^{\frac{1}{|F|}}=0=\ddet_{\cN\Gamma}g$.
\end{proof}

\begin{remark} \label{R-twisted}
Denote by $S^1$ the unit circle in the complex plane. A {\it normalized unitary 2-cocycle} of $\Gamma$ is a map $\alpha: \Gamma \times \Gamma \rightarrow S^1$, such that
$$\alpha(s_1, s_2)\alpha(s_1s_2, s_3) = \alpha(s_1, s_2s_3)\alpha(s_2, s_3), \quad \alpha(s_1,e) = \alpha(e,s_1)=1,\quad \forall  s_1, s_2, s_3 \in \Gamma.$$
Let $\alpha$ be a normalized unitary 2-cocycle of $\Gamma$. Then one has the twisted left and right unitary representations $l_\alpha$ and $r_\alpha$ of $\Gamma$ on $\ell^2(\Gamma)$ defined by
$$ (l_{\alpha, s}x)_t=\alpha(s, s^{-1}t)x_{s^{-1}t} \mbox{ and } (r_{\alpha,s}x)_t=\alpha(ts, s^{-1})x_{ts}$$
for $s, t\in \Gamma$ and $x\in \ell^2(\Gamma)$. (These twisted representations are not representations, but satisfy $l_{\alpha, s_1}l_{\alpha, s_2}=\alpha(s_1,s_2)l_{\alpha, s_1s_2}$ and $r_{\alpha, s_1}r_{\alpha, s_2}=\alpha(s_2^{-1}, s_1^{-1})r_{\alpha, s_1s_2}$ for all $s_1, s_2\in \Gamma$.)
The {\it twisted group von Neumann algebra} $\cN_\alpha\Gamma$ is defined as the sub-$*$-algebra of $B(\ell^2(\Gamma))$ consisting of elements commuting with $r_\alpha$.  The {\it twisted group algebra} $\Cb_{\alpha} \Gamma$ is the ring with underlying vector space $\Cb \Gamma$ and associative multiplication determined by $s \cdot t := \alpha(s,t) st$ for $s, t\in \Gamma$.
Via $l_\alpha$, we identify $\Cb_\alpha\Gamma$ as a sub-$*$-algebra of $\cN_\alpha\Gamma$.
Taking $s_1=s_3=s$ and $s_2=s^{-1}$ one obtains $\alpha(s, s^{-1})=\alpha(s^{-1}, s)$ for all $s\in \Gamma$.
Then one still has the trace $\tr_{\cN_\alpha\Gamma}$ and the Fuglede-Kadison determinant $\ddet_{\cN_\alpha\Gamma}$ defined, in exactly the same way as $\tr_{\cN\Gamma}$ and $\ddet_{\cN\Gamma}$ are.

When $\alpha$ is the constant function $1$, $\cN_\alpha\Gamma$ is simply $\cN\Gamma$.
An interesting example of a twisted group von Neumann algebra arises already for $\Zb^2$ and the cocycle $\alpha_\theta: \Zb^2 \times \Zb^2 \rightarrow S^1$ given by $\alpha_\theta((n_1, n_2), (m_1, m_2)):= \exp(2\pi i \theta (n_1m_2-n_2m_1))$ for $\theta \in \Rb$. When $\theta$ is irrational, $\cN_{\alpha_\theta} \Zb^2$ is the hyperfinite ${\rm II}_1$-factor \cite[Corollary 1.16]{Boca}.

Lemmas~\ref{L-invertible} and \ref{L-approximate invertible} and Theorems~\ref{T-FK} and \ref{T-approximate}, and their proofs all work with $\cN\Gamma$ replaced by $\cN_\alpha\Gamma$.
\end{remark}

\subsection{Estimates on the spectrum near zero}

\begin{notation} \label{N-product of eigenvalues}
For any positive $g\in M_d(\cN\Gamma)$, $F\in \cF(\Gamma)$, and $\kappa>0$, we denote by $\sD_{g, F, \kappa}$ the product of the eigenvalues of $g_F$
in the interval $(0, \kappa]$ counted with multiplicity. If $g_F$ has no eigenvalue in $(0, \kappa]$, we set $\sD_{g, F, \kappa}=1$.
\end{notation}

Using Theorem~\ref{T-approximate} we shall prove the following result, describing the asymptotic behavior of $\sD_{g, F, \kappa}$ under the condition that $\ddet_{\cN\Gamma}g>0$.

\begin{proposition} \label{P-measure weak convergence}
Let $g\in M_d(\cN\Gamma)$  be positive with $\det_{\cN\Gamma}g>0$.
Let $\lambda>1$. Then there exists $0<\kappa<\min(1, \|g\|)$ such that
$$ \limsup_F(\sD_{g, F, \kappa})^{-\frac{1}{|F|}}\le \lambda.$$
\end{proposition}

To prove Proposition~\ref{P-measure weak convergence} we need some preparation.

Let $g\in M_d(\cN\Gamma)$ be positive and $F\in \cF(\Gamma)$.
Denote by $\tr_F$ the $\Cb$-valued trace on $B((\ell^2(F))^{d\times 1})$ normalized by $\tr_F(1)=d$.
 Then there exists a unique Borel measure $\mu_{g, F}$, called the {\it spectral measure of $g_F$}, on the interval $[0, \|g_F\|]\subseteq [0, \|g\|]$
satisfying
\begin{align} \label{E-spectral measure on F}
 \int_0^{\|g\|}p(t)\, d\mu_{g, F}(t)=\tr_F(p(g_F))
\end{align}
for every one-variable real-coefficients polynomial $p$. In particular, $\mu_{g, F}([0, \|g\|])=d$.
More explicitly, for any $t\in [0, \|g\|]$, one has
\begin{align} \label{E-spectral measure and multiplicity}
 \mu_{g, F}(\{t\})=\frac{\mbox{ multiplicity of } t \mbox{ as an eigenvalue of } g_F}{|F|}.
\end{align}

\begin{lemma} \label{L-measure weak convergence}
Let $g\in M_d(\cN\Gamma)$ be positive and non-zero.
Let $0<\kappa<\min(\|g\|, 1)$.
Then
$$ \limsup_F\int_{\kappa+}^{\|g\|}\log t \, d\mu_{g, F}(t)\le \int_{\kappa+}^{\|g\|}\log t \, d\mu_g(t).$$
\end{lemma}
\begin{proof}

By a result of L\"{u}ck, Dodziuk-Mathai, and Schick \cite{Luck94} \cite[Lemma 2.3]{DM} \cite[Lemma 4.6]{Schick} \cite[Lemma 13.42]{Luck},  one has
$$ \tr_{\cN\Gamma}(p(g))=\lim_F\tr_F(p(g_F))$$
for every one-variable real-coefficients polynomial $p$.
In view of \eqref{E-spectral measure} and \eqref{E-spectral measure on F}, this means
\begin{align} \label{E-measure convergence}
 \int_0^{\|g\|}p(t) \, d\mu_g(t)=\lim_F\int_0^{\|g\|}p(t) \, d\mu_{g, F}(t).
\end{align}
By the Stone-Weierstra\ss\ approximation theorem, the space of one-variable real-coefficients polynomials is dense in
the space of real-valued continuous functions on the interval $[0, \|g\|]$, under the supremum norm. Thus \eqref{E-measure convergence} holds for
every real-valued continuous function $p$ on $[0, \|g\|]$. That is, the net $\{\mu_{g, F}\}_F$ converges to $\mu_g$ weakly, as $F\in \cF(\Gamma)$ becomes more
and more left invariant. Then one has
$$\int_0^{\|g\|}h(t) \, d\mu_g(t)\ge \limsup_F\int_0^{\|g\|}h(t) \, d\mu_{g, F}(t)$$
for every real-valued upper semicontinuous function $h$ on $[0, \|g\|]$ \cite[page 24, Exercise 2.6]{Billingsley}.

Set $h(t)=0$ for $t\in [0, \kappa]$ and $h(t)=\log t$ for $t\in (\kappa, \|g\|]$. Then $h$ is a real-valued upper semicontinuous  function on $[0, \|g\|]$.
Therefore
\begin{align*}
\int_{\kappa+}^{\|g\|}\log t \, d\mu_g(t)&=\int_0^{\|g\|}h(t) \, d\mu_g(t)\\
&\ge \limsup_F\int_0^{\|g\|}h(t) \, d\mu_{g, F}(t)\\
&=\limsup_F\int_{\kappa+}^{\|g\|}\log t \, d\mu_{g, F}(t).
\end{align*}
\end{proof}

\begin{remark} \label{R-lemma invertible}
Note that once the weak convergence $\lim_F \mu_{g, F}= \mu_g$ has been established, Lemma~\ref{L-approximate invertible} is an immediate consequence. Indeed, if $g \geq \varepsilon$ for some constant $\varepsilon > 0$, then the support of $\mu_{g,F}$ is contained in $[\varepsilon,\| g\|]$ for all $F$. This implies the lemma, since the function $t\mapsto \log t$ is continuous and bounded on this interval.
\end{remark}

We are ready to prove Proposition~\ref{P-measure weak convergence}.

\begin{proof}[Proof of Proposition~\ref{P-measure weak convergence}]
Because $\int_0^{\|g\|}\log t\, d\mu_g(t)=\log\det_{\cN\Gamma}g>-\infty$,
we  can find some $0<\kappa<\min(1, \|g\|)$ such that
$$\int_0^\kappa \log t \, d\mu_g(t)\ge -\log \lambda.$$

Since $\det_{\cN\Gamma}g>0$, by Theorem~\ref{T-FK} we have $\ker g=\{0\}$. Let $F\in \cF(\Gamma)$. By Lemma~\ref{L-invertible} we get that $g_F$ is injective.
From \eqref{E-spectral measure and multiplicity} we have
\begin{align*}
-\frac{\log \sD_{g, F, \kappa}}{|F|}
=-\frac{\log \det(g_F)}{|F|}+\int_{\kappa+}^{\|g\|}\log t \, d\mu_{g, F}(t).
\end{align*}

From Theorem~\ref{T-approximate} and  Lemma~\ref{L-measure weak convergence} we get
\begin{align*}
 \limsup_F \left(-\frac{\log \sD_{g, F, \kappa}}{|F|} \right)&=\limsup_F \left(-\frac{\log \det(g_F)}{|F|}+\int_{\kappa+}^{\|g\|}\log t \, d\mu_{g, F}(t) \right)\\
 &= \lim_F \left(-\frac{\log \det(g_F)}{|F|} \right)+\limsup_{n\to \infty}\int_{\kappa+}^{\|g\|}\log t \, d\mu_{g, F}(t)\\
 &\le -\log \ddet_{\cN\Gamma}g+\int_{\kappa+}^{\|g\|}\log t \, d\mu_g(t)\\
 &=-\int_0^{\|g\|}\log t \, d\mu_g(t)+\int_{\kappa+}^{\|g\|}\log t \, d\mu_g(t)\\
 &=-\int_0^{\kappa}\log t \, d\mu_g(t)\le \log\lambda.
\end{align*}
\end{proof}

\begin{remark} \label{R-twisted 2}
Proposition~\ref{P-measure weak convergence} and its proof also work in the twisted case.
\end{remark}

In order to apply Proposition~\ref{P-measure weak convergence}, we need to know which $g\in M_d(\cN\Gamma)$ has strictly positive determinant.
Schick showed that this is the case for $g\in M_d(\Zb\Gamma)$  with $\ker g=\{0\}$ \cite[Theorem 1.21]{Schick} \cite[Theorem 13.3]{Luck}:

\begin{lemma} \label{L-integral to positive determinant}
For any $g\in M_d(\Zb\Gamma)$  with $\ker g=\{0\}$, one has $\ddet_{\cN\Gamma}g\ge 1$.
\end{lemma}

Though Lemma~\ref{L-integral to positive determinant} is sufficient for our proof of Theorems~\ref{T-main} and \ref{T-torsion}, we note by passing a slightly more general result from \cite{DLMSY}. Denote by $\bar{\Qb}$ the algebraic closure of $\Qb$ in $\Cb$, i.e. the field of algebraic numbers.

\begin{lemma} \label{L-algebraic to positive determinant}
For any $g\in M_d(\bar{\Qb}\Gamma)$  with $\ker g=\{0\}$, one has $\ddet_{\cN\Gamma}g>0$.
\end{lemma}
\begin{proof}
Since $\ddet_{\cN\Gamma}g=(\ddet_{\cN\Gamma}(g^*g))^{1/2}$, $\ker (g^*g)=\ker g=\{0\}$, and $g^*g\in M_d(\bar{\Qb}\Gamma)$, we may assume that $g\ge 0$.

Multiplying $g$ by a suitable positive integer, we may assume further that $g\in M_d(\cO\Gamma)$, where $\cO \subseteq \bar{\Qb}$ denotes the ring of algebraic integers in $\bar{\Qb}$. The coefficients of $g$ are contained in a finite Galois field extension of $\Qb$ and we denote by $G$ its Galois group. We write $g = \sum_{s \in K} g_s s$ with $g_s \in M_d(\cO)$, where $K \subset \Gamma$ is the support of $g$. Let $C>0$ be an upper bound for $\|\sigma(g_s)\|$ for all $s \in K$ and $\sigma \in G$.
Note that $\|\sigma(g_F)\| \leq C |K|$ for all $F \in \cF(\Gamma)$, and hence $|\det(\sigma(g_F))| \leq (C  |K|)^{d|F|}$.

Note that $\prod_{\sigma \in G} \sigma(\det(g_F))\in \Zb$ as it is a Galois invariant algebraic integer. Since $g_F$ is invertible by Lemma~\ref{L-invertible}, one has $\det(g_F) \neq 0$ and hence $\det(\sigma(g_F)) = \sigma(\det(g_F)) \neq 0$ for all $\sigma \in G$.
Then $\prod_{\sigma \in G} \sigma(\det(g_F)) \neq 0$, and we conclude $$\prod_{\sigma \in G} |\sigma(\det(g_F)) |\ge 1.$$
Now, $|\sigma(\det(g_F))| \leq (C |K|)^{d |F|}$
for all $\sigma \in G$ and we get
$$\det(g_F) \geq \prod_{\sigma\in G\setminus \{e_G\}}|\sigma(\det(g_F))|^{-1}\geq \prod_{\sigma\in G\setminus \{e_G\}} (C |K|)^{-d |F|}=(C |K|)^{-d |F| (|G|-1)},$$
where $e_G$ denotes the identity element of $G$.

From Theorem~\ref{T-approximate} we conclude that
$$\ddet_{\cN\Gamma}g = \inf_{F \in \cF(\Gamma)} \det(g_F)^{\frac1{|F|}}\geq (C|K|)^{-d (|G|-1)} >0.$$
This finishes the proof.
\end{proof}

We are unable to settle the following question:

\begin{question} \label{Q-positive determinant}
Do we have $\ddet_{\rm \cN\Gamma}g>0$
for all  $g\in M_d(\Cb\Gamma)$ with $\ker g=\{0\}$?
\end{question}

\section{Entropy and determinant} \label{S-entaction}

In this section we prove Theorem~\ref{T-main}. Throughout this section $\Gamma$ will be a countable discrete amenable group.

\subsection{A formula for entropy of finitely generated algebraic actions} \label{SS-approximate solution}

In this subsection we prove Theorem~\ref{T-approximate solution formula for entropy}, giving a formula for the entropy of finitely generated algebraic actions
in terms of approximate solutions to the equations defining the algebraic action.

Let $\theta$ be a continuous pseudometric
on a compact metrizable space $X$. For $\varepsilon>0$, we say that $\cW\subseteq X$ is {\it $(\theta, \varepsilon)$-separated} if $\theta(x, y)>\varepsilon$ for
all distinct $x, y\in \cW$. Denote by $N_\varepsilon(X, \theta)$ the maximal cardinality of $(\theta, \varepsilon)$-separated subsets of $X$.

Let $\Gamma$ act continuously on $X$. We say that $\theta$ is {\it dynamically generating} if for any distinct $x, y\in X$ one has $\theta(sx, sy)>0$ for some
$s\in \Gamma$. For each $F\in \cF(\Gamma)$, we define continuous pseudometrics $\theta_{F, 2}$ and $\theta_{F, \infty}$ on
$X$ by
\begin{align}
\label{E-2 distance} \theta_{F, 2}(x, y)&= \left(\frac{1}{|F|}\sum_{s\in F}(\theta(sx, sy))^2\right)^{1/2},\\
\label{E-infinity distance} \theta_{F, \infty}(x, y)&=\max_{s\in F}\theta(sx, sy).
\end{align}

The following result says that the topological entropy of the action can be calculated using $N_\varepsilon(X, \theta_{F, 2})$ or $N_\varepsilon(X, \theta_{F, \infty})$, for any dynamically generating continuous pseudometric $\theta$ on $X$. The formula in terms of $N_\varepsilon(X, \theta_{F, \infty})$
was proved by Deninger \cite[Proposition 2.3]{Deninger06}, and the formula in terms of  $N_\varepsilon(X, \theta_{F, 2})$ is in \cite[Theorem 4.2]{Li}.

\begin{lemma} \label{L-l2 definition of entropy}
Let $\Gamma$ act on a compact metrizable space $X$ continuously, and let $\theta$ be a dynamically generating continuous pseudometric
on $X$.  Then
$$ \rh_{\rm top}(X)=\sup_{\varepsilon>0} \limsup_F\frac{\log N_\varepsilon(X, \theta_{F, 2})}{|F|}=\sup_{\varepsilon>0} \limsup_F\frac{\log N_\varepsilon(X, \theta_{F, \infty})}{|F|}.$$
\end{lemma}

For any countable left $\Zb\Gamma$-module $\cM$, we denote by $\widehat{\cM}$ the Pontryagin dual of the discrete abelian group $\cM$.
The left $\Zb\Gamma$-module structure on $\cM$ corresponds to an action of $\Gamma$ on the discrete abelian group $\cM$ by automorphisms, and induces a natural action of $\Gamma$ on the compact metrizable abelian group $\widehat{\cM}$ by (continuous) automorphisms.

Let $d\in \Nb$. We may identify $((\Rb/\Zb)^{d\times 1})^\Gamma$ with  $\widehat{(\Zb\Gamma)^{1\times d}}$ naturally through the pairing
$ (\Zb\Gamma)^{1\times d}\times ((\Rb/\Zb)^{d\times 1})^\Gamma\rightarrow \Rb/\Zb$ given by
$$ \left<f, x\right>=(fx)_e,$$
where $fx\in (\Rb/\Zb)^\Gamma$ is defined similar to the product in $\Zb\Gamma$:
$$ (fx)_t=\sum_{s\in \Gamma, 1\le j\le d}f_{s, j}x_{s^{-1}t, j}.$$
It is easy to check that the induced natural $\Gamma$-action on $((\Rb/\Zb)^{d\times 1})^\Gamma$, denoted by $\sigma$, is given by
$$ (\sigma_s(x))_t=x_{ts}$$
for $x\in ((\Rb/\Zb)^{d\times 1})^\Gamma$ and $s, t\in \Gamma$.

Consider the following metric $\vartheta$ on $\Rb/\Zb$:
$$ \vartheta(x+\Zb, y+\Zb)=\min_{z\in \Zb}|x-y-z|.$$
Using $\vartheta$ we define two metrics $\vartheta_2$ and $\vartheta_\infty$ on $(\Rb/\Zb)^{d\times 1}$ by
\begin{align*}
\vartheta_2(x, y)&=\left(\frac{1}{d}\sum_{j=1}^d(\vartheta(x_j, y_j))^2\right)^{1/2},\\
\vartheta_\infty(x, y)&=\max_{1\le j\le d}\vartheta(x_j, y_j),
\end{align*}
for $x=(x_j)_{1\le j\le d}$ and $y=(y_j)_{1\le j\le d}$ in $(\Rb/\Zb)^{d\times 1}$.
Via the coordinate map at $e$, we shall also think of both $\vartheta_2$ and $\vartheta_\infty$ as continuous pseudometrics
on $((\Rb/\Zb)^{d\times 1})^\Gamma$, i.e.
$$ \vartheta_2(x, y)=\vartheta_2(x_e, y_e), \quad \vartheta_\infty(x, y)=\vartheta_\infty(x_e, y_e).$$
It is clear that both $\vartheta_2$ and $\vartheta_\infty$ are dynamically generating on $((\Rb/\Zb)^{d\times 1})^\Gamma$.
For any $F\in \cF(\Gamma)$, we shall write $\vartheta_{F, 2}$ and $\vartheta_{F, \infty}$ for $(\vartheta_2)_{F, 2}$ and $(\vartheta_\infty)_{F, \infty}$ respectively. Explicitly, for $x=(x_{s, j})_{s\in \Gamma, 1\le j\le d}$ and $y=(y_{s, j})_{s\in \Gamma, 1\le j\le d}$ in $((\Rb/\Zb)^{d\times 1})^\Gamma$,
one has
\begin{align}
\label{E-l2 metric} \vartheta_{F, 2}(x, y)&=\left(\frac{1}{d|F|}\sum_{s\in F, 1\le j\le d}(\vartheta(x_{s, j}, y_{s, j}))^2\right)^{1/2}, \\
\label{E-linfinity metric} \vartheta_{F, \infty}(x, y)&=\max_{s\in F, 1\le j\le d}\vartheta(x_{s, j}, y_{s, j}).
\end{align}
If $X$ is a closed $\Gamma$-invariant subgroup of $((\Rb/\Zb)^{d\times 1})^\Gamma$, then from Lemma~\ref{L-l2 definition of entropy} we get
\begin{align} \label{E-l2 definition of entropy}
\rh(X)=\sup_{\varepsilon>0} \limsup_F\frac{\log N_\varepsilon(X, \vartheta_{F, 2})}{|F|}=\sup_{\varepsilon>0} \limsup_F\frac{\log N_\varepsilon(X, \vartheta_{F, \infty})}{|F|}.
\end{align}

Let $J$ be a (possibly infinite) subset of $(\Zb\Gamma)^{1\times d}$. We denote by $\cM_J$ the $\Zb\Gamma$-submodule of $(\Zb\Gamma)^{1\times d}$ generated by $J$, and set
\begin{align} \label{E-X_J}
X_J=\widehat{(\Zb\Gamma)^{1\times d}/\cM_J}=\{x\in ((\Rb/\Zb)^{d\times 1})^\Gamma: fx=0 \mbox{ for all } f\in J\}.
\end{align}
For $F\in \cF(\Gamma)$, we set
\begin{align} \label{E-X_JF}
X_{J, F}=\{x\in ((\Rb/\Zb)^{d\times 1})^\Gamma: fx=0 \mbox{ on } F \mbox{ for all } f\in J\}.
\end{align}

The following theorem is the main result in this subsection.

\begin{theorem}  \label{T-approximate solution formula for entropy}
Let $J\subseteq (\Zb \Gamma)^{1\times d}$.  Then
$$ \rh(X_J)=\sup_{\varepsilon>0}\limsup_F \frac{\log N_\varepsilon(X_{J, F}, \vartheta_{F, \infty})}{|F|}.$$
\end{theorem}

Theorem~\ref{T-approximate solution formula for entropy} follows from \eqref{E-l2 definition of entropy} and the lemma below.

\begin{lemma}  \label{L-approximate solution formula for entropy}
Let $J\subseteq (\Zb\Gamma)^{1\times d}$ and $\varepsilon>0$.
Then
\begin{align} \label{E-approximate 0}
\limsup_F \frac{\log N_{4\varepsilon}(X_{J, F}, \vartheta_{F, \infty})}{|F|}\le  \limsup_F \frac{\log N_\varepsilon(X_J, \vartheta_{F, \infty})}{|F|}\le \limsup_F \frac{\log N_\varepsilon(X_{J, F}, \vartheta_{F, \infty})}{|F|}.
\end{align}
\end{lemma}

In order to prove the previous lemma,  we need the following quasitiling result of  Ornstein and Weiss \cite[page 24, Theorem 6]{OW} \cite[Theorem 8.3 and Remark 8.4]{Li}.

\begin{lemma} \label{L-quasitile}
Let $\varepsilon>0$ and  $K\in \cF(\Gamma)$.  Then there exist $\delta>0$ and
$K', F_1, \dots, F_m\in \cF(\Gamma)$  such
that
\begin{enumerate}
\item
$F_j\in \cB(K, \varepsilon)$ for each $1\le j\le m$;

\item for any $A\in \cB(K', \delta)$, there are
finite subsets $D_1, \dots,  D_m$ of $\Gamma$ such that $\bigcup_{1\le j\le m}F_jD_j\subseteq A$, the family
$\{F_jc_j: 1\le j\le m, c_j\in D_j\}$ of subsets of $\Gamma$ are pairwise disjoint, and $|\bigcup_{1\le j\le m}F_jD_j|\ge (1-\varepsilon)|A|$.
\end{enumerate}
\end{lemma}

We are now ready to prove Lemma~\ref{L-approximate solution formula for entropy}.
Let us first sketch the idea of the proof of Lemma~\ref{L-approximate solution formula for entropy} briefly. While the second inequality is obvious, the first relies on a quasi-tiling argument using Lemma~\ref{L-quasitile}. The limit superior on the left side is almost realized by sets which become more and more left invariant and can all be tiled with a fixed precision by finitely many fixed tiles. There are plenty of approximate solutions on certain right translations of the more and more left invariant set, being sufficiently separated on one of the tiles. We arrange the right translations of the more and more left invariant sets to exhaust the group $\Gamma$. Then the approximate solutions converge to precise solutions.

\begin{proof}[Proof of Lemma~\ref{L-approximate solution formula for entropy}] Since $X_J\subseteq X_{J, F}$ for every $F\in \cF(\Gamma)$, clearly the second
inequality in \eqref{E-approximate 0} holds.

Note that
\begin{align} \label{E-approximate solution 4}
N_\varepsilon(((\Rb/\Zb)^{d\times 1})^\Gamma, \vartheta_{F, \infty})\le (1+\varepsilon^{-1})^{d|F|}
\end{align}
for all $\varepsilon>0$ and  $F\in \cF(\Gamma)$.

Set $C=\limsup_F \frac{\log N_{4\varepsilon}(X_{J, F}, \vartheta_{F, \infty})}{|F|}<+\infty$.
To show the first inequality in \eqref{E-approximate 0}  it suffices to show that
for any $\eta>0$, $K\in \cF(\Gamma)$ and $\delta>0$ there exists $F\in \cB(K, \delta)$ such that
\begin{align} \label{E-approximate solution 5}
 \frac{\log N_\varepsilon(X_J, \vartheta_{F, \infty})}{|F|}\ge C-2\eta.
\end{align}
Let $\eta>0$, $K\in \cF(\Gamma)$ and $\delta>0$. We may assume that $C-2\eta>0$. Take $\delta_1>0$ such that
$(1+(2\varepsilon)^{-1})^{2\delta_1 d}\le \exp(\eta)$.

By Lemma~\ref{L-quasitile} there exist $\delta'>0$ and
$K', F_1, \dots, F_m\in \cF(\Gamma)$  such
that
\begin{enumerate}
\item
$F_j\in \cB(K, \delta)$ for each $1\le j\le m$;

\item for any $A\in \cB(K', \delta')$, there are
finite subsets $D_1, \dots,  D_m$ of $\Gamma$ such that $\bigcup_{1\le j\le m}F_jD_j\subseteq A$, the family
$\{F_jc_j: 1\le j\le m, c_j\in D_j\}$ of subsets of $\Gamma$ are pairwise disjoint, and $|\bigcup_{1\le j\le m}F_jD_j|\ge (1-\delta_1)|A|$.
\end{enumerate}
Enlarging $K'$ if necessary, we may assume that $e\in K'$.

Take an increasing sequence $\{e\}\in K_1\subseteq K_2\subseteq \dots$ of finite subsets of $\Gamma$ such that their union is $\Gamma$.

Fix $n\in \Nb$ and take $B\in \cB(K_nK', \min(\delta_1, \delta'))$ such that
\begin{align} \label{E-approximate solution 2}
 \frac{\log N_{4\varepsilon}(X_{J, B}, \vartheta_{B, \infty})}{|B|}\ge C-\eta.
\end{align}
Set $A=\{s\in B: K_ns\subseteq B\}$. Since $K_nK'\supseteq K_n$, one has $|A|\ge (1-\delta_1)|B|$.
Furthermore,
$$ |\{s\in A:K's\subseteq A\}|=|\{s\in B: K_nK's\subseteq B\}|\ge (1-\delta')|B|\ge (1-\delta')|A|.$$
Thus $A\in \cB(K', \delta')$.
Then we have finite subsets $D_1, \dots,  D_m$ of $\Gamma$ as above.

Set $W=B\setminus \bigcup_{1\le j\le m} F_jD_j$. Then
$$ |B\setminus A|=|B|-|A|\le \delta_1|A|,$$
and
$$\left|A\setminus \bigcup_{1\le j\le m} F_jD_j \right|=|A|-\left|\bigcup_{1\le j\le m} F_jD_j \right|\le \delta_1|A|.$$
Thus
$$ |W|=|B\setminus A|+\left|A\setminus \bigcup_{1\le j\le m} F_jD_j \right|\le \delta_1|B|+\delta_1|A|\le 2\delta_1|B|.$$
From \eqref{E-approximate solution 4} we have
\begin{align} \label{E-approximate solution 1}
N_{2\varepsilon}(X_{J, B}, \vartheta_{W, \infty})
 \le (1+(2\varepsilon)^{-1})^{d|W|}
 \le (1+(2\varepsilon)^{-1})^{2\delta_1 d|B|}
 \le \exp(|B|\eta).
\end{align}

Let $\cW_{j, c_j}$ (resp. $\cW_W$) be a maximal $2\varepsilon$-separated subset of $X_{J, B}$ under $\vartheta_{F_jc_j, \infty}$ for every $1\le j\le m$ and $c_j\in D_j$ (resp. under $\vartheta_{W, \infty}$). Also let $\cW_B$ be a $4\varepsilon$-separated subset of $X_{J, B}$ under $\vartheta_{B, \infty}$.
For each $x\in \cW_B$, we can find a point $\varphi(x)$ in $\cW_W\times\prod_{1\le j\le m, c_j\in D_j}\cW_{j, c_j}$ such
that
$$ \vartheta_{W, \infty}(x, \varphi(x)_W)\le 2\varepsilon,   \mbox{ and } \vartheta_{F_jc_j, \infty}(x, \varphi(x)_{j, c_j})\le 2\varepsilon, \forall 1\le j\le m, c_j\in D_j,$$
where $\varphi(x)_W$ denotes the coordinate of $\varphi(x)$ in $\cW_W$ and $\varphi(x)_{j, c_j}$ denotes the coordinate of $\varphi(x)$ in $\cW_{j, c_j}$.
Since $B$ is the union of $F_jc_j$ for $1\le j\le m, c_j\in D_j$ and $W$, we have
$$ \vartheta_{B, \infty}(x, y)=\max\left\{\vartheta_{W, \infty}(x, y), \max_{1\le j\le m}\max_{c_j\in D_j}\vartheta_{F_jc_j, \infty}(x, y) \right\}$$
for all $x, y\in ((\Rb/\Zb)^{d\times 1})^\Gamma$. It follows that the map $\varphi:\cW_B\rightarrow \cW_W\times\prod_{1\le j\le m, c_j\in D_j}\cW_{j, c_j}$ is injective.
Therefore
\begin{align} \label{E-approximate solution 3}
N_{4\varepsilon}(X_{J, B}, \vartheta_{B, \infty})&\le N_{2\varepsilon}(X_{J, B}, \vartheta_{W, \infty})\prod_{1\le j\le m}\prod_{c_j\in D_j}N_{2\varepsilon}(X_{J, B}, \vartheta_{F_jc_j, \infty})\\
\nonumber &\overset{\eqref{E-approximate solution 1}}\le \exp(|B|\eta)\prod_{1\le j\le m}\prod_{c_j\in D_j}N_{2\varepsilon}(X_{J, B}, \vartheta_{F_jc_j, \infty}).
\end{align}

From \eqref{E-approximate solution 3} and \eqref{E-approximate solution 2} we get
\begin{align*}
\prod_{1\le j\le m}\prod_{c_j\in D_j}N_{2\varepsilon}(X_{J, B}, \vartheta_{F_jc_j, \infty})&\ge \exp(|B|(C-2\eta))\\
&\ge \exp\left((C-2\eta)\sum_{1\le j\le m}\sum_{c_j\in D_j}|F_jc_j|\right),
\end{align*}
and hence we can find some $1\le j_n\le m$ and $c_{(n)}\in D_{j_n}$ such that
$$ N_{2\varepsilon}(X_{J, B}, \vartheta_{F_{j_n}c_{(n)}, \infty})\ge \exp((C-2\eta)|F_{j_n}c_{(n)}|)=\exp((C-2\eta)|F_{j_n}|).$$
Note that $\sigma_s(X_{J, Fs})=X_{J, F}$ and $\vartheta_{F, \infty}(\sigma_s(x), \sigma_s(y))=\vartheta_{Fs, \infty}(x, y)$
for all $s\in \Gamma$, $F\in \cF(\Gamma)$ and $x, y\in ((\Rb/\Zb)^{d\times 1})^\Gamma$. It follows that
$$  N_{2\varepsilon}(X_{J, B c_{(n)}^{-1}}, \vartheta_{F_{j_n}, \infty})=N_{2\varepsilon}(X_{J, B}, \vartheta_{F_{j_n}c_{(n)}, \infty})\ge \exp((C-2\eta)|F_{j_n}|).$$

Note that if $B_1\subseteq B_2$ are in $\cF(\Gamma)$, then $X_{J, B_2}\subseteq X_{J, B_1}$. Since $Bc_{(n)}^{-1}\supseteq K_nF_{j_n}$,
we have $X_{J, Bc_{(n)}^{-1}}\subseteq X_{J, K_nF_{j_n}}$. Therefore
$$ N_{2\varepsilon}(X_{J, K_nF_{j_n}}, \vartheta_{F_{j_n}, \infty})\ge N_{2\varepsilon}(X_{J, Bc_{(n)}^{-1}}, \vartheta_{F_{j_n}, \infty})\ge \exp((C-2\eta)|F_{j_n}|).$$

Passing to a subsequence of $\{K_n\}_{n\in \Nb}$ if necessary, we may assume that $j_n$ does not depend on $n$. Set $F=F_{j_n}$.
Denote by $M$ the smallest integer no less than $\exp((C-2\eta)|F|)$.
Then  $F\in \cB(K, \delta)$ and
$$N_{2\varepsilon}(X_{J, K_nF}, \vartheta_{F, \infty})\ge M$$
for all $n\in \Nb$.

For each $n\in \Nb$, take $x_{n, 1}, \dots, x_{n, M}$ in $X_{J, K_nF}$ such that
$$\vartheta_{F, \infty}(x_{n, i}, x_{n, j})\ge 2\varepsilon$$
for all $1\le i<j\le M$. Then $fx_{n, i}=0$ on $K_nF$ for all $n\in \Nb$, $1\le i\le M$, and $f\in J$.
Since $((\Rb/\Zb)^{d\times 1})^\Gamma$ is a compact metrizable space, passing to a subsequence of $\{K_n\}_{n\in \Nb}$ if necessary,
we may assume that, for each $1\le i\le M$, the sequence $x_{n, i}$ converges to some $x_i$ in $((\Rb/\Zb)^{d\times 1})^\Gamma$ as $n\to \infty$. Since $\vartheta_{F,
\infty}$ is a continuous pseudometric on $((\Rb/\Zb)^{d\times 1})^\Gamma$, we have
$$   \vartheta_{F, \infty}(x_i, x_j)=\lim_{n\to \infty}\vartheta_{F, \infty}(x_{n, i}, x_{n, j})\ge 2\varepsilon$$
for all $1\le i<j\le M$. Note that the map $((\Rb/\Zb)^{d\times 1})^\Gamma\rightarrow (\Rb/\Zb)^\Gamma$ sending $x$ to $fx$ is continuous for
every $f\in (\Zb\Gamma)^{1\times d}$. Thus for each $1\le i\le M$ one has
$$ fx_i=\lim_{n\to \infty}fx_{n, i}=0$$
for all $f\in J$,
and hence $x_i\in X_J$. Therefore
$$N_\varepsilon(X_J, \vartheta_{F, \infty})\ge M\ge \exp((C-2\eta)|F|),$$
establishing \eqref{E-approximate solution 5}.
\end{proof}

\subsection{The positive case} \label{SS-positive}

In this subsection we prove Theorem~\ref{T-main} in the case $f\in M_d(\Zb\Gamma)$ is positive in $M_d(\cN\Gamma)$ with $\ker f=\{0\}$.

\begin{lemma} \label{L-positive}
Let $f\in M_d(\Zb \Gamma)$ be positive in $M_d(\cN\Gamma)$ with $\ker f=\{0\}$.  Then
$$\rh(X_f)=\log \ddet_{\cN\Gamma}f.$$
\end{lemma}

To show $\rh(X_f)\le \log \ddet_{\cN\Gamma}f$, we apply the method used first in \cite{BL} for the case $\Gamma$ is a finitely generated residually finite group and
$f$ is a generalized Laplace operator. For any set $\cX$, we denote by $\ell^2_\Rb(\cX)$ the Hilbert space of all real-valued square-summable functions on $\cX$.

\begin{lemma} \label{L-bound preimage}
Let $\cX$ be a nonempty finite set, and $T\in B(\ell^2_\Rb(\cX))$ be injective and positive.
For each  $\eta>0$ denote by $B_\eta$ the closed ball in $\ell^2_\Rb(\cX)$ with center $0$ and radius $\eta |\cX|^{1/2}$.
Let $0<\kappa\le 1/2$. Denote by $\sD_\kappa$ the product of the eigenvalues of $T$ in $(0, \kappa]$ counted with multiplicity.
If $T$ has no eigenvalue in $(0, \kappa]$, we set $\sD_\kappa=1$.
Then
$N_1(T^{-1}(B_{\kappa/4}), \|\cdot\|_2/|\cX|^{1/2})\le 1/\sD_\kappa$.
\end{lemma}
\begin{proof}
Denote by $V$ the linear span of the eigenvectors of $T$ in $\ell^2_\Rb(\cX)$ with eigenvalue no bigger than $\kappa$, and
 by $P$ the orthogonal projection of $\ell^2_\Rb(\cX)$ onto $V$.

Note that for each $x\in \ell^2_\Rb(\cX)$ one has
$$\|Tx\|_2^2=\|T(Px)\|_2^2+\|T(x-Px)\|_2^2\ge \|T(x-Px)\|_2^2\ge \kappa^2\|x-Px\|_2^2.$$
Thus $\|x-Px\|_2/|\cX|^{1/2}\le 1/4$ for every $x\in T^{-1}(B_{\kappa/4})$.

Let $\cW$ be a $1$-separated subset of $T^{-1}(B_{\kappa/4})$ under $\|\cdot \|_2/|\cX|^{1/2}$ with
$$N_1(T^{-1}(B_{\kappa/4}), \|\cdot\|_2/|\cX|^{1/2})=|\cW|.$$
For any
distinct $x, y$ in $\cW$, one has $\|(x-y)-P(x-y)\|_2/|\cX|^{1/2}\le 1/2$, and hence $\|Px-Py\|_2/|\cX|^{1/2}>1/2$.

For each $z\in P(\cW)$,
denote by $B_z$ the closed ball in $V$ centered at $z$ with radius $1/4$ under $\|\cdot \|_2/|\cX|^{1/2}$.
For each $z\in P(\cW)$, say $z=Px$ for some $x\in \cW$, one has
$$\|Tz\|_2/|\cX|^{1/2}=\|TPx\|_2/|\cX|^{1/2}=\|PTx\|_2/|\cX|^{1/2}\le \|Tx\|_2/|\cX|^{1/2}\le \kappa/4.$$
Note that $\|Tx\|_2\le \kappa\|x\|_2$ for all $x\in V$. Thus every element in $T(\bigcup_{z\in P(\cW)}B_z)$ has $\|\cdot \|_2/|\cX|^{1/2}$-norm at most $\kappa/2$.

Denote by $E$ the multiset of all eigenvalues of $T$ in $(0, \kappa]$ listed with multiplicity.
Then we can find a basis of $V$  under which the matrix of $T|_V$ is diagonal with the diagonal entries being exactly the elements of $E$. Thus the volume of $T(\bigcup_{z\in P(\cW)}B_z)$ is $\det(T|_V)=\prod_{t \in E}t$ times the volume of $\bigcup_{z\in P(\cW)}B_z$.
Since the balls $B_z$ for $z\in P(\cW)$ are pairwise disjoint, we have
$$ |\cW| \prod_{t\in E} t\le \left(\frac{\kappa /2}{1/4}\right)^{\dim_\Rb (V)}=(2\kappa)^{\dim_\Rb (V)}\le 1.$$
Therefore
\begin{align*}
N_1(T^{-1}(B_{\kappa/4}), \|\cdot \|_2/|\cX|^{1/2})=|\cW| \le \prod_{t\in E}t^{-1}=1/\sD_\kappa.
\end{align*}
\end{proof}

We need the following well-known fact (see \cite[Lemma 4]{Solomyak} or \cite[Lemma 3.1]{Li}).

\begin{lemma} \label{L-determinant quotient group size}
Let $n\in \Nb$ and let $T:\Cb^n\rightarrow \Cb^n$ be an invertible linear map, preserving $\Zb^n$.
Then $|\det T|=|\Zb^n/T\Zb^n|$.
\end{lemma}

For $f=(f_{s,j,k})_{s\in \Gamma, 1\le j\le d', 1\le k\le d}\in M_{d'\times d}(\Zb\Gamma)$, we set
$$\|f\|_1=\sum_{s\in \Gamma}\sum_{1\le j\le d'}\sum_{1\le k\le d}|f_{s, j, k}|.$$
For $g=(g_{s, j})_{s\in \Gamma, 1\le j\le d}\in (\Rb^{d\times 1})^\Gamma$, we set
$$ \|g\|_\infty=\sup_{s\in \Gamma, 1\le j\le d}|g_{s, j}|,$$
and
$$ \|g\|_2=(\sum_{s\in \Gamma} \sum_{1\le j\le d}|g_{s, j}|^2)^{1/2}.$$

For a finite set $\cX$, we denote by $\Zb[\cX]$ the set of all $\Zb$-valued functions on $\cX$.
For any subset $K$ of $\Gamma$, we identify $(\Rb^{d\times 1})^K$ with the set of elements in $(\Rb^{d\times 1})^\Gamma$ with support contained in $K$.

For any $f\in M_{d'\times d}(\Zb\Gamma)$, let $J$ be the set of rows of $f$, then from \eqref{E-X_J} we have
\begin{equation} \label{E-Xf}
 X_f=X_J=\{x\in ((\Rb/\Zb)^{d\times 1})^\Gamma: fx=0\}.
\end{equation}

\begin{lemma} \label{L-upper}
Let $f\in M_d(\Zb \Gamma)$ be positive  in $M_d(\cN\Gamma)$ with $\ker f=\{0\}$. Then
$$\rh(X_f)\le \log \ddet_{\cN\Gamma}f.$$
\end{lemma}

The idea of the proof is to lift elements of some $\varepsilon$-separated subset of $X_f$ to elements in $(\Rb^{d\times 1})^{\Gamma}$ and restrict them to elements in $(\ell^2 (F))^{d \times 1}$. These are mapped by $f_F$ to $(\Zb[F])^{d\times 1}/f_F(\Zb[F])^{d\times 1}$. We use Lemma~\ref{L-bound preimage} to control the cardinality of the fiber of this map by spectral information about $f_F$. Finally, this implies the inequality using Lemma~\ref{L-measure weak convergence}.

\begin{proof}
Denote by $K$ the support of $f$ as a $M_d(\Zb)$-valued function on $\Gamma$.
For each $x\in ((\Rb/\Zb)^{d\times 1})^\Gamma$, take $\tilde{x}\in ([-1/2, 1/2)^{d\times 1})^\Gamma$ such that $x_s=\tilde{x}_s+\Zb^{d\times 1}$ for all $s\in \Gamma$. Then for each $x\in X_f$ one has $f\tilde{x}\in (\Zb^{d\times 1})^\Gamma$.

Let $0<\varepsilon<1$, and $0<\kappa\le 1/2$.
Set $\eta=\varepsilon \kappa/4$.
Let $F\in \cF(\Gamma)$ such that $\|f\|\cdot |K^{-1}F\setminus F|^{1/2}/2\le \eta |F|^{1/2}$.
Recall the linear map $f_F:(\ell^2(F))^{d\times 1}\rightarrow (\ell^2(F))^{d\times 1}$ defined by \eqref{E-g_F}.
Define a map $\Phi_F: X_f\rightarrow (\Zb[F])^{d\times 1}/f_F(\Zb[F])^{d\times 1}$ sending
$x$ to $(f\tilde{x})|_F+f_F(\Zb[F])^{d\times 1}$. By Lemma~\ref{L-invertible} we know that $f_F$ is invertible.
From Lemma~\ref{L-determinant quotient group size} we get
\begin{align} \label{E-upper bound 2}
|(\Zb[F])^{d\times 1}/f_F(\Zb[F])^{d\times 1}|=\det(f_F).
\end{align}

Let $\cW$ be a $(\vartheta_{F, 2}, \varepsilon)$-separated subset of $X_f$ with $|\cW|=N_\varepsilon(X_f, \vartheta_{F, 2})$.
Then we can find a subset $\cW_1$ of $\cW$ such that
$$|\cW|\le |\cW_1|\cdot |(\Zb[F])^{d\times 1}/f_F(\Zb[F])^{d\times 1}|$$
and  $\Phi_F$ takes the same value on $\cW_1$.

Let $x\in \cW_1$. Then $f\tilde{x}\in (\Zb^{d\times 1})^\Gamma$ and $\|f\tilde{x}\|_\infty\le \|f\|_1/2$.
Since $f_F$ is bijective and has real coefficients under the natural basis of $(\ell^2(F))^{d\times 1}$, it restricts to an invertible linear map
$T:(\ell^2_\Rb(F))^{d\times 1}\rightarrow (\ell^2_\Rb(F))^{d\times 1}$.
Thus we can find a unique $x'\in (\ell^2_\Rb(F))^{d\times 1}$ such that $fx'=f(\tilde{x}|_{\Gamma \setminus F})$ on $F$.
Then $fx'+f(\tilde{x}|_F)=f\tilde{x}$ on $F$.
Note that $f(\tilde{x}|_{\Gamma \setminus F})=f(\tilde{x}|_{K^{-1}F\setminus F})$ on $F$.
Thus
\begin{align*}
\|f_Fx'\|_2&=\|(fx')|_F\|_2\\
&=\|(f(\tilde{x}|_{\Gamma\setminus F}))|_F\|_2\\
&=\|(f(\tilde{x}|_{K^{-1}F\setminus F}))|_F\|_2\\
&\le \|f(\tilde{x}|_{K^{-1}F\setminus F})\|_2\\
&\le \|f\|\cdot \|\tilde{x}|_{K^{-1}F\setminus F}\|_2\\
&\le \|f\|\cdot (d|K^{-1}F\setminus F|)^{1/2}/2\le \eta (d|F|)^{1/2}.
\end{align*}
For $r>0$ denote by $B_{F, r}$ the closed  ball in $(\ell^2_\Rb(F))^{d\times 1}$ with center $0$ and radius $r(d|F|)^{1/2}$. Then $x'\in T^{-1}(B_{F, \eta})$.

Recall $\sD_{f, F, \kappa}$  in Notation~\ref{N-product of eigenvalues}. Taking $\cX=F\times \{1, \dots, d\}$  in  Lemma~\ref{L-bound preimage}
we have
\begin{align*}
N_\varepsilon(T^{-1}(B_{F, \eta}), \|\cdot\|_2/(d|F|)^{1/2})
&=N_1(T^{-1}(B_{F, \kappa/4}), \|\cdot\|_2/(d|F|)^{1/2})\\
&\le 1/\sD_{f, F, \kappa}.
\end{align*}
Thus we can find some $\cW_2\subseteq \cW_1$ and $y\in \cW_2$ such that
$$|\cW_1|\le  |\cW_2|/\sD_{f, F, \kappa}$$
and for every $x\in \cW_2$ one has $\|x'-y'\|_2\le \varepsilon (d|F|)^{1/2}$. We fix such an element $y$ of $\cW_2$.

Let $x\in \cW_2$. Then $\Phi_F(x)=\Phi_F(y)$. Thus there exists $w_x\in (\Zb[F])^{d\times 1}$ such
that $fw_x=f\tilde{x}-f\tilde{y}$ on $F$. Then $fw_x=f(x'-y')+f(\tilde{x}|_F-\tilde{y}|_F)$ on $F$. Since
$f_F$ is injective, we get $w_x=x'-y'+\tilde{x}|_F-\tilde{y}|_F$. From \eqref{E-l2 metric} we get
$$\vartheta_{F, 2}(x, y)\le \|\tilde{x}|_F-\tilde{y}|_F-w_x\|_2/(d|F|)^{1/2}=\|x'-y'\|_2/(d|F|)^{1/2}\le \varepsilon.$$
As $\cW$ is $(\vartheta_{F, 2}, \varepsilon)$-separated, we get $x=y$. Thus $|\cW_2|=1$, and hence
\begin{align} \label{E-upper bound 4}
N_\varepsilon(X_f, \vartheta_{F, 2})&=|\cW|\\
\nonumber &\le |\cW_1|\cdot |(\Zb[F])^{d\times 1}/f_F(\Zb[F])^{d\times 1}| \\
\nonumber &\le |\cW_2|\cdot |(\Zb[F])^{d\times 1}/f_F(\Zb[F])^{d\times 1}|/\sD_{f, F, \kappa}\\
\nonumber &=  |(\Zb[F])^{d\times 1}/f_F(\Zb[F])^{d\times 1}|/\sD_{f, F, \kappa}\\
\nonumber &\overset{\eqref{E-upper bound 2}}= \det(f_F)/\sD_{f, F, \kappa}\\
\nonumber &\overset{\eqref{E-spectral measure and multiplicity}}=\exp\left(|F|\int_{\kappa+}^{\|f\|}\log t \, d\mu_{f, F}(t)\right).
\end{align}

Taking $g=f$ in  Lemma~\ref{L-measure weak convergence} we get
$$ \limsup_F\frac{\log N_\varepsilon(X_f, \vartheta_{F, 2})}{|F|}\overset{\eqref{E-upper bound 4}}\le \limsup_F\int_{\kappa+}^{\|f\|}\log t \, d\mu_{f, F}(t)\le \int_{\kappa+}^{\|f\|}\log t \, d\mu_f(t).$$
Letting $\kappa\to 0+$, we obtain
$$ \limsup_F\frac{\log N_\varepsilon(X_f, \vartheta_{F, 2})}{|F|}\le \int_{0+}^{\|f\|}\log t \, d\mu_f(t).$$
Since $\ker f=\{0\}$, from Section~\ref{SS-determinant} we have $\mu_f(\{0\})=0$. Thus
$$ \limsup_F\frac{\log N_\varepsilon(X_f, \vartheta_{F, 2})}{|F|}\le \int_{0+}^{\|f\|}\log t \, d\mu_f(t)=\int_0^{\|f\|}\log t \, d\mu_f(t)\overset{\eqref{E-determinant}}=\log \ddet_{\rm \cN\Gamma}f.$$
From \eqref{E-l2 definition of entropy} we get
$$ \rh(X_f)=\sup_{\varepsilon>0} \limsup_F\frac{\log N_\varepsilon(X_f, \vartheta_{F, 2})}{|F|}\le \log \ddet_{\rm \cN\Gamma}f.$$
\end{proof}

To show $\rh(X_f)\ge \log \ddet_{\rm \cN\Gamma}f$, we use Theorem~\ref{T-approximate solution formula for entropy} to prove the following
lemma and then apply Theorem~\ref{T-approximate}.

\begin{lemma}  \label{L-lower bound for positive case}
Let $f\in M_d(\Zb \Gamma)$ be positive in $M_d(\cN\Gamma)$ with $\ker f=\{0\}$.  Then
$$ \rh(X_f)\ge \limsup_{F} \frac{\log \det(f_F)}{|F|}.$$
\end{lemma}
\begin{proof} Set $\varepsilon=1/(2\|f\|_1)$.
Denote by $P$ the natural quotient map $(\ell^\infty_{\Rb}(\Gamma))^{d\times 1}\rightarrow ((\Rb/\Zb)^\Gamma)^{d\times 1}=((\Rb/\Zb)^{d\times 1})^\Gamma$. Denote by $J$ the set of rows of $f$. Then $J\subseteq (\Zb\Gamma)^{1\times d}$ and
$X_f=X_J$.

Let $F\in \cF(\Gamma)$. Denote by $Y_F$ the set of $x\in ([0, 1)^F)^{d\times 1}$ satisfying $f_Fx\in (\Zb[F])^{d\times 1}$. Then $P(Y_F)\subseteq X_{J, F}$.
Let $x, y\in Y_F$. Suppose that $\vartheta_{F, \infty}(P(x), P(y))\le \varepsilon$. From \eqref{E-linfinity metric} we can find $z\in (\Zb[F])^{d\times 1}$ such
that $\|x-y-z\|_\infty\le \varepsilon$. Note that
$f(x-y-z)=fx-fy-fz$ takes values in $\Zb^{d\times 1}$ on $F$, and
$\|f(x-y-z)\|_\infty\le \|f\|_1\|x-y-z\|_\infty\le 1/2$.
Thus $f(x-y-z)=0$ on $F$, i.e., $f_F(x-y-z)=0$. Since $f_F$ is injective by Lemma~\ref{L-invertible},
we get $x-y-z=0$. Because $x-y$ takes values in $(-1, 1)^{d\times 1}$ on $F$ and $z$ takes values in $\Zb^{d\times 1}$ on $F$, we get $x=y$.
Therefore
$$N_\varepsilon(X_{J, F}, \vartheta_{F, \infty})\ge |Y_F|.$$

From Lemma~\ref{L-determinant quotient group size} one has
\begin{align*}
|Y_F|=\det(f_F).
\end{align*}

By Theorem~\ref{T-approximate solution formula for entropy} we get
$$ \rh(X_f)=\rh(X_J)\ge \limsup_F\frac{\log N_\varepsilon(X_{J, F}, \vartheta_{F, \infty})}{|F|}\ge \limsup_F\frac{\log \det(f_F)}{|F|}.$$
\end{proof}

Now Lemma~\ref{L-positive} follows from Lemmas~\ref{L-upper} and \ref{L-lower bound for positive case}, and Theorem~\ref{T-approximate}.

\subsection{Proof of Theorem~\ref{T-main}} \label{SS-proof}

In this subsection we prove Theorem~\ref{T-main}.
Let $\Gamma$ act on a compact metrizable abelian group $X$ by automorphisms.
For any nonempty finite subset $\sE$ of the Pontryagin dual $\widehat{X}$ of $X$, the function $F\mapsto \log |\sum_{s\in F}s^{-1}\sE|$ defined on $\cF(\Gamma)$
satisfies
the conditions of the Ornstein-Weiss lemma \cite[Theorem 6.1]{LW}, thus the limit
$$\lim_F\frac{\log |\sum_{s\in F}s^{-1}\sE|}{|F|}$$ exists and is a nonnegative real number.
 We need the following  result of Peters \cite[Theorem 6]{Peters}:

\begin{theorem} \label{T-Peters}
Let $\Gamma$ act on a compact metrizable abelian group $X$ by automorphisms. Then
$$ \rh(X)=\sup_\sE\lim_F\frac{\log |\sum_{s\in F}s^{-1}\sE|}{|F|},$$
where $\sE$ ranges over all nonempty finite subsets of $\widehat{X}$.
\end{theorem}

In \cite{Peters}, Theorem~\ref{T-Peters} was stated and proved only for the case $\Gamma=\Zb$, but the proof there works for general countable discrete amenable groups.
It was used in \cite{CL} first to study entropy properties of algebraic actions.

We specialize Theorem~\ref{T-Peters} to the case of finitely presented algebraic actions.
The following lemma shows that we may restrict attention to a specific family of finite subsets of the Pontrjagin dual if $X=X_g$.
We need to introduce some notations.
For $n\in \Nb$ and $F\in \cF(\Gamma)$, we denote by $\Zb[F, n]$ the set of $x\in \Zb[F]$ satisfying $\|x\|_\infty \le n$. For $f\in M_{d'\times d}(\Zb\Gamma)$ and $W\subseteq (\Zb\Gamma)^{d\times 1}$ (resp.~$W\subseteq (\Zb\Gamma)^{1\times d}$), we denote
by $W+f^*(\Zb\Gamma)^{d'\times 1}$ (resp.~$W+(\Zb\Gamma)^{1\times d'} f$) the subset of $(\Zb\Gamma)^{d\times 1}/f^*(\Zb\Gamma)^{d'\times 1}$ (resp.~$(\Zb\Gamma)^{1\times d}/(\Zb\Gamma)^{1\times d'} f$) consisting of elements of the form $x+f^*(\Zb\Gamma)^{d'\times 1}$
(resp.~$x+(\Zb\Gamma)^{1\times d'} f$) with $x\in W$.

\begin{lemma} \label{L-Peters}
Let $g\in M_{d'\times d}(\Zb\Gamma)$. For each $n\in \Nb$, the limit
$$\lim_F\frac{\log |(\Zb[F, n])^{d\times 1}+g^*(\Zb\Gamma)^{d'\times 1}|}{|F|}$$ exists, and is a nonnegative real number.
Furthermore,
$$ \rh(X_g)=\sup_{n\in \Nb}\lim_F\frac{\log |(\Zb[F, n])^{d\times 1}+g^*(\Zb\Gamma)^{d'\times 1}|}{|F|}.$$
\end{lemma}
\begin{proof}
The adjoint map $(\Zb\Gamma)^{1\times d}\rightarrow (\Zb\Gamma)^{d\times 1}$ sending $h$ to $h^*$ induces an abelian group isomorphism $\Phi$ from
$\widehat{X_g}=(\Zb\Gamma)^{1\times d}/(\Zb\Gamma)^{1\times d'} g$ onto $(\Zb\Gamma)^{d\times 1}/g^*(\Zb\Gamma)^{d'\times 1}$.

Let $n\in \Nb$. Via the embedding $\Zb\hookrightarrow \Zb\Gamma$ sending $a$ to $ae$ we identify $\Zb$ with a subring of $\Zb\Gamma$. Then we identify $\Zb^{1\times d}$ with a subgroup of $(\Zb\Gamma)^{1\times d}$. Denote by $\sE$ the finite subset $\{k+(\Zb\Gamma)^{1\times d'}g: k\in \Zb^{1\times d}, \|k\|_\infty\le n\}$ of $(\Zb\Gamma)^{1\times d}/(\Zb\Gamma)^{1\times d'}g$. For each $F\in \cF(\Gamma)$, we have
$\sum_{s\in F}s^{-1}\sE=(\Zb[F^{-1}, n])^{1\times d}+(\Zb\Gamma)^{1\times d'} g$, and hence
$\Phi(\sum_{s\in F}s^{-1}\sE)=(\Zb[F, n])^{d\times 1} +g^*(\Zb\Gamma)^{d'\times 1}$. Thus
$$\lim_F\frac{\log |(\Zb[F, n])^{d\times 1}+g^*(\Zb\Gamma)^{d'\times 1}|}{|F|}=\lim_F\frac{\log |\sum_{s\in F}s^{-1}\sE|}{|F|}\in [0, +\infty).$$
 From Theorem~\ref{T-Peters} we get
$$ \rh(X_g)\ge \sup_{n\in \Nb}\lim_F\frac{\log |(\Zb[F, n])^{d\times 1}+g^*(\Zb\Gamma)^{d'\times 1}|}{|F|}.$$

Let $\sE$ be a nonempty finite subset of $\widehat{X_g}=(\Zb\Gamma)^{1\times d}/(\Zb\Gamma)^{1\times d'} g$. Then there exist $n\in \Nb$ and $K\in \cF(\Gamma)$ such that
$\sE\subseteq (\Zb[K, n])^{1\times d}+(\Zb\Gamma)^{1\times d'} g$. Let $F\in \cF(\Gamma)$.
Note that each element of $\Gamma$ can be written as $s^{-1}t$ with $s\in F$ and $t\in K$ for at most $|K|$ ways. Thus
$$\sum_{s\in F}s^{-1}\sE\subseteq (\sum_{s\in F}s^{-1}\Zb[K, n])^{1\times d}+(\Zb\Gamma)^{1\times d'} g\subseteq (\Zb[F^{-1}K, |K|n])^{1\times d}+(\Zb\Gamma)^{1\times d'} g.$$
When $F$ becomes more and more left invariant, $K^{-1}F$ also becomes more and more left invariant, and $|K^{-1}F|/|F|\to 1$. Therefore
\begin{align*}
\lim_F\frac{\log|\sum_{s\in F}s^{-1}\sE|}{|F|}&\le \limsup_F\frac{\log|(\Zb[F^{-1}K, |K|n])^{1\times d}+(\Zb\Gamma)^{1\times d'} g|}{|F|} \\
&=\limsup_F\frac{\log|(\Zb[F^{-1}K, |K|n])^{1\times d}+(\Zb\Gamma)^{1\times d'} g|}{|K^{-1}F|} \\
&=\lim_F\frac{\log|(\Zb[F^{-1}, |K|n])^{1\times d}+(\Zb\Gamma)^{1\times d'} g|}{|F|} \\
&=\lim_F\frac{\log|(\Zb[F, |K|n])^{d\times 1}+g^*(\Zb\Gamma)^{d'\times 1} |}{|F|}.
\end{align*}
Since $\sE$ is an arbitrary nonempty finite subset of $\widehat{X_g}$, from Theorem~\ref{T-Peters} we get
$$ \rh(X_g)\le \sup_{n\in \Nb}\lim_F\frac{\log |(\Zb[F, n])^{d\times 1}+g^*(\Zb\Gamma)^{d'\times 1}|}{|F|}.$$
\end{proof}

We need the following result:

\begin{lemma}[Lemma 5.1 in \cite{Li}] \label{L-ball 1}
There exists some universal constant $C>0$ such that for any
$\lambda>1$,
there is some $0<\delta<1$ so that for any nonempty finite
set $\cX$, any positive integer $m$ with $|\cX|\le \delta m$, and any $M\ge 1$ one has
$$ |\{x\in \Zb[\cX]: \|x\|_2\le M\cdot m^{1/2}\}|\le
C\lambda^m M^{|\cX|}.$$
\end{lemma}

Now, we establish a relationship between the cardinality of $(\Zb[F, n])^{d\times 1}+f^*(\Zb\Gamma)^{d'\times 1}$ and $N_{\varepsilon}(X_{J, F}, \vartheta_{F, \infty})$. This will be the key to prove Lemma~\ref{L-f < J}. Once we have established Lemma~\ref{L-f < J}, Theorem~\ref{T-main} will basically follow from an application of Yuzvinski\u{\i}'s addition formula.

\begin{lemma} \label{L-f < f*}
Let $\lambda>1$, $0<\kappa\le 1/2$, and $n\in \Nb$.
Let $C$ and $\delta$ be as in Lemma~\ref{L-ball 1}.
Let $f\in M_{d'\times d}(\Zb\Gamma)$ such that $\ker f=\{0\}$.
Denote by $J$ the subset of $(\Zb\Gamma)^{1\times d'}$ consisting of all rows of $f^*$ and $g\in (\Zb\Gamma)^{1\times d'}$ satisfying
$gf=0$.
Set $M=(8n\|f\|^2/\kappa) +2\|f\|(d'/d)^{1/2}$.
Denote by $K$ the support of $f$ as a $M_{d'\times d}(\Zb)$-valued function on $\Gamma$ union with $\{e\}$.
Then for any $F\in \cF(\Gamma)$ satisfying
$|K^{-1}KF\setminus F|\le \delta |F|$,
one has
\begin{align} \label{E-compare f and J}
 |(\Zb[F, n])^{d\times 1}+f^*(\Zb\Gamma)^{d'\times 1}|&\le C\lambda^{d|F|}M^{d|K^{-1}KF\setminus F|}
 (1+4\|f\|_1)^{d'|KF\setminus F|}\\
\nonumber &\hspace*{1cm} N_{1/(4\|f\|_1)}(X_{J, F}, \vartheta_{F, \infty})(\sD_{f^*f, F, \kappa})^{-1},
\end{align}
where $X_{J, F}$ and $\sD_{f^*f, F, \kappa}$ are defined by \eqref{E-X_JF} and  Notation~\ref{N-product of eigenvalues} respectively.
If furthermore $\Gamma$ is finite, then
\begin{align} \label{E-compare f and J finite}
 |(\Zb[\Gamma, n])^{d\times 1}+f^*(\Zb\Gamma)^{d'\times 1}|\le  N_{1/(2\|f\|_1)}(X_{J, \Gamma}, \vartheta_{\Gamma, \infty})(\sD_{f^*f, \Gamma, \kappa})^{-1}.
\end{align}
\end{lemma}
\begin{proof} Take $\cW\subseteq (\Zb[F, n])^{d\times 1}$ such that
\begin{align} \label{E-f < f* 1}
|\cW|=|(\Zb[F, n])^{d\times 1}+f^*(\Zb\Gamma)^{d'\times 1}|
\end{align}
 and $x-y\not \in f^*(\Zb\Gamma)^{d'\times 1}$ for all distinct $x, y\in \cW$.

Note that $\ker (f^*f)=\ker f=\{0\}$. Thus by Lemma~\ref{L-invertible} the linear map $(f^*f)_F$ defined by \eqref{E-g_F} is invertible.
Since $(f^*f)_F$ is bijective and has real coefficients under the natural basis of $(\ell^2(F))^{d\times 1}$, it restricts to an invertible linear map
$T: (\ell^2_\Rb(F))^{d\times 1}\rightarrow (\ell^2_\Rb(F))^{d\times 1}$.

For each $x\in \cW$, take $x'\in (\ell^2_\Rb(F))^{d\times 1}$ such that $f^*fx'=x$ on $F$. Set $\cW'=\{x': x\in \cW\}$.
For $r>0$ denote by $B_{F, r}$ the closed ball in $(\ell^2_\Rb(F))^{d\times 1}$ with center $0$ and radius $r(d|F|)^{1/2}$.
For each $x\in \cW$, one has
$$\|(f^*fx')|_F\|_2/(d|F|)^{1/2}=\|x\|_2/(d|F|)^{1/2}\le \|x\|_\infty \le n.$$
Thus $\cW'\subseteq T^{-1}(B_{F, n})$.
Taking $\cX=F\times \{1, \dots, d\}$ in Lemma~\ref{L-bound preimage} we get
\begin{align*}
N_{4n/\kappa}(\cW', \|\cdot\|_2/(d|F|)^{1/2})
&\le N_{4n/\kappa}(T^{-1}(B_{F, n}), \|\cdot\|_2/(d|F|)^{1/2})\\
&= N_1(T^{-1}(B_{F, \kappa/4}), \|\cdot\|_2/(d|F|)^{1/2})\\
&\le 1/\sD_{f^*f, F, \kappa}.
\end{align*}
Thus we can find some $\cW_1\subseteq \cW$ and $y\in \cW_1$ such that
\begin{align} \label{E-f < f* 2}
|\cW|\le |\cW_1|/\sD_{f^*f, F, \kappa}
\end{align}
and for every $x\in \cW_1$ one has
\begin{align} \label{E-f < f* 3}
\|x'-y'\|_2\le (4n/\kappa) (d|F|)^{1/2}.
\end{align}
We fix such an element $y$ of $\cW_1$.

Denote by $P$ the natural quotient map $(\ell^\infty_\Rb(\Gamma))^{d'\times 1}\rightarrow ((\Rb/\Zb)^\Gamma)^{d'\times 1}=((\Rb/\Zb)^{d'\times 1})^\Gamma$. For each $x\in \cW_1$, the function
$f^*f(x'-y')=x-y$ takes values  in $\Zb^{d\times 1}$ on $F$, and it follows that $P(f(x'-y'))\in X_{J, F}$.
Then we can find some $\cW_2\subseteq \cW_1$ and $z\in \cW_2$ such that
\begin{align} \label{E-f < f* 4}
 |\cW_1|\le |\cW_2|N_{1/(2\|f\|_1)}(X_{J, F}, \vartheta_{KF, \infty})
\end{align}
and for every $x\in \cW_2$ one has
\begin{align} \label{E-f<f* 10}
\vartheta_{KF, \infty}(P(f(x'-z')), 0)=\vartheta_{KF, \infty}(P(f(x'-y')), P(f(z'-y')))\le 1/(2\|f\|_1).
\end{align}
We fix such an element $z$ of $\cW_2$.

Note that
$$ \vartheta_{KF, \infty}(u, v)=\max(\vartheta_{F, \infty}(u, v), \vartheta_{KF\setminus F}(u, v))$$
for all $u, v\in ((\Rb/\Zb)^{d'\times 1})^\Gamma$. Thus
\begin{align} \label{E-f < f* 5}
N_{1/(2\|f\|_1)}(X_{J, F}, \vartheta_{KF, \infty})&\le N_{1/(4\|f\|_1)}(X_{J, F}, \vartheta_{F, \infty}) \cdot N_{1/(4\|f\|_1)}(X_{J, F}, \vartheta_{KF\setminus F, \infty}) \\
\nonumber &\le N_{1/(4\|f\|_1)}(X_{J, F}, \vartheta_{F, \infty}) \cdot N_{1/(4\|f\|_1)}(((\Rb/\Zb)^{d'\times 1})^\Gamma, \vartheta_{KF\setminus F, \infty})\\
\nonumber &\overset{\eqref{E-approximate solution 4}}\le N_{1/(4\|f\|_1)}(X_{J, F}, \vartheta_{F, \infty}) \cdot (1+4\|f\|_1)^{d'|KF\setminus F|}.
\end{align}

Let $x\in \cW_2$. Note that the support of $f(x'-z')$ as a $\Rb^{d'\times 1}$-valued function on $\Gamma$ is contained in $KF$. Take $\tilde{x}\in ([-1/2, 1/2)^{d'\times 1})^{KF}$
such that $f(x'-z')-\tilde{x}\in (\Zb[KF])^{d'\times 1}$. From \eqref{E-linfinity metric} we have
$$\|\tilde{x}\|_\infty=\vartheta_{KF, \infty}(P(f(x'-z')), 0)\overset{\eqref{E-f<f* 10}}\le 1/(2\|f\|_1).$$
Set $x^\dag=f(x'-z')-\tilde{x}\in (\Zb[KF])^{d'\times 1}$.
Note that both $x-z$ and $f^*x^\dag$ are  in $(\Zb\Gamma)^{d\times 1}$, and $\|f^*\tilde{x}\|_\infty\le \|f\|_1\|\tilde{x}\|_\infty\le 1/2$.
Since
$$ x-z=f^*f(x'-z')=f^*\tilde{x}+f^*x^\dag$$
on $F$, we get $f^*\tilde{x}=0$ on $F$ and $x-z=f^*x^\dag$ on $F$. Also note that
\begin{align} \label{E-f < f* 6}
 \|f^*x^\dag\|_2 &\le \|f^*\|\cdot \|x^\dag\|_2\\
\nonumber &\le \|f\|\cdot\|f(x'-z')\|_2+\|f\|\cdot \|\tilde{x}\|_2\\
\nonumber &\le \|f\|^2\|x'-z'\|_2+\|f\|\cdot \|\tilde{x}\|_\infty (d'|KF|)^{1/2}\\
\nonumber &\le \|f\|^2(\|x'-y'\|_2+\|y'-z'\|_2)+\|f\|(d'|K^{-1}KF|)^{1/2} \\
\nonumber &\overset{\eqref{E-f < f* 3}}\le (8n\|f\|^2/\kappa)(d|F|)^{1/2} +2\|f\|(d'|F|)^{1/2}\\
\nonumber &= M (d|F|)^{1/2}.
\end{align}

From \eqref{E-f < f* 6} and Lemma~\ref{L-ball 1} we get that the cardinality of the set $\{(f^*x^\dag)|_{K^{-1}KF\setminus F}: x\in \cW_2\}$ is at most
$C\lambda^{d|F|}M^{d|K^{-1}KF\setminus F|}$. Thus we can find some $\cW_3\subseteq \cW_2$  such that
\begin{align} \label{E-f < f* 7}
|\cW_2|\le |\cW_3| C\lambda^{d|F|}M^{d|K^{-1}KF\setminus F|}
\end{align}
and $(f^*x^\dag)|_{K^{-1}KF\setminus F}$ is the same for all $x\in \cW_3$.

Let $x, w\in \cW_3$.
Then
$$x-w=(x-z)-(w-z)=f^*(x^\dag-w^\dag)$$
on $F$. Also $f^*(x^\dag-w^\dag)=0$ on $K^{-1}KF\setminus F$. Since the supports of $f^*(x^\dag-w^\dag)$ and $x-w$ as $\Zb^{d\times 1}$-valued functions on $\Gamma$ are contained in $K^{-1}KF$ and $F$ respectively,
we conclude that
\begin{align*}
 x-w=f^*(x^\dag-w^\dag)\in f^*(\Zb\Gamma)^{d'\times 1}.
\end{align*}
By the choice of $\cW$ we have $x=w$. Therefore $|\cW_3|=1$.

 Now \eqref{E-compare f and J} follows from the inequalities \eqref{E-f < f* 1}, \eqref{E-f < f* 2}, \eqref{E-f < f* 4}, \eqref{E-f < f* 5}, and \eqref{E-f < f* 7}.

 Suppose that $\Gamma$ is finite and $F=\Gamma$. For any $x\in \cW_2$, we have $x-z=f^*x^\dag\in f^*(\Zb\Gamma)^{d'\times 1}$, and hence $x=z$ by the choice of $\cW$. Thus $|\cW_2|=1$. Then \eqref{E-compare f and J finite} follows from the inequalities \eqref{E-f < f* 1}, \eqref{E-f < f* 2}, and \eqref{E-f < f* 4}.
\end{proof}

\begin{lemma} \label{L-f < J}
Let $f\in M_{d'\times d}(\Zb\Gamma)$ with $\ker f=\{0\}$.
Denote by $J$ the subset of $(\Zb\Gamma)^{1\times d'}$ consisting of all rows of $f^*$ and $g\in (\Zb\Gamma)^{1\times d'}$ satisfying
$gf=0$. Then
$$ \rh(X_f)\le \rh(X_J)\le \rh(X_{f^*}),$$
where $X_J$ is defined by \eqref{E-X_J}.
\end{lemma}
\begin{proof}According to \eqref{E-Xf}, we have $X_J\subseteq X_{f^*}$. Thus $\rh(X_J)\le \rh(X_{f^*})$.

In order to show $\rh(X_f)\le \rh(X_J)$, by Lemma~\ref{L-Peters} applied to $g=f$, it suffices to show
$$\lim_F\frac{\log |(\Zb[F, n])^{d\times 1}+f^*(\Zb\Gamma)^{d'\times 1}|}{|F|}\le \rh(X_J)+(d+1)\log \lambda $$
for all $n\in \Nb$ and $\lambda>1$. Let $n\in \Nb$ and $\lambda>1$.

We have $\ker (f^*f)=\ker f=\{0\}$. By Lemma~\ref{L-integral to positive determinant} and Proposition~\ref{P-measure weak convergence} applied to $g=f^*f$  there exists $0<\kappa<1$ such that
$$ \limsup_F\frac{\log (\sD_{f^*f, F, \kappa})^{-1}}{|F|}\le \log \lambda.$$
We may assume that $\kappa\le 1/2$.

By Theorem~\ref{T-approximate solution formula for entropy} we have
$$ \limsup_F\frac{\log N_{1/(4\|f\|_1)}(X_{J, F}, \vartheta_{F, \infty})}{|F|}\le \rh(X_J).$$

 From Lemma~\ref{L-f < f*} we get
\begin{align*}
\lim_F\frac{\log |(\Zb[F, n])^{d\times 1}+f^*(\Zb\Gamma)^{d'\times 1}|}{|F|}&\le d\log  \lambda +\limsup_F\frac{\log N_{1/(4\|f\|_1)}(X_{J, F}, \vartheta_{F, \infty})}{|F|}\\
&\hspace*{1cm}+\limsup_F\frac{\log (\sD_{f^*f, F, \kappa})^{-1}}{|F|}\\
&\le (d+1)\log \lambda+\rh(X_J)
\end{align*}
as desired.
\end{proof}

\begin{lemma} \label{L-exact}
Let $f\in M_{d'\times d}(\Zb\Gamma)$.
Denote by $J$ the subset of $(\Zb\Gamma)^{1\times d'}$ consisting of all rows of $f^*$ and $g\in (\Zb\Gamma)^{1\times d'}$ satisfying
$gf=0$.
Then one has a $\Gamma$-equivariant short exact sequence of compact metrizable groups
$$ 1\rightarrow X_f\rightarrow X_{f^*f}\rightarrow
X_J\rightarrow 1,$$
where the homomorphism $ X_{f^*f}\rightarrow X_J$ is given by left multiplication by $f$.
\end{lemma}
\begin{proof}
The dual sequence of the above one is the following
\begin{align} \label{E-exact}
0\leftarrow (\Zb\Gamma)^{1\times d}/(\Zb\Gamma)^{1\times d'}f \leftarrow (\Zb\Gamma)^{1\times d}/(\Zb\Gamma)^{1\times d}
f^*f\leftarrow (\Zb\Gamma)^{1\times d'}/\cM_J\leftarrow
0,
\end{align}
where the homomorphism $(\Zb\Gamma)^{1\times d}/(\Zb\Gamma)^{1\times d} f^*f \leftarrow (\Zb\Gamma)^{1\times d'}/\cM_J$ is given by right multiplication by $f$.
By Pontryagin duality it suffices to show that
\eqref{E-exact} is exact. Clearly it is exact at
$(\Zb\Gamma)^{1\times d}/(\Zb\Gamma)^{1\times d'}f$ and $(\Zb\Gamma)^{1\times d}/(\Zb\Gamma)^{1\times d}
f^*f$.
Suppose that
$x\in (\Zb\Gamma)^{1\times d'}/\cM_J$ and $xf=0$ in
$(\Zb\Gamma)^{1\times d}/(\Zb\Gamma)^{1\times d}
f^*f$. Say, $x$ is represented by $\tilde{x}$
in $(\Zb\Gamma)^{1\times d'}$. Then $\tilde{x}f=\tilde{z}f^*f$
for some $\tilde{z}\in (\Zb\Gamma)^{1\times d}$. It follows easily that $\tilde{x}$ lies in $\cM_J$.
Consequently, $x=0$ and hence \eqref{E-exact} is also exact at
$(\Zb\Gamma)^{1\times d'}/\cM_J$.
\end{proof}

The following result is well known. For the convenience of the reader, we give a proof.

\begin{proposition} \label{P-zero divisor}
Let $f\in M_d(\Zb\Gamma)$. Then the following are equivalent:
\begin{enumerate}
\item $\ker f\neq \{0\}$;


\item $\ker f^*\neq \{0\}$;

\item $fg=0$ for some nonzero $g\in M_d(\Zb\Gamma)$.
\end{enumerate}
\end{proposition}
\begin{proof} (1)$\Rightarrow$(2): for any $T\in B(\ell^2(\Gamma)^{d\times 1})$ one has the {\it polar decomposition} as follows: there exist unique $U, S\in B((\ell^2(\Gamma))^{d\times 1})$ satisfying that $S\ge 0$, $\ker U=\ker S=\ker T$,  $U$ is an isometry from the orthogonal complement of $\ker T$ onto the closure of $\im T$,
and $T=US$ \cite[Theorem 6.1.2]{KR2}. Take $T=f\in M_d(\cN\Gamma)$. Since $M_d(\cN\Gamma)$ is the subalgebra of $B((\ell^2(\Gamma))^{d\times 1})$ consisting of elements commuting with the right representation of $\Gamma$, from the uniqueness of the polar decomposition we have $U, S\in M_d(\cN\Gamma)$. Denote by $P$ and $Q$ the orthogonal projections from $(\ell^2(\Gamma))^{d\times 1}$ onto $\ker f$ and $\ker f^*$ respectively. Then both $P=1-U^*U$ and $Q=1-UU^*$ are in $M_d(\cN\Gamma)$. Since $\tr_{\cN\Gamma}$ is faithful and $P\neq 0$,
we have $\tr_{\cN\Gamma}P=\tr_{\cN\Gamma}(P^*P)>0$. Then
$$\tr_{\cN\Gamma}Q=\tr_{\cN\Gamma}(1-UU^*)=\tr_{\cN\Gamma}(1-U^*U)=\tr_{\cN\Gamma}P>0.$$
Thus $Q\neq 0$, which means that $\ker f^*\neq \{0\}$.

(2)$\Rightarrow$(1) follows from (1)$\Rightarrow$(2) by symmetry.

(1)$\Rightarrow$(3): by Lemma~\ref{L-Elek} we have $fx=0$ for some nonzero $x\in (\Cb\Gamma)^{d\times 1}$. Taking the real or imaginary part of
$x$, we may assume that $x\in (\Rb\Gamma)^{d\times 1}$. By \cite[Theorem 4.11]{CL} we may furthermore assume that $x\in (\Zb \Gamma)^{d\times 1}$. Take $g\in M_d(\Zb\Gamma)$ to be the square matrix with every column being $x$. Then $fg=0$.

(3)$\Rightarrow$(1) is trivial.
\end{proof}

We need the Yuzvinski\u{\i} addition formula:

\begin{lemma} [Corollary 6.3 in \cite{Li}]  \label{L-addition}
For any
$\Gamma$-equivariant exact sequence of compact metrizable groups
$$1\rightarrow Y_1\rightarrow Y_2\rightarrow
Y_3\rightarrow 1,$$
one has
$$\rh(Y_2)=\rh(Y_1)+\rh(Y_3).$$
\end{lemma}

We are ready to prove Theorem~\ref{T-main}.

\begin{proof}[Proof of Theorem~\ref{T-main}]
Note that $\ker (f^*f)=\ker f=\{0\}$, and $f^*f$ is positive in $M_d(\cN\Gamma)$.
From Lemma~\ref{L-positive} we have
$$ \frac{1}{2}\rh(X_{f^*f})=\frac{1}{2}\log \ddet_{\cN\Gamma}(f^*f)\overset{\eqref{E-determinant square root}}=\log \ddet_{\cN\Gamma}f.$$

Denote by $J$ the subset of $(\Zb\Gamma)^{1\times d'}$ consisting of all rows of $f^*$ and $g\in (\Zb\Gamma)^{1\times d'}$ satisfying
$gf=0$. From Lemma~\ref{L-addition} and  the short exact sequence in Lemma~\ref{L-exact} we get
\begin{align} \label{E-addition}
 \rh(X_{f^*f})=\rh(X_J)+\rh(X_f),
\end{align}
where $X_J$ is defined by \eqref{E-X_J}.
From Lemma~\ref{L-f < J} we have
$$ \rh(X_f)\le \rh(X_J).$$
Therefore
$$\rh(X_f)\le \frac{1}{2}\rh(X_{f^*f})=\log \ddet_{\cN\Gamma}f.$$

Now assume that $d'=d$. From Lemma~\ref{L-f < J} we have $ \rh(X_f)\le \rh(X_{f^*})$.
By Proposition~\ref{P-zero divisor} we have $\ker f^*=\{0\}$. Thus we also have $\rh(X_{f^*})\le \rh(X_f)$. Therefore $\rh(X_f)=\rh(X_{f^*})$.
Since $\ker f^*=\{0\}$, $J$ consists of all rows of $f^*$ and the zero element of $(\Zb\Gamma)^{1\times d'}$.
Thus from \eqref{E-Xf} we have $X_J=X_{f^*}$.  By \eqref{E-addition} we get
$$ \rh(X_f)=\frac{1}{2}\rh(X_{f^*f})=\log \ddet_{\cN\Gamma}f.$$
\end{proof}

\section{Entropy and $L^2$-torsion} \label{S-entropy-torsion}

Throughout this section $\Gamma$ will be a countable discrete amenable group.

\subsection{Proof of Theorem~\ref{T-torsion}} \label{SS-entropy-torsion}

In this subsection we prove Theorem~\ref{T-torsion}. Throughout this subsection we let $\cM$ be a left $\Zb\Gamma$-module of type FL${}_k$ for some $k\in \Nb$ with  a partial resolution $\cC_*\rightarrow \cM$  by finitely generated free left $\Zb\Gamma$-modules as in \eqref{E-resolution 0}. We choose an ordered basis for each  $\cC_j$, identify $\cC_j$ with $(\Zb\Gamma)^{1\times d_j}$, and take $f_j\in M_{d_j\times d_{j-1}}(\Zb\Gamma)$ so that  $\partial_j(y)=yf_j$ for all $y\in (\Zb\Gamma)^{1\times d_j}$.

We show first that the conditions $\dim_{\cN\Gamma}(\cN\Gamma\otimes_{\Zb\Gamma}\cM)=0$ and $\chi(\cM)=0$ in Theorem~\ref{T-torsion} are equivalent
to $\ker f_1=\{0\}$. Indeed the latter condition is the one we really use in the proof of Theorem~\ref{T-torsion}. We choose to use $\dim_{\cN\Gamma}(\cN\Gamma\otimes_{\Zb\Gamma}\cM)$ and $\chi(\cM)$ in the statement of Theorem~\ref{T-torsion} because they are well-known intrinsic invariants of $\cM$.

\begin{lemma} \label{L-Euler characteristic}
The following are equivalent:
\begin{enumerate}
\item $\dim_{\cN\Gamma}(\cN\Gamma\otimes_{\Zb\Gamma}\cM)=0$,

\item $\ker f_1=\{0\}$.
\end{enumerate}
If furthermore $\cM$ is  of type FL and $\cC_*\rightarrow \cM$ is a resolution  as in
\eqref{E-resolution 2}, then $\dim_{\cN\Gamma}(\cN\Gamma\otimes_{\Zb\Gamma}\cM)=\chi(\cM)$, and in particular the conditions (1) and (2) are also equivalent to
\begin{enumerate}
\item[(3)] $\chi(\cM)=0$.
\end{enumerate}
\end{lemma}
\begin{proof}
Let $g\in M_{d_1\times d_0}(\cN\Gamma)$.  An argument similar to that in Section~\ref{SS-determinant} shows that the orthogonal projection $p_g$ from $(\ell^2(\Gamma))^{d_0\times 1}$ onto the closure of $(\ell^2(\Gamma))^{1\times d_1}g$ lies in $M_{d_0}(\cN\Gamma)$. For each left $\cN\Gamma$-module
$\tilde{\cM}$, denote by ${\bf T}\tilde{\cM}$ the submodule of $\tilde{\cM}$ consisting of elements having image $0$ under every $\cN\Gamma$-module homomorphism $\tilde{\cM}\rightarrow \cN\Gamma$, and denote by ${\bf P}\tilde{\cM}$ the quotient module $\tilde{\cM}/{\bf T}\tilde{\cM}$.
When $\tilde{\cM}$ is finitely generated, one has $\dim_{\cN\Gamma}\tilde{\cM}=\dim_{\cN\Gamma}({\bf P}\tilde{\cM})$ \cite[Theorem 6.7]{Luck}.
Also, from \cite[Lemma 6.52]{Luck} one has
$${\bf P}((\cN\Gamma)^{1\times d_0}/(\cN\Gamma)^{1\times d_1}g)=(\cN\Gamma)^{1\times d_0}(I_{d_0}-p_g),$$
where $I_{d_0}$ denotes the $d_0\times d_0$ identity matrix.
Thus
$$ \dim_{\cN\Gamma}((\cN\Gamma)^{1\times d_0}/(\cN\Gamma)^{1\times d_1}g)=\dim_{\cN\Gamma}((\cN\Gamma)^{1\times d_0}(I_{d_0}-p_g))=\tr_{\cN\Gamma}(I_{d_0}-p_g).$$

For any unital ring $R$, any right $R$-module $\cM^\natural$, and any short exact sequence
$$0\rightarrow \cM_1\rightarrow \cM_2\rightarrow \cM_3\rightarrow 0$$
of left $R$-modules, the sequence
$$\cM^\natural\otimes_R\cM_1\rightarrow \cM^\natural\otimes_R\cM_2 \rightarrow \cM^\natural\otimes_R\cM_3\rightarrow 0$$
is exact \cite[Proposition 19.13]{AF}.
From the exact sequence \eqref{E-resolution 0}, taking $R=\Zb\Gamma$, $\cM^\natural=\cN\Gamma$, $\cM_1=\partial_1(\cC_1)$, $\cM_2=\cC_0$ and $\cM_3=\cM$,  we find that
$$ \cN\Gamma\otimes_{\Zb\Gamma}\cC_1\overset{1\otimes \partial_1}\rightarrow \cN\Gamma\otimes_{\Zb\Gamma}\cC_0\rightarrow \cN\Gamma\otimes_{\Zb\Gamma}\cM\rightarrow 0$$
is exact.
Taking $g=f_1$, we get
$$\dim_{\cN\Gamma}(\cN\Gamma\otimes_{\Zb\Gamma}\cM)=\dim_{\cN\Gamma}((\cN\Gamma)^{1\times d_0}/(\cN\Gamma)^{1\times d_1}f_1)=\tr_{\cN\Gamma}(I_{d_0}-p_{f_1}).$$
Thus $\dim_{\cN\Gamma}(\cN\Gamma\otimes_{\Zb\Gamma}\cM)=0$ if and only if $\tr_{\cN\Gamma}(I_{d_0}-p_{f_1})=0$, equivalently $I_{d_0}=p_{f_1}$, i.e.
$(\ell^2(\Gamma))^{1\times d_1}f_1$ is dense in $(\ell^2(\Gamma))^{1\times d_0}$. The latter condition is equivalent to that the map $(\ell^2(\Gamma))^{1\times d_0}\rightarrow (\ell^2(\Gamma))^{1\times d_1}$ sending $y$ to  $yf_1^*$ is  injective. Taking adjoints we find that the last condition is equivalent to $\ker f_1=\{0\}$. This proves (1)$\Leftrightarrow$(2).

Now we assume further that $\cM$ is  of type FL and $\cC_*\rightarrow \cM$ is a resolution  as in
\eqref{E-resolution 2}. Note that
$ \Cb\otimes_\Zb\Zb\Gamma=\Cb\Gamma$
and hence for any left $\Zb\Gamma$-module $\tilde{\cM}$ one has
$$\Cb\Gamma\otimes_{\Zb\Gamma}\tilde{\cM}=(\Cb\otimes_\Zb\Zb\Gamma)\otimes_{\Zb\Gamma}\tilde{\cM}=\Cb\otimes_\Zb(\Zb\Gamma\otimes_{\Zb\Gamma}\tilde{\cM})=\Cb\otimes_\Zb\tilde{\cM}.$$
Since $\Cb$ is a torsion-free $\Zb$-module, the functor $\Cb\otimes_\Zb?$ from the category of $\Zb$-modules to the category of $\Cb$-modules is exact \cite[Proposition XVI.3.2]{Lang}. Thus the functor $\Cb\Gamma\otimes_{\Zb\Gamma}?$ from the category of left $\Zb\Gamma$-modules to the category of left $\Cb\Gamma$-modules is exact.
Set $\cC'_j=\Cb\Gamma\otimes_{\Zb\Gamma}\cC_j$ and $\cM'=\Cb\Gamma\otimes_{\Zb\Gamma}\cM$. Then from \eqref{E-resolution 2} we have the exact sequence
\begin{align} \label{E-ker dim}
0 \to \cC'_k \overset{\partial'_k}\rightarrow \cdots \overset{\partial'_2}\rightarrow \cC'_1 \overset{\partial'_1}\rightarrow \cC'_0 \rightarrow \cM'\rightarrow 0.
\end{align}
The sequence
$$ 0\rightarrow \cN\Gamma\otimes_{\Cb\Gamma}\cC'_k\overset{1\otimes \partial'_k}\rightarrow \cdots \overset{1\otimes \partial'_2}\rightarrow \cN\Gamma\otimes_{\Cb\Gamma}\cC'_1\overset{1\otimes \partial'_1}\rightarrow \cN\Gamma\otimes_{\Cb\Gamma}\cC'_0\rightarrow \cN\Gamma\otimes_{\Cb\Gamma}\cM'\rightarrow 0$$
is exact at $\cN\Gamma\otimes_{\Cb\Gamma}\cC'_0$ and $\cN\Gamma\otimes_{\Cb\Gamma}\cM'$, but may fail to be exact at other places.
Note that \eqref{E-ker dim} is a resolution of $\cM'$ by free left $\Cb\Gamma$-modules.
L\"{u}ck showed \cite[Theorem 6.37]{Luck} that
$$ \dim_{\cN\Gamma}(\ker(1\otimes \partial'_j)/\im(1\otimes \partial'_{j+1}))=0$$
for all $1\le j\le k$, where we set $1\otimes \partial'_{k+1}=0$. Since $\dim_{\cN\Gamma}$ is additive in the sense that for any short exact sequence
$$0\rightarrow \cM_1\rightarrow \cM_2\rightarrow \cM_3\rightarrow 0$$
of left $\cN\Gamma$-modules one has $\dim_{\cN\Gamma}\cM_2=\dim_{\cN\Gamma}\cM_1+\dim_{\cN\Gamma}\cM_3$ \cite[Theorem 6.7]{Luck}, we get
\begin{align*}
 \dim_{\cN\Gamma}(\cN\Gamma\otimes_{\Zb\Gamma}\cM)&=\dim_{\cN\Gamma}(\cN\Gamma\otimes_{\Cb\Gamma}\cM')=
\sum_{j=0}^k(-1)^j\dim_{\cN\Gamma}(\cN\Gamma\otimes_{\Cb\Gamma}\cC'_j)\\
&=\sum_{j=0}^k(-1)^jd_j=\chi(\cM).
\end{align*}
\end{proof}

\begin{remark} \label{R-vanishdim}
It follows from Lemma~\ref{L-Euler characteristic} and \cite[Theorem 4.11]{CL} that for a finitely presented left $\Zb \Gamma$-module $\cM$, $\rh(\widehat \cM)$ is finite if and only if $\dim_{\cN \Gamma} (\cN \Gamma \otimes_{\Zb \Gamma} \cM) = 0$. In particular, if $\cM$ is of type FL, then $\rh(\widehat \cM)$ is finite if and only if $\chi(\cM)=0$.
\end{remark}

Next we show that (in the case of amenable groups) $\ker f_1=\{0\}$ is the only condition needed to define $\rho^{(2)}(\cC_*)$.
Set $f_0=0$.

\begin{lemma} \label{L-resolution to acyclic}
Suppose that $\ker f_1=\{0\}$. Then the chain complex $\ell^2(\Gamma)\otimes_{\Zb\Gamma}\cC_*$ in \eqref{E-resolution 3} is weakly acyclic, and $\ker (f_{j+1}^*f_{j+1}+f_jf_j^*)=\{0\}$ for all $0\le j<k$.
\end{lemma}
\begin{proof} We show first that for any $1\le j<k$ the map $(\Zb\Gamma)^{1\times d_j}\rightarrow (\Zb\Gamma)^{1\times d_j}$ sending $y$ to  $y(f_{j+1}^*f_{j+1}+f_jf_j^*)$ is injective. Suppose that $y(f_{j+1}^*f_{j+1}+f_jf_j^*)=0$. Computing $\left<y(f_{j+1}^*f_{j+1}+f_jf_j^*), y\right>$, we find
that $yf_{j+1}^*=0$ and $yf_j=0$. Since \eqref{E-resolution 0} is exact at $\cC_j=(\Zb\Gamma)^{1\times d_j}$, we have $y=zf_{j+1}$ for some $z\in (\Zb\Gamma)^{1\times d_{j+1}}$. From $\left<zf_{j+1}, zf_{j+1}\right>=\left<zf_{j+1}f_{j+1}^*, z\right>=0$ we get $y=zf_{j+1}=0$. This proves our claim.

Since $f_{j+1}^*f_{j+1}+f_jf_j^*$ is self-adjoint, taking adjoints we find that for each $1\le j<k$ the map $(\Zb\Gamma)^{d_j\times 1}\rightarrow (\Zb\Gamma)^{d_j\times 1}$ sending $y$ to  $(f_{j+1}^*f_{j+1}+f_jf_j^*)y$ is also injective. From Proposition~\ref{P-zero divisor} we
conclude that $\ker(f_{j+1}^*f_{j+1}+f_jf_j^*)=\{0\}$ for $1\le j<k$. The assertion $\ker(f_1^*f_1+f_0f_0^*)=\{0\}$ follows directly from $\ker f_1=\{0\}$.

Let $0\le j<k$. Taking adjoints again, we find that the map $(\ell^2(\Gamma))^{d_j\times 1}\rightarrow (\ell^2(\Gamma))^{d_j\times 1}$ sending $y$ to  $y(f_{j+1}^*f_{j+1}+f_jf_j^*)$ is injective. Any $y$ in the orthogonal complement of the closure of $\im(1\otimes \partial_{j+1})$ inside of $\ker(1\otimes \partial_j)$ satisfies $y(f_{j+1}^*f_{j+1}+f_jf_j^*)=0$ and hence is equal to $0$. Therefore
$\ker(1\otimes \partial_j)$ is equal to the closure of $\im(1\otimes \partial_{j+1})$. That is, $\ell^2(\Gamma)\otimes_{\Zb\Gamma}\cC_*$ is weakly acyclic.
\end{proof}

In order to prove Theorem~\ref{T-torsion} we need some preparation.
For each $0<j\le k$, we define a left $\Zb\Gamma$-module homomorphism
$\partial_j^*: \cC_{j-1}\rightarrow \cC_j$ by $\partial_j^*(y)=yf_j^*$ for all $y\in (\Zb\Gamma)^{1\times d_{j-1}}=\cC_{j-1}$.
Then $\partial_{j+1}^*\partial_j^*=0$ for all $0<j<k$.
For each $0\le j\le k$, set $\cD_j=\bigoplus_{0\le i\le j, j-i\in 2\Zb}\cC_i$ and $d_j'=\sum_{0\le i\le j, j-i\in 2\Zb}d_i$. For each $0<j\le k$,  consider the $\Zb\Gamma$-module homomorphism $T_j:\cD_j\rightarrow \cD_{j-1}$ defined as $\partial_i+\partial_{i+1}^*$ on $\cC_i$ for all $0<i<j$ with $j-i\in 2\Zb$,  as $\partial_j$ on $\cC_j$,
and also as $\partial_1^*$ on $\cC_0$ if $j$ is even.
Then the chosen ordered basis of $\cC_i$'s give rise to an ordered basis for $\cD_j$, under which $T_j$ is represented by the matrix
\begin{align*}
g_j=\left(
\begin{matrix}
f_j& 0 & 0 &\cdots\\
f_{j-1}^* & f_{j-2} & 0 &\cdots \\
0 & f_{j-3}^* & f_{j-4}& \cdots \\
\cdots &\cdots &\cdots &\cdots
\end{matrix}
\right)\in M_{d'_j\times d'_{j-1}}(\Zb\Gamma).
\end{align*}

\begin{lemma} \label{L-short exact}
Let $0<j<k$. Then one has a $\Gamma$-equivariant short exact sequence of compact metrizable groups
\begin{align} \label{E-short exact}
 1\rightarrow X_{g_j}\rightarrow X_{g_j^*g_j}\rightarrow
X_{g_{j+1}}\rightarrow 1,
\end{align}
where the homomorphism $ X_{g_j^*g_j}\rightarrow X_{g_{j+1}}$ is given by left multiplication by $g_j$.
\end{lemma}
\begin{proof}
Denote by $J_j$ the subset of $(\Zb\Gamma)^{1\times d'_j}$ consisting of all rows of $g_j^*$ and $h\in (\Zb\Gamma)^{1\times d'_j}$ satisfying
$hg_j=0$. Recall that $\cM_{J_j}$ denotes the submodule of $(\Zb \Gamma)^{1 \times d'_j}$ generated by  $J_j$. We want to show that
$\cM_{J_j} = (\Zb\Gamma)^{1\times d'_{j+1}}g_{j+1}$.
Set
\begin{align*}
 w=\left(
\begin{matrix}
f_{j+1} & 0
\end{matrix}
\right)\in M_{d_{j+1}\times d'_j}(\Zb\Gamma).
\end{align*}
Then
\begin{align*}
 g_{j+1}=\left(
\begin{matrix}
w\\
g_j^*
\end{matrix}
\right),
\end{align*}
and $wg_j=0$. It follows that $(\Zb\Gamma)^{1\times d'_{j+1}}g_{j+1}\subseteq \cM_{J_j}$. Moreover, any row of $g_j^*$ is obviously in $(\Zb \Gamma)^{1 \times d'_{j+1}}g_{j+1}$. It remains to show that every $h\in (\Zb\Gamma)^{1\times d'_j}$ satisfying $hg_j=0$ lies in $(\Zb \Gamma)^{1 \times d'_{j+1}}g_{j+1}$.

Note that $g_jg_j^*$ is a block-diagonal matrix with the diagonal blocks being $f_jf_j^*$ and $f_{i+1}^*f_{i+1}+f_if_i^*$ for $0\le i\le j-2$ with $j-i\in 2\Zb$.
Since \eqref{E-resolution 0} is exact and $\ker f_1=0$, for any $0\le i<j$, any $y\in (\Zb\Gamma)^{1\times d'_j}$ satisfying $y(f_{i+1}^*f_{i+1}+f_if_i^*)$ must be $0$.
Thus any $h\in (\Zb\Gamma)^{1\times d'_j}$ satisfying $hg_j=0$ must be of the form $(x, 0)$ for some
$x\in (\Zb\Gamma)^{1\times d_j}$ satisfying $xf_jf_j^*=0$, equivalently $xf_j=0$.
It follows from the exactness of \eqref{E-resolution 0} that $(x,0) = (yf_{j+1},0)$ for some
$y\in (\Zb\Gamma)^{1\times d_{j+1}}$. Thus, $\cM_{J_j}\subseteq (\Zb\Gamma)^{1\times d'_{j+1}}g_{j+1}$,
and hence $\cM_{J_j}=(\Zb\Gamma)^{1\times d'_{j+1}}g_{j+1}$, which leads to the short exact sequence \eqref{E-short exact} by
Lemma~\ref{L-exact}.
\end{proof}

\begin{lemma} \label{L-addition for torsion}
For any $0<j\le k$ one has $\ker g_j=0$. For $0<j<k$ one has
\begin{align*}
\rh(X_{g_j})+\rh(X_{g_{j+1}})=\sum_{i=1}^j\log \ddet_{\cN\Gamma} (f_i^*f_i + q_{f_i})\in \Rb_{\ge 0},
\end{align*}
where $q_{f_i}$ denotes the orthogonal projection from $(\ell^2(\Gamma))^{d_{j-1}\times 1}$ onto $\ker f_i$.
We also have
\begin{align*}
\rh(X_{g_k^*g_k})=2\log\ddet_{\cN\Gamma}(g_k)=\sum_{i=1}^k\log \ddet_{\cN\Gamma} (f_i^*f_i + q_{f_i})\in \Rb_{\ge 0}.
\end{align*}
\end{lemma}
\begin{proof}
Let $0<j\le k$. Note that $g_j^*g_j$ is a block-diagonal matrix with the diagonal  blocks being $f_{i+1}^*f_{i+1}+f_if_i^*$ for $0\le i<j$ with $j-i\not\in 2\Zb$.
From Lemma~\ref{L-resolution to acyclic} we get $\ker g_j=\ker (g_j^*g_j)=\{0\}$.
For any $0\le i<j$, since $f_{i+1}^*f_{i+1}$ and $f_if_i^*$ are self-adjoint, $(\ell^2(\Gamma))^{d_i\times 1}$ is both the orthogonal direct sum of $\ker(f_{i+1}^*f_{i+1})$ and $\overline{\im(f_{i+1}^*f_{i+1})}$, and  the orthogonal direct sum of $\ker(f_if_i^*)$ and $\overline{\im(f_if_i^*)}$.
Because
$f_{i+1}^*f_{i+1}\cdot f_if_i^*=f_if_i^*\cdot f_{i+1}^*f_{i+1}=0$, and
$\ker (f_{i+1}^*f_{i+1}+f_if_i^*)=\{0\}$, we have $\ker(f_{i+1}^*f_{i+1})=\overline{\im(f_if_i^*)}$ and $\ker(f_if_i^*)=\overline{\im(f_{i+1}^*f_{i+1})}$.
It follows that
$$f_{i+1}^*f_{i+1}+f_if_i^*=(f_{i+1}^*f_{i+1}+q_{f_{i+1}^*f_{i+1}})(f_if_i^*+q_{f_if_i^*})=(f_{i+1}^*f_{i+1}+q_{f_{i+1}})(f_if_i^*+q_{f_i^*}).$$
From Lemma~\ref{L-positive} we have
\begin{eqnarray*}
 \rh(X_{g_j^*g_j})&=&\log \ddet_{\cN\Gamma}(g_j^*g_j)\\
&=&\sum_{0\le i<j, j-i\not\in 2\Zb}\log \ddet_{\cN\Gamma} (f_{i+1}^*f_{i+1}+f_if_i^*)\\
&=&\sum_{0\le i<j, j-i\not\in 2\Zb}\big(\log \ddet_{\cN\Gamma} (f_{i+1}^*f_{i+1}+q_{f_{i+1}})+\log \ddet_{\cN\Gamma} (f_if_i^*+q_{f_i^*})\big)\\
&\overset{\eqref{E-symmetry}}=&\sum_{i=1}^j\log \ddet_{\cN\Gamma} (f_i^*f_i + q_{f_i}).
\end{eqnarray*}

Now the lemma follows from \eqref{E-determinant square root} and the observation that for $0<j<k$ from Lemmas~\ref{L-short exact} and \ref{L-addition} we have

$$\rh(X_{g_j})+\rh(X_{g_{j+1}})=\rh(X_{g_j^*g_j}).$$
\end{proof}

We are ready to prove Theorem~\ref{T-torsion}.

\begin{proof}[Proof of Theorem~\ref{T-torsion}] From Lemma~\ref{L-addition for torsion} we have
\begin{align*}
(-1)^k\rh(X_{g_1})+\rh(X_{g_k})
&=\sum_{j=1}^{k-1}(-1)^{k+1+j}(\rh(X_{g_j})+\rh(X_{g_{j+1}}))\\
&=\sum_{j=1}^{k-1}(-1)^{k+1+j}\sum_{i=1}^j\log \ddet_{\cN\Gamma} (f_i^*f_i + q_{f_i}).
\end{align*}
From Lemmas~\ref{L-addition for torsion} and Theorem~\ref{T-main} we have
\begin{align*}
\rh(X_{g_k})\le \frac12 \sum_{i=1}^k\log \ddet_{\cN\Gamma} (f_i^* f_i + q_{f_i}).
\end{align*}
Note that $g_1=f_1$. Therefore
\begin{eqnarray*}
(-1)^k\rh(\widehat{\cM})&=&(-1)^k\rh(X_{f_1})\\
&=&(-1)^k\rh(X_{g_1})\\
&\ge& \sum_{j=1}^{k-1}(-1)^{k+1+j}\sum_{i=1}^j\log \ddet_{\cN\Gamma} (f_i^*f_i +q_{f_i})-\frac12 \sum_{i=1}^k\log \ddet_{\cN\Gamma} (f_i^*f_i + q_{f_i}) \\
&=&  \frac{(-1)^k}2 \sum_{i=1}^k(-1)^{i+1}\log \ddet_{\cN\Gamma} (f_i^* f_i + q_{f_i})\\
& =&(-1)^k \rho^{(2)}(\cC_*).
\end{eqnarray*}

When $\cC_*\rightarrow \cM$ is a resolution as in \eqref{E-resolution 2}, we may also think of it as a partial resolution with length $k+1$ by setting $\cC_{k+1}=0$. Then we also have $(-1)^{k+1}\rh(\widehat{\cM})\ge (-1)^{k+1} \rho^{(2)}(\cC_*)$, and hence $\rh(\widehat{\cM})=\rho^{(2)}(\cC_*)$.
\end{proof}

\subsection{Applications to $L^2$-torsion}  \label{SS-application to torsion}

The first application of Theorem~\ref{T-torsion} to $L^2$-torsion is a proof of Theorem~\ref{T-Luck conjecture}.

If the trivial left $\Zb\Gamma$-module $\Zb$ is of type FL, then $\Gamma$ is torsion-free \cite[Corollary VIII.2.5]{Brown}, and in particular $\Gamma$ can not be a non-trivial finite group. The latter fact can also be proved using the following quick argument we learned from a comment of Ian Agol:
Suppose that $\Gamma$ is finite and there exists a resolution
$$0 \to (\Zb \Gamma)^{1 \times d_k} \to \cdots \to (\Zb \Gamma)^{1 \times d_0} \to \Zb \to 0.$$
Counting ranks of $\Zb$-modules we get $1=|\Gamma| \cdot \sum_{j=0}^k (-1)^j d_j$, and hence $\Gamma$ is trivial.

From the definition of topological entropy one observes easily that the trivial action of an infinite amenable group on a compact metrizable space has topological entropy $0$.
Now Theorem~\ref{T-Luck conjecture} is an immediate consequence of Theorem~\ref{T-torsion} and Remark~\ref{R-vanishdim}.

As a second application, since the entropy of an action is non-negative, we note from Theorem~\ref{T-torsion} that if a left $\Zb\Gamma$-module $\cM$ is of type FL and $\chi(\cM)=0$, then $\rho^{(2)}(\cM)\ge 0$. This is non-trivial since the $L^2$-torsion is defined as an alternating sum of non-negative numbers.

Let us mention one more application.

\begin{theorem} \label{T-infinite index}
Let $\Gamma$ be a countable discrete amenable group which contains $\Zb$ as a subgroup of infinite index. Let $\cM$ be a left $\Zb\Gamma$-module. If $\cM$ is finitely generated as an abelian group and $\cM$ is of type FL as a left $\Zb\Gamma$-module, then $\chi(\cM)=0$ and $\rho^{(2)}(\cM)=0.$
\end{theorem}

To prove Theorem~\ref{T-infinite index} we need the following well-known dynamical fact. For the convenience of the reader, we give a proof.
We allow smooth manifolds to have different dimensions for different connected components, including $0$ dimension. In particular, compact smooth manifolds could be finite sets.
We say that an action of $\Gamma$ on a compact smooth manifold $X$ is {\it differentiable} if the homeomorphism of $X$ given by each $s\in \Gamma$ is $C^{(1)}$.
\begin{lemma} \label{L-smooth to zero}
Let $\Gamma$ be a countable discrete amenable group containing $\Zb$ as a subgroup of infinite index.
Then every differentiable action of $\Gamma$ on a compact smooth manifold $X$ has  topological entropy $0$.
\end{lemma}
\begin{proof} Endow $X$ with a Riemannian metric. Since $X$ is a compact manifold, it has finitely many connected components. Thus we may take a compatible
metric $\theta$ on $X$ which restricts to the geodesic distance on each connected component. Take $0<\eta<1$ such that if $x, y\in X$ have distance less that
$\eta$, then $x$ and $y$ are in the same connected component of $X$. Denote by $L$ the diameter of $X$. Recall that a subset $Z$ of $X$ is called $\delta$-dense
for $\delta>0$ if for any $x\in X$ one has $\theta(x, z)\le \delta$ for some $z\in Z$.

For each $p\in \Nb\cup \{0\}$, consider the supremum norm $\|\cdot \|_\infty$ on $\Rb^p$ given by $\|(u_1, \dots, u_p)\|_\infty=\max_{1\le j\le p}|u_j|$.
For $r>0$, denote by $B(0, r)$ the open ball in $\Rb^p$ with center $0$ and radius $r$ in this norm. For a connected component $Y$ of $X$ with dimension $p$,
choose smooth charts $f_j: B(0, 2)\rightarrow Y$ for $1\le j\le k_Y$ such that $\bigcup_{1\le j\le k_Y}f_j(B(0, 1))=Y$. Take $K_Y>0$ such that
$\theta(f_j(x), f_j(y))\le K_Y\|x-y\|_\infty$ for all $1\le j\le k_Y$ and $x, y\in B(0, 1)$.
For any $0<\delta'<1$,  $B(0, 1)$ has a $\delta'$-dense subset with
cardinality at most $(4/\delta')^p$, and hence $Y$ has a $(K_Y\delta')$-dense subset with cardinality at most $k_Y(4/\delta')^p$.
Denoting by $q$ the highest
dimension of the connected components of $X$, we see that there exist constants $C_1, C_2>1$ such that for any $0<\delta'<1$,
$X$ has a $(C_1\delta')$-dense
subset with cardinality at most $C_2(4/\delta')^q$.

For each $s\in \Gamma$ and $x\in X$, denote by $\xi_{s, x}$ the linear map induced by $s$ from the tangent space of $X$ at $x$ to the tangent space of $X$ at $sx$. Denote by $\|\xi_{s, x}\|$ the operator norm of $\xi_{s, x}$. Since $X$ is compact, one has $K_s:=\sup_{x\in X}\|\xi_{s, x}\|<+\infty$.
Then for any $x$ and $y$ in the same connected component of $X$, one has $\theta(sx, sy)\le K_s\theta(x, y)$.

Let $0<\varepsilon<\eta/2$.
By assumption, we can find $s_0\in \Gamma$ with infinite order such that the subgroup $H$ of $\Gamma$ generated by $s_0$ has infinite index.
Let $M, N\in \Nb$.
Take $t_1=e, \dots, t_N$ in $\Gamma$ such that the left cosets $t_1H, \dots, t_NH$ are pairwise distinct. Set  $P=\{t_js_0^k: 1\le j\le N, 0\le k\le M-1\}$.

Let $F\in \cB(P, 1/|P|)$, i.e. $F$ is a nonempty finite subset of $\Gamma$ and
$$|\{t\in F: Pt\subseteq F\}|\ge \left(1-\frac{1}{|P|}\right)|F|.$$
Then
$$ |PF|\le |F|+|P|\cdot |\{t\in F: Pt\nsubseteq F\}|\le 2|F|.$$
Take a maximal subset $\Omega$ of $F$ subject to the condition that for any distinct $s, t\in \Omega$,
the sets $Ps$ and $Pt$ are disjoint. For any $s\in F$, one has $Ps\cap Pt\neq \emptyset$ for some $t\in \Omega$ and hence $s\in (P^{-1}P)t$.
That is, $F\subseteq P^{-1}P\Omega$.  Note that
\begin{align} \label{E-smooth 1}
N\cdot M\cdot |\Omega|=|P|\cdot |\Omega|=|P\Omega|\le |PF|\le 2|F|.
\end{align}

Set $K_1=\max(K_{s_0}, K_{s_0^{-1}})\ge 1$ and $K_2=\max_{1\le j, k\le N}K_{t_j^{-1}t_k}\ge 1$. Let $x, y\in X$. If  $\theta(tx, ty)\le \eta/2$ for some $t\in \Omega$, then $tx$ and $ty$ lie in the same connected component of $X$, and hence
\begin{align*}
\theta(stx, sty)\le K_1^{2M}K_2\theta(tx, ty)
\end{align*}
for all $s\in P^{-1}P$. Thus, if $\theta(tx, ty)\le \min(\eta/2, \varepsilon/(K_1^{2M}K_2))=(K_1^{2M}K_2)^{-1}\varepsilon$ for all $t\in \Omega$, then
$\theta_{F, \infty}(x, y)\le \varepsilon$, where $\theta_{F, \infty}$ is defined by \eqref{E-infinity distance}.
Taking $\delta'=C_1^{-1}(2K_1^{2M}K_2)^{-1}\varepsilon$ in the second paragraph of the proof, we see that $X$ has  a $(2K_1^{2M}K_2)^{-1}\varepsilon$-dense subset $Z$ with cardinality at most $C_2(8K_1^{2M}K_2C_1/\varepsilon)^q$.
For each $x\in X$, there is some $\varphi(x)\in Z^\Omega$ such that $\theta(tx, \varphi(x)_t)\le (2K_1^{2M}K_2)^{-1}\varepsilon$ for all $t\in \Omega$.
If $\cW$ is a $(\theta_{F, \infty}, \varepsilon)$-separated subset of $X$ and $\varphi(x)=\varphi(y)$ for some $x, y\in \cW$, then
$$\theta(tx, ty)\le \theta(tx, \varphi(x)_t)+\theta(\varphi(y)_t, ty)\le (K_1^{2M}K_2)^{-1}\varepsilon$$
for all $t\in \Omega$, and hence $\theta_{F, \infty}(x, y)\le \varepsilon$, which implies that $x=y$.
It follows that
$$ N_\varepsilon(X, \theta_{F, \infty})\le |Z^\Omega|\le (C_2(8K_1^{2M}K_2C_1/\varepsilon)^q)^{|\Omega|}\overset{\eqref{E-smooth 1}}\le (C_2(8K_1^{2M}K_2C_1/\varepsilon)^q)^{2|F|/MN}.$$
Therefore
$$
\limsup_F\frac{\log N_\varepsilon(X, \theta_{F, \infty})}{|F|}\le \frac{2}{MN}(\log (C_2(8C_1/\varepsilon)^q)+2Mq\log K_1+q\log K_2).$$
Fix $N$ and $t_1, \dots, t_N$ first. Then $K_2$ is fixed. Letting $M\to \infty$, we get
$$\limsup_F\frac{\log N_\varepsilon(X, \theta_{F, \infty})}{|F|}\le \frac{4q}{N}\log K_1.$$
Next letting $N\to \infty$, we get
$$\limsup_F\frac{\log N_\varepsilon(X, \theta_{F, \infty})}{|F|}=0.$$
Since $\varepsilon$ is arbitrary, by Lemma~\ref{L-l2 definition of entropy} we conclude that $\htopol(X)=0$ as desired.
\end{proof}

Now Theorem~\ref{T-infinite index} follows from Lemma~\ref{L-smooth to zero}, Theorem~\ref{T-torsion} and Remark~\ref{R-vanishdim}.

\subsection{Application to entropy} \label{SS-application to entropy}

For $\Gamma=\Zb^d$, the fact that
$\Zb[\Zb^d]$ is a factorial Noetherian integral domain allows one to apply the tools from commutative algebra. Using such tools Lind, Schmidt and Ward \cite[Section 4]{LSW} showed
that the calculation of the entropy for algebraic actions of $\Zb^d$ can be reduced to calculating the entropy of principal algebraic actions, i.e.
$$\rh \left(\widehat{\Zb[\Zb^d]/\Zb[\Zb^d]f} \right) =\log \ddet_{\cN(\Zb^d)}f$$ for all non-zero $f\in \Zb[\Zb^d]$.
For general countable discrete amenable groups, it is not clear how to carry out such a reduction.
 However, we shall see that the calculation of the entropy for algebraic actions of poly-$\Zb$ groups can be reduced to calculating the $L^2$-torsion.

A group $\Gamma$ is called {\it poly-$\Zb$}  if there is a sequence of subgroups $\Gamma=\Gamma_1\rhd \Gamma_2\rhd\cdots \rhd \Gamma_n=\{e\}$ such
that $\Gamma_j/\Gamma_{j+1}=\Zb$  for every $1\le j\le n-1$.
Let $\Gamma$ be a poly-$\Zb$ group. Then $\Zb\Gamma$ satisfies the following conditions:
\begin{enumerate}
\item $\Zb\Gamma$ is {\it left Noetherian}, i.e. every left ideal of $\Zb\Gamma$ is finitely generated.

\item $\Zb\Gamma$ has {\it finite global dimension}, i.e. there exists $k\in \Nb$ such that for any left $\Zb\Gamma$-module $\cM$ and any exact sequence of left $\Zb\Gamma$-modules as in \eqref{E-resolution 2} with $\cC_j$ being a projective left $\Zb\Gamma$-module for each $0\le j<k$, $\cC_k$ is also a projective left $\Zb\Gamma$-module.

\item Every finitely generated projective left $\Zb\Gamma$-module $\cM$ is {\it stably free}, i.e. $\cM\oplus (\Zb\Gamma)^m=(\Zb\Gamma)^n$ for some $m,n \in \Nb$.

\end{enumerate}
This was proved in the proof of Theorem 13.4.9 in \cite{Passman} for $K\Gamma$ with $K$ being a field, but the argument there also works for $\Zb\Gamma$, using
that $\Zb$ has finite global dimension \cite[page 433]{Passman}. It follows easily that every finitely generated left $\Zb\Gamma$-module $\cM$ is of type FL.
If $\chi(\cM)>0$, then by Lemma~\ref{L-Euler characteristic} and \cite[Theorem 4.11]{CL} one has $\rh(\widehat{\cM})=\infty$.
If $\chi(\cM)=0$, then by Theorem~\ref{T-torsion} we have $\rh(\widehat{\cM})=\rho^{(2)}(\cM)$. For any countable left $\Zb\Gamma$-module $\cM$,
take an increasing sequence $\{\cM_j\}_{j\in \Nb}$ of finitely generated submodules of $\cM$ with union $\cM$, then from Theorem~\ref{T-Peters} or using the fact that
$\widehat{\cM}$ is the projective limit of $\{\widehat{\cM_j}\}_{j\in \Nb}$ one has $$\rh(\widehat{\cM})=\lim_{j\to \infty}\rh(\widehat{\cM_j})=\sup_{j\in \Nb}\rh(\widehat{\cM_j}).$$

\subsection{Torsion for arbitrary modules and the Milnor-Turaev formula} \label{SS-milnorturaev}

Note that the definition of $L^2$-torsion makes sense only for left $\Zb\Gamma$-modules $\cM$ of type FL. Furthermore, if $\chi(\cM)\neq 0$, then by Lemma~\ref{L-Euler characteristic} any resolution $\cC_*\rightarrow \cM$ of $\cM$ by finitely generated free left $\Zb\Gamma$-modules as in \eqref{E-resolution 2} fails to be weakly acyclic, and as a consequence $\rho^{(2)}(\cC_*)$ defined in \eqref{E-torsion} may depend on the choice of the resolution. Thus the $L^2$-torsion is well-defined only for left $\Zb\Gamma$-modules $\cM$ of type FL with $\chi(\cM)=0$.
Motivated by Theorem~\ref{T-torsion}, we define the \emph{$L^2$-torsion} or just \emph{torsion} of a countable left $\Zb\Gamma$-module $\cM$ to be $\rh(\widehat{\cM})$.
 Theorem~\ref{T-torsion} shows that this extends the definition of $L^2$-torsion for left $\Zb\Gamma$-modules of type FL with $\chi(\cM)=0$. 
 By Remark~\ref{R-vanishdim}, the torsion of $\cM$ is infinite if $\cM$ is of type FL and $\chi(\cM)\neq 0$. Lemma~\ref{L-addition} shows that torsion is additive for extensions of countable left $\Zb\Gamma$-modules and hence serves as a well-behaved invariant. From now on, we use the notation $\rho(\cM)$ to denote $\rh(\widehat{\cM})$.

\begin{proof}[Proof of Theorem~\ref{T-milnorturaev}:]
First of all, the chain complex $\cC_*$ is $\Delta$-acyclic if and only if the chain complex $\ell^2(\Gamma)\otimes_{\Zb\Gamma}\cC_*$ is weakly acyclic, see Proposition~\ref{P-definitions of torsion}.
The proof follows closely the proof of Theorem~\ref{T-torsion}.
We use the notation in Section~\ref{SS-entropy-torsion} and the proof of Lemma~\ref{L-short exact}, and also set $f_{k+1}=0$.

Since the chain complex $\ell^2(\Gamma)\otimes_{\Zb\Gamma}\cC_*$ is weakly acyclic, the map $(\ell^2(\Gamma))^{1\times d_j}\rightarrow (\ell^2(\Gamma))^{1\times d_j}$ sending
$y$ to $y(f_{j+1}^*f_{j+1}+f_jf_j^*)$ is injective for every $0\le j\le k$. Let $0<j<k$. Note that if $zg_j^*g_j=0$ for some $z\in (\Zb\Gamma)^{1\times d'_{j-1}}$, then $zg_j^*=0$. Thus the intersection of $(\Zb\Gamma)^{1\times d'_{j-1}}g_j^*$ and $\{y\in (\Zb\Gamma)^{1\times d_j}: yg_j=0\}$ is $\{0\}$.
In the proof of Lemma~\ref{L-short exact} from the assumption $H_j(\cC_*)=0$ one obtains $\cM_{J_j}=(\Zb\Gamma)^{1\times d'_{j+1}}g_{j+1}$. Now
$H_j(\cC_*)$ could be nonzero, but
the argument in the proof of Lemma~\ref{L-short exact} still shows that   one has a short
exact sequence of left $\Zb\Gamma$-modules
$$ 0\rightarrow (\Zb\Gamma)^{1\times d'_{j+1}}g_{j+1}\rightarrow \cM_{J_j}\rightarrow H_j(\cC_*)\rightarrow 0.$$
Then one has  a short exact sequence of left $\Zb\Gamma$-modules
$$0\rightarrow (\Zb\Gamma)^{1\times d'_j}/\cM_{J_j}\rightarrow (\Zb\Gamma)^{1\times d'_j}/(\Zb\Gamma)^{1\times d'_{j+1}}g_{j+1}\rightarrow H_j(\cC_*)\rightarrow 0,$$
and its dual $\Gamma$-equivariant exact sequence of compact metrizable groups
$$ 1  \rightarrow \widehat{H_j(\cC_*)}  \rightarrow X_{g_{j+1}}   \rightarrow X_{J_j}\rightarrow 1.$$
From Lemma~\ref{L-exact} we still have
the  $\Gamma$-equivariant exact sequence of compact metrizable groups
$$1 \rightarrow X_{g_j} \rightarrow X_{g_j^*g_j} \to X_{J_j} \rightarrow 1.$$
In the proof of Lemma~\ref{L-short exact} from the assumption $H_j(\cC_*)=0$ one obtains
The argument in the proof of Lemma~\ref{L-short exact} also shows that any $y\in (\ell^2(\Gamma))^{1\times d'_k}$ satisfying $yg_k=0$ must be $0$.
Taking adjoints we get $\ker g_k^*=\{0\}$. The argument in the proof of Lemma~\ref{L-addition for torsion} still shows $\ker g_k=\{0\}$. By Lemma~\ref{L-f < J}
we have $\rh(X_{g_k})=\rh(X_{g_k^*})$. From Lemmas~\ref{L-exact} and \ref{L-addition} we get $\rh(X_{g_k})=\frac12 \rh(X_{g_k^*g_k})$.
The rest of the proof is analogous to the proof of Theorem~\ref{T-torsion}.
\end{proof}

\begin{remark}
Note that in the trivial case, where $\Gamma=\{e\}$, one has
$$\rho( H_j(\cC_*) ) = \log | H_j(\cC_*) |.$$
In the classical cases $\Gamma=\Zb^d$, Theorem~\ref{T-milnorturaev} is a consequence of the results of Milnor \cite[page 131]{Milnor68} and Turaev \cite[Lemma 2.1.1]{Turaev}. The classical result of Milnor in the case $\Gamma=\Zb$ shows an identity of elements in the reduced $K_1$-group of the ring of rational functions $\Qb(z)$. In this case, the Mahler measure of the determinant computes the $L^2$-torsion, which gives the relationship with Theorem~\ref{T-milnorturaev}.
\end{remark}

\section{$L^2$-torsion of modules} \label{S-torsion}

We want to provide a fresh view on $L^2$-torsion and show that it is -- if set up correctly -- a completely classical torsion theory, much in the spirit of classical Reidemeister torsion. All desired properties follow from the work of Milnor \cite{Milnor66}. Throughout this section $\Gamma$ will be a countable discrete (not necessarily amenable) group.

\subsection{Whitehead torsion} \label{SS-Whitehead}

Let $R$ be a unital ring. For each $n\in \Nb$, we have the multiplicative group $\GL_n(R)$ consisting of all invertible $n\times n$ matrices over $R$. One may think of $\GL_n(R)$ as a subgroup of $\GL_{n+1}(R)$ via identifying $A\in \GL_n(R)$ with
\begin{align*}
\left(
\begin{matrix}
A &  0\\
0& 1
\end{matrix}
\right)\in \GL_{n+1}(R).
\end{align*}
Denote by $\GL_\infty(R)$ the union of $\GL_n(R)$ for all $n\in \Nb$. The {\it $K_1$-group} of $R$, denoted by $K_1(R)$, is defined as the abelian quotient group of $\GL_\infty(R)$ by its commutator subgroup $[\GL_\infty(R), \GL_\infty(R)]$ \cite[Definition 2.1.5]{Rosenberg}. The {\it reduced $K_1$-group} of $R$, denoted by $\bar{K}_1(R)$, is the quotient group of $K_1(R)$  by the image of $\{1, -1\}\subseteq \GL_1(R)$. We shall write the abelian group $\bar{K}_1(R)$ as an additive group.

In the rest of this subsection we assume that $R$ satisfies the condition \eqref{E-rank}.

For an acyclic (i.e.\ exact) chain complex $\cC_*$ of finitely generated free left $R$-modules of the form
\begin{align}  \label{E-acyclic}
0\rightarrow \cC_k \overset{\partial_k}\rightarrow \cdots \overset{\partial_2}\rightarrow \cC_1\overset{\partial_1}\rightarrow \cC_0\rightarrow 0
\end{align}
with a chosen unordered basis for each $\cC_j$, Milnor defined the {\it Whitehead torsion} $\tau(\cC_*)$ of $\cC_*$, as an element of $\bar{K}_1(R)$ \cite[Sections 3 and 4]{Milnor66}. Instead of recalling Milnor's definition, we recall the equivalent definition in \cite[Section 15]{Cohen}.

Since $\cC_*$ is a finite acyclic chain complex of free left $R$-modules, it has a {\it contraction} $\delta$, i.e. a left $R$-module homomorphism $\delta_j: \cC_j\rightarrow\cC_{j+1}$
for each $j\in \Zb$, such that $\partial_{j+1}\delta_j+\delta_{j-1}\partial_j=\id$ for every $j\in \Zb$ \cite[page 47]{Cohen}.
Set
\begin{align*}
\cC_\odd&=\sum_{j\not \in 2\Zb}\cC_j, \\
\cC_\even&=\sum_{j\in 2\Zb}\cC_j, \\
(\partial+\delta)_\odd&=(\partial+\delta)|_{\cC_\odd}:\cC_\odd\rightarrow \cC_\even.
\end{align*}
It turns out that $(\partial+\delta)_{\odd}$ is an isomorphism from $\cC_\odd$ onto $\cC_\even$ \cite[page 53]{Cohen}. The unions of the chosen unordered basis of each $\cC_j$ give rise to unordered basis of $\cC_\odd$ and $\cC_\even$ respectively. Under these bases the matrix of $(\partial+\delta)_\odd$ (up to switching rows and columns) is an element of $\GL_\infty(R)$, whose image in $\bar{K}_1(R)$ is the Whitehead torsion $\tau(\cC_*)$. (The fact that the matrix of $(\partial+\delta)_\odd$ is a square matrix uses the condition \eqref{E-rank}.)

Let $(\cC_*, \partial)$ be an acyclic  chain complex of finitely generated free left $R$-modules of finite length as in \eqref{E-acyclic}, with a chosen unordered basis for each $\cC_j$. Its {\it suspension} is the chain complex $(\Sigma\cC_*, \Sigma\partial)$  defined
by $(\Sigma\cC)_j=\cC_{j-1}$ and $(\Sigma\partial)_j=-\partial_{j-1}$ for all $j\in \Zb$ \cite[page 5]{Brown}.
Note that $\Sigma\cC_*$ is also acyclic. The chosen unordered basis of $\cC_j$ gives rise to an unordered basis of $(\Sigma\cC)_{j+1}$ naturally.  One has
\cite[page 53]{Cohen}
\begin{align} \label{E-suspension}
\tau(\Sigma\cC_*)=-\tau(\cC_*) \in \bar{K}_1(R).
\end{align}

Let
$$0 \rightarrow \cC'_* \rightarrow \cC_* \rightarrow \cC''_* \rightarrow 0$$
be a short exact sequence of chain complexes of finitely generated free left $R$-modules of finite length as in
\eqref{E-acyclic}. For each $j\in \Zb$, denote by $H_j(\cC_*)$ the $j$-th homology $\ker(\partial_j)/\im(\partial_{j+1})$ of $\cC_*$, which is a left $R$-module. Similarly, define $H_j(\cC'_*)$ and $H_j(\cC''_*)$. Then one has the long exact sequence
\begin{align} \label{E-long exact}
\cdots \rightarrow H_{j+1}(\cC'_*)\rightarrow H_{j+1}(\cC_*)\rightarrow H_{j+1}(\cC''_*)\rightarrow H_j(\cC'_*)\rightarrow H_j(\cC_*)\rightarrow H_j(\cC''_*)\rightarrow \cdots
\end{align}
of left $R$-modules \cite[Theorem 3.3]{Osborne}. It follows that if two of $\cC'_*, \cC_*$ and $\cC''_*$ are acyclic, then so is the other.
Moreover, if this is the case and for chosen unordered basis of $\cC'_j$ and $\cC''_j$ we take a left $R$-module lifting $\cC''_j\rightarrow \cC_j$ for the quotient map $\cC_j\rightarrow \cC''_j$ and endow $\cC_j$ with the unordered basis
as the union of the images of the chosen bases of $\cC'_j$ and $\cC''_j$ under $\cC'_j\rightarrow \cC_j$ and $\cC''_j\rightarrow \cC_j$ for each $j\in \Zb$, then \cite[Theorem 3.1]{Milnor66} shows that
\begin{align} \label{E-torsion addition}
\tau(\cC_*) = \tau(\cC'_*) + \tau(\cC''_*) \in \bar{K}_1(R).
\end{align}

\subsection{The Haagerup-Schultz algebra} \label{SS-HS algebra}

Let $R$ be a unital ring. A subset $S$ of $R$ is called {\it multiplicative} if $1_R\in S$, $0\not\in S$ and $ab\in S$ for all $a, b\in S$.
For a multiplicative set $S$ consisting of non-zero-divisors of $R$, the pair $(R, S)$ is said to satisfy the {\it right Ore condition} if for any $s\in S$ and $a\in R$ there exist $t\in S$ and $b\in R$ with $sb=at$. In such case,
one can form the {\it Ore localization} of $R$ with respect to $S$, denoted by $RS^{-1}$, which is a unital ring containing $R$ as a subring
such that every $s\in S$ is invertible in $RS^{-1}$ and every element of $RS^{-1}$ is of the form $as^{-1}$ for some $a\in R$ and $s\in S$  \cite[Section 10A]{Lam}. Similarly, when $(R, S)$ satisfies the left Ore condition, one can define the Ore localization $S^{-1}R$. If $(R, S)$ satisfies both the left and right Ore conditions, then $S^{-1}R=RS^{-1}$ \cite[Corollary 10.14]{Lam}.

Let $\Gamma$ be a countable discrete group.
Denote by $S$ the set of elements $g$ in $\cN\Gamma$ satisfying $\ddet_{\cN\Gamma}g>0$.
If $fg=0$ for some $f, g\in \cN\Gamma$ and $g\neq 0$, then from the injective map $\cN\Gamma\hookrightarrow \ell^2(\Gamma)$ sending $h$ to $he$ one has
$\ker f\neq \{0\}$, which implies $\ddet_{\cN\Gamma}f=0$ by Theorem~\ref{T-FK}.(4). If $gf=0$ for some $f, g\in \cN\Gamma$ and $g\neq 0$, then from $f^*g^*=0$
and Theorem~\ref{T-FK}.(2) we get $\ddet_{\cN\Gamma}f=\ddet_{\cN\Gamma}(f^*)=0$.
Thus $S$ is a multiplicative set
consisting of non-zero-divisors of $\cN\Gamma$. Furthermore,  the pair $(\cN\Gamma, S)$ satisfies both the left and right Ore conditions \cite[Lemma 2.4]{HS}. The {\it Haagerup-Schultz algebra} of $\Gamma$, denoted by $\cN\Gamma^\Delta$,  is defined as the Ore localization $S^{-1}\cN\Gamma=\cN\Gamma S^{-1}$. From Theorem~\ref{T-FK} we have $S^*=S$. Thus $\cN\Gamma^\Delta$ has a unique involution $b\mapsto b^*$ extending that of $\cN\Gamma$.

For each $d\in \Nb$, the Fuglede-Kadison determinant $\det_{\cN\Gamma}: M_d(\cN\Gamma)\rightarrow \Rb_{\ge 0}$ has a unique multiplicative extension $ \ddet_{\cN\Gamma}: M_d(\cN\Gamma^\Delta) \rightarrow \Rb_{\ge 0}$ satisfying $\ddet_{\cN\Gamma}(b^*)=\ddet_{\cN\Gamma}b$ for all $b\in M_d(\cN\Gamma^\Delta)$ \cite[Proposition 2.5]{HS}, since every element of $M_d(\cN\Gamma^\Delta)$ can be written as $a(tI_d)^{-1}$ for some $a\in M_d(\cN\Gamma)$ and $t\in S$, where $I_d$ denotes the $d\times d$ identity matrix \cite[page 301]{Lam}. From the latter fact it also follows that for any
$d\in \Nb$ and $A\in \GL_d(\cN\Gamma^\Delta)$, one has
\begin{align*}
\ddet_{\cN\Gamma}A=\ddet_{\cN\Gamma}\left(
\begin{matrix}
A &  0\\
0& 1
\end{matrix}
\right).
\end{align*}
Also note that $\ddet_{\cN\Gamma}(-1)=(\ddet_{\cN\Gamma}((-1)^2))^{1/2}=1$. Thus $\ddet_{\cN\Gamma}$ induces a group homomorphism $\bar{K}_1(\cN\Gamma^\Delta)\rightarrow \Rb_{>0}$, sending the image of $A\in \GL_d(\cN\Gamma^\Delta)$ in $\bar{K}_1(\cN\Gamma^\Delta)$ to
$\ddet_{\cN\Gamma}A$, which we still denote by $\ddet_{\cN\Gamma}$.

\begin{remark} \label{R-determinant}
L\"uck and R\o rdam showed in \cite{luckrordam} that $\ddet_{\cN \Gamma} \colon \bar{K}_1(\cN \Gamma) \to \Rb_{>0}$ is an isomorphism if $\cN \Gamma$ is a factor (this happens if and only if $\Gamma$ has only infinite non-trivial conjugacy classes).
\end{remark}

\begin{lemma} \label{L-distinct free module}
The ring $\cN\Gamma^\Delta$ satisfies the condition \eqref{E-rank}.
\end{lemma}
\begin{proof} It suffices to show that, for any $k>l$ in $\Nb$, every homomorphism $\varphi:(\cN\Gamma^\Delta)^k\rightarrow (\cN\Gamma^\Delta)^l$ of left $\cN\Gamma^\Delta$-modules fails to be injective. Choosing an ordered basis, we may identify $(\cN\Gamma^\Delta)^k$ and $(\cN\Gamma^\Delta)^l$ with
$(\cN\Gamma^\Delta)^{1\times k}$ and $(\cN\Gamma^\Delta)^{1\times l}$ respectively. Then $\varphi$ is represented by a matrix $A\in M_{k, l}(\cN\Gamma^\Delta)$. We can write $A$ as $B(tI_l)^{-1}$ for some $B\in M_{k,l}(\cN\Gamma)$ and $t\in \cN\Gamma$ satisfying $\ddet_{\cN\Gamma}t>0$, where $I_l$ denotes the $l\times l$ identity matrix \cite[page 301]{Lam}. Since $\dim_{\cN\Gamma}((\cN\Gamma)^k)=k>l=\dim_{\cN\Gamma}((\cN\Gamma)^l)$, the $\cN\Gamma$-module homomorphism
$(\cN\Gamma)^{1\times k}\rightarrow (\cN\Gamma)^{1\times l}$ represented by $B$ can not be injective. That is,
there exists a nonzero $y\in (\cN\Gamma)^{1\times k}$ satisfying $yB=0$.
Then $y$ is a nonzero element of $(\cN\Gamma^\Delta)^{1\times k}$ and $yA=0$. Thus $\varphi$ is not injective.
\end{proof}

Thus for every  acyclic chain complex $\cC_*$ of finitely generated free left $\cN\Gamma^\Delta$-modules of finite length as in \eqref{E-acyclic} with a chosen unordered basis
for each $\cC_j$, the Whitehead torsion $\tau(\cC_*)\in \bar{K}_1(\cN\Gamma^\Delta)$ is defined.

Though we do not need this fact, let us mention that the algebra $\cN\Gamma^\Delta$ can be identified with the algebra of closed and densely defined (possibly unbounded) operators $T$ on $\ell^2(\Gamma)$ affiliated with $\cN\Gamma$ satisfying $\int^\infty_1\log t\, d\mu_{|T|}(t)<\infty$, where $\mu_{|T|}$  denotes the spectral measure of $|T|$ \cite[Lemma 2.4]{HS}. The phenomenon that the usage of algebras of unbounded operators affiliated with the group von Neumann algebra simplifies algebraic matters in the theory of $L^2$-invariants has  already been used successfully in \cite{PT}.

\subsection{A new view on $L^2$-torsion}
\label{SS-newview}

Let $\cC_*$ be a (not necessarily acyclic) chain complex of finitely generated free left $\Zb\Gamma$-modules  of finite length as in \eqref{E-acyclic}.
As in \cite[Section 18]{Cohen}, we may consider the chain complex $\cN\Gamma^\Delta\otimes_{\Zb\Gamma}\cC_*$ of left $\cN\Gamma^\Delta$-modules:
$$ 0\rightarrow \cN\Gamma^\Delta\otimes_{\Zb\Gamma}\cC_k\overset{1\otimes \partial_k}\rightarrow \cdots \overset{1\otimes \partial_2}\rightarrow \cN\Gamma^\Delta\otimes_{\Zb\Gamma}\cC_1 \overset{1\otimes \partial_1}\rightarrow \cN\Gamma^\Delta\otimes_{\Zb\Gamma}\cC_0\rightarrow 0.$$
We call the chain complex $\cC_*$ {\it $\Delta$-acyclic} if $\cN\Gamma^\Delta\otimes_{\Zb\Gamma}\cC_*$ is acyclic.

Assume that $\cC_*$ is $\Delta$-acyclic and choose an unordered basis for each $\cC_j$. The latter gives rise to an unordered basis of $\cN\Gamma^\Delta\otimes_{\Zb\Gamma}\cC_j$ naturally.
Thus we have the Whitehead torsion $\tau(\cN\Gamma^\Delta\otimes_{\Zb\Gamma}\cC_*)\in \bar{K}_1(\cN\Gamma^\Delta)$ defined, and hence can define
the {\it $L^2$-torsion}
\begin{align*}
\tilde{\rho}^{(2)}(\cC_*) := \log \ddet_{\cN \Gamma} (\tau(\cN\Gamma^\Delta\otimes_{\Zb\Gamma}\cC_*)) \in \Rb.
\end{align*}
In Proposition~\ref{P-definitions of torsion} below we shall show that $\tilde{\rho}^{(2)}(\cC_*)$ coincides with $\rho^{(2)}(\cC_*)$ defined in Section~\ref{SS-torsion}. We see the main advantage of our approach in the fact that it is more algebraic; all the analysis has been put into the properties of the ring $\cN \Gamma^{\Delta}$.

\begin{definition} \label{D-determinant condition}
We say that $\Gamma$ satisfies the \emph{determinant condition} if for any $d\in \Nb$ and any $g\in M_d(\Zb\Gamma)$ with $\ker g=\{0\}$
one has $\ddet_{\cN\Gamma}g\ge 1$.
\end{definition}

\begin{remark} \label{R-sofic has determinant}
L\"{u}ck's determinant conjecture \cite[Conjecture 13.2]{Luck} says that for every group $\Gamma$, every $d\in \Nb$, and every self-adjoint $f\in M_d(\Zb\Gamma)$, one has $\ddet_{\cN\Gamma}(f+q_f)\ge 1$, where $q_f$ denotes the orthogonal projection from $(\ell^2(\Gamma))^{d\times 1}$ onto $\ker f$.
If a group $\Gamma$ satisfies the determinant conjecture, then clearly it  satisfies the determinant condition. We refer to \cite[Theorem 13.3]{Luck} for a class of groups satisfying the determinant conjecture, see also Lemma~\ref{L-integral to positive determinant} for amenable groups.
It was shown by Elek and Szab\'o \cite[page 439]{ES} that all sofic groups satisfy the determinant conjecture.
So far there are no examples of groups known to fail the determinant condition.
\end{remark}

In the rest of this section, we assume that $\Gamma$ satisfies the determinant condition. Then for any $A\in \GL_\infty(\Zb\Gamma)$, we have
$\ddet_{\cN\Gamma}A, \ddet_{\cN\Gamma}(A^{-1})\ge 1$, and
$$1=\ddet_{\cN\Gamma}(A\cdot A^{-1})=\ddet_{\cN\Gamma}(A)\cdot \ddet_{\cN\Gamma}(A^{-1}),$$
and hence $\ddet_{\cN\Gamma}A=1$. It follows that for any $\Delta$-acyclic $\cC_*$, the $L^2$-torsion $\tilde{\rho}^{(2)}(\cC_*)$ does not depend on the choice of the unordered basis for each $\cC_j$.

\begin{lemma} \label{L-acyclic}
Let $\cC_*$ be an acyclic chain complex of finitely generated free left $\Zb\Gamma$-modules  of finite length as in \eqref{E-acyclic}.
Then $\cC_*$ is $\Delta$-acyclic and
\begin{align} \label{E-zero torsion}
\tilde{\rho}^{(2)}(\cC_*)=0.
\end{align}
\end{lemma}
\begin{proof}
The chain complex $\cC_*$ has a contraction \cite[page 47]{Cohen}. It follows that
$\cN\Gamma^\Delta\otimes_{\Zb\Gamma}\cC_*$ has an induced contraction, and hence $\cC_*$ is $\Delta$-acyclic \cite[Proposition 0.3]{Brown}.
Choose an unordered basis for $\cC_j$ and endow $\cN\Gamma^\Delta\otimes_{\Zb\Gamma}\cC_j$ with the corresponding unordered basis for each $j\in \Zb$.
In Section~\ref{SS-euler} we have observed that $\Zb\Gamma$ satisfies the condition \eqref{E-rank}.
Thus the Whitehead torsions $\tau(\cC_*)\in \bar{K}_1(\Zb\Gamma)$ and $\tau(\cN\Gamma^\Delta\otimes_{\Zb\Gamma}\cC_*)\in \bar{K}_1(\cN\Gamma^\Delta)$ are defined. Moreover, clearly $\tau(\cN\Gamma^\Delta\otimes_{\Zb\Gamma}\cC_*)$ is the image of $\tau(\cC_*)$ under the natural group homomorphism $\bar{K}_1(\Zb\Gamma)\rightarrow \bar{K}_1(\cN\Gamma^\Delta)$ induced by the embedding $\Zb\Gamma\rightarrow \cN\Gamma^\Delta$. Therefore
$$ \tilde{\rho}^{(2)}(\cC_*)=\log\ddet_{\cN\Gamma}(\tau(\cN\Gamma^\Delta\otimes_{\Zb\Gamma}\cC_*))=\log \ddet_{\cN\Gamma}(\tau(\cC_*))=0.$$
\end{proof}

Let
$$0 \rightarrow \cC'_* \to \cC_* \rightarrow \cC''_* \rightarrow 0$$
be a short exact sequence of chain complexes of finitely generated free left $\Zb\Gamma$-modules of finite length as in \eqref{E-acyclic}. Since each $\cC''_j$ is a free left $\Zb\Gamma$-module, the short sequence
$$ 0 \rightarrow \cN\Gamma^\Delta\otimes_{\Zb\Gamma}\cC'_* \rightarrow \cN\Gamma^\Delta\otimes_{\Zb\Gamma}\cC_* \rightarrow \cN\Gamma^\Delta\otimes_{\Zb\Gamma}\cC''_* \rightarrow 0$$
is also exact.
Thus, if two of $\cC'_*, \cC_*$ and $\cC''_*$ are $\Delta$-acyclic, then by the discussion in Section~\ref{SS-Whitehead} so is the other one.
Moreover, if this is the case, then
 from \eqref{E-torsion addition}  we have
\begin{align} \label{E-additive}
\tilde{\rho}^{(2)}(\cC_*) = \tilde{\rho}^{(2)}(\cC'_*) + \tilde{\rho}^{(2)}(\cC''_*) \in \Rb.
\end{align}

\begin{lemma} \label{L-torsion}
Let $\cM$ be a left $\Zb\Gamma$-module of type FL. Suppose that there is a finite resolution $(\cC_*, \partial) \rightarrow \cM$
of $\cM$ by
finitely generated free left $\Zb\Gamma$-modules as in \eqref{E-resolution 2} such that the chain complex $\cC_*$ is $\Delta$-acyclic.
Then for every finite resolution $(\cC'_*, \partial') \rightarrow \cM$ of $\cM$ by finitely generated free left $\Zb \Gamma$-modules, the chain complex $\cC'_*$ is $\Delta$-acyclic and $\tilde{\rho}^{(2)}(\cC'_*)=\tilde{\rho}^{(2)}(\cC_*)$.
\end{lemma}
\begin{proof}
There exists a homotopy equivalence $\varphi: \cC_* \rightarrow \cC'_*$ of chain complexes (index by $\Zb$) of left $\Zb\Gamma$-modules \cite[Theorem I.7.5]{Brown}.
Then $1\otimes \varphi: \cN\Gamma^\Delta\otimes_{\Zb\Gamma}\cC_*\rightarrow \cN\Gamma^\Delta\otimes_{\Zb\Gamma}\cC'_*$ is a homotopy equivalence of chain complexes of $\cN\Gamma^\Delta$-modules.  Thus $1\otimes \varphi$ induces an isomorphism from the homology groups of $\cN\Gamma^\Delta\otimes_{\Zb\Gamma}\cC_*$ to those of $\cN\Gamma^\Delta\otimes_{\Zb\Gamma}\cC'_*$. Since $\cN\Gamma^\Delta\otimes_{\Zb\Gamma}\cC_*$ is acyclic, so is $\cN\Gamma^\Delta\otimes_{\Zb\Gamma}\cC'_*$.
That is, $\cC'_*$ is also $\Delta$-acyclic. Consider
the {\it mapping cone}  of $\varphi$, which is the chain complex $({\rm cone}(\varphi)_*, \partial'')$ defined by
${\rm cone}(\varphi)_j := \cC'_j\oplus (\Sigma\cC)_j$ and
$\partial''_j(x,y) = (\partial'_j(x)+\varphi_{j-1}(y), -\partial_{j-1}(y))$.
Since $\varphi$ is a  homotopy equivalence, ${\rm cone}(\varphi)_*$  is
 acyclic \cite[Proposition I.0.6]{Brown}. Thus ${\rm cone}(\varphi)_*$ is $\Delta$-acyclic and from
 \eqref{E-zero torsion} we have $\tilde{\rho}^{(2)}({\rm cone}(\varphi)_*)=0$. Note that $\cN\Gamma^\Delta\otimes_{\Zb\Gamma}\Sigma\cC_*$ is exactly
 the suspension of $\cN\Gamma^\Delta\otimes_{\Zb\Gamma}\cC_*$.  Thus from \eqref{E-suspension} we have
 $\tilde{\rho}^{(2)}(\Sigma\cC_*)=-\tilde{\rho}^{(2)}(\cC_*)$.
Also note that we have a short exact sequence
\begin{align} \label{E-cone}
0 \rightarrow  \cC'_* \rightarrow {\rm cone}(\varphi)_* \rightarrow \Sigma\cC_*  \rightarrow 0
\end{align}
of chain complexes of left $\Zb\Gamma$-modules.
Therefore
\begin{align*}
0=\tilde{\rho}^{(2)}({\rm cone}(\varphi)_*)\overset{\eqref{E-additive}}=\tilde{\rho}^{(2)}(\cC'_*)+\tilde{\rho}^{(2)}(\Sigma\cC_*)=\tilde{\rho}^{(2)}(\cC'_*)-\tilde{\rho}^{(2)}(\cC_*)
\end{align*}
as desired.
\end{proof}

We say that a left $\Zb\Gamma$-module $\cM$ of type FL is {\it $\Delta$-acyclic} if it satisfies the conditions in Lemma~\ref{L-torsion}.
For such $\cM$ we can define the {\it $L^2$-torsion} of $\cM$, denoted by $\tilde{\rho}^{(2)}(\cM)$, as $\tilde{\rho}^{(2)}(\cC_*)$.

\begin{lemma} \label{additive}
Let $0 \to \cM' \to \cM \to \cM'' \to 0$ be a short exact sequence of left $\Zb \Gamma$-modules. If two $\Zb \Gamma$-modules out of the set $\{\cM',\cM, \cM''\}$ are of type FL and $\Delta$-acyclic, then so is the third and
$$\tilde{\rho}^{(2)}(\cM) = \tilde{\rho}^{(2)} (\cM') + \tilde{\rho}^{(2)}(\cM'').$$
\end{lemma}
\begin{proof}
Let us first assume that $\cM'$ and $\cM''$ are of type FL.
By the Horseshoe lemma \cite[Proposition 6.5]{Osborne}, any two resolutions $\cC'_* \rightarrow \cM'$ and $\cC''_* \rightarrow \cM''$ consisting of free left $\Zb\Gamma$-modules, can be combined to a free resolution $\cC_* \to \cM$, such that $\cC_*$ fits into an exact sequence
\begin{align} \label{E-extension}
0 \rightarrow \cC'_* \rightarrow \cC_* \rightarrow \cC''_* \rightarrow 0
\end{align}
of chain complexes of left $\Zb\Gamma$-modules. This shows that $\cM$ is of type FL if $\cM'$ and $\cM''$ are. If $\cM'$ and $\cM$ are of type FL, then consider resolutions $\cC'_* \rightarrow \cM'$ and $\cC_* \rightarrow \cM$ consisting of free left $\Zb\Gamma$-modules. The inclusion $\cM' \rightarrow \cM$ lifts to a map of chain complexes $\varphi \colon\cC'_* \rightarrow \cC_*$ \cite[Lemma I.7.4]{Brown}. Now, the mapping cone ${\rm cone}(\varphi)_*$ of $\varphi$ fits in a short exact sequence
$$0 \rightarrow \cC_* \rightarrow {\rm cone}(\varphi)_* \rightarrow \Sigma \cC'_* \rightarrow 0$$
of chain complexes.
From the associated long exact sequence \eqref{E-long exact}, we see that $H_j({\rm cone}(\varphi)_*) = \cM''$ if $j=0$ and $H_j({\rm cone}(\varphi)_*) = 0$ otherwise. That is,
$\cC''_*:={\rm cone}(\varphi)_*$ is a resolution of $\cM''$ consisting of  free left $\Zb\Gamma$-modules. This shows that $\cM''$ is of type FL if $\cM$ and $\cM'$ are. Finally, let us assume that $\cM$ and $\cM''$ are of type FL. As in the second case, we obtain in a similar way an extension
$$0 \rightarrow \cC''_* \rightarrow {\rm cone}(\varphi')_* \rightarrow \Sigma \cC_* \rightarrow 0$$
where $\cC_* \rightarrow \cM$ and $\cC''_* \rightarrow \cM''$ are resolutions consisting of free left $\Zb\Gamma$-modules and $\varphi': \cC_*\rightarrow \cC''_*$ is a chain map
lifting $\cM\rightarrow \cM''$. The long exact sequence \eqref{E-long exact} now yields
$H_j({\rm cone}(\varphi')_*) = \cM'$ if $j=1$ and $H_j({\rm cone}(\varphi')_*) = 0$ otherwise.
In particular, the differential ${\rm cone}(\varphi')_1\rightarrow {\rm cone}(\varphi')_0$ is surjective. Since ${\rm cone}(\varphi')_0$ is a free left $\Zb\Gamma$-module, we may choose a split of the differential ${\rm cone}(\varphi')_1 \rightarrow {\rm cone}(\varphi')_0$ and define a new chain complex $\cC'_j := {\rm cone}(\varphi')_{j+1}$ for $j \ge 2$ or $j=0$, and $\cC'_1:= {\rm cone}(\varphi')_2 \oplus {\rm cone}(\varphi')_0$. The differentials are defined in the obvious way using the split. It is easy to see that $\cC'_*$ is a resolution of $\cM'$ consisting of free left $\Zb\Gamma$-modules. This shows that $\cM'$ of type FL if $\cM$ and $\cM''$ are.

Now we may assume that $\cM'$ and $\cM''$ are of type FL.
The proof is finished in view of the short exact sequence \eqref{E-extension}  using \eqref{E-additive}.
\end{proof}

\begin{proposition} \label{P-definitions of torsion}
Let $\cC_*$ be a chain complex of finitely generated free left $\Zb\Gamma$-modules of finite length as in \eqref{E-acyclic}.
Then $\cC_*$ is $\Delta$-acyclic if and only if the chain complex $\ell^2(\Gamma) \otimes_{\Zb \Gamma} \cC_*$  is weakly acyclic.
Moreover, in such case one has
\begin{align} \label{E-defitions of torsion}
\tilde{\rho}^{(2)}(\cC_*) = \rho^{(2)}(\cC_*).
\end{align}
\end{proposition}
\begin{proof} Choose an ordered basis for each $\cC_j$, and denote the rank
of $\cC_j$ by $d_j$. Then the differential $\partial_j:\cC_j\rightarrow \cC_{j-1}$ is represented by a matrix $f_j\in M_{d_j\times d_{j-1}}(\Zb\Gamma)$.
The chosen ordered basis of
$\cC_j$ gives rise to an ordered basis of $\cN\Gamma^\Delta\otimes_{\Zb\Gamma}\cC_j$ naturally. Thus we may identify $\cN\Gamma^\Delta\otimes_{\Zb\Gamma}\cC_j$ with
$(\cN\Gamma^\Delta)^{1\times d_j}$, and the $\cN\Gamma^\Delta$-module homomorphism $1\otimes \partial_j$ is also represented by $f_j$.
Note that $\Delta_j:=f_{j+1}^*f_{j+1}+f_jf_j^*$ is a self-adjoint element of $M_{d_j}(\Zb\Gamma)$.

We prove the ``only if'' part first. Assume that $\cC_*$ is $\Delta$-acyclic. Then $\cN\Gamma^\Delta\otimes_{\Zb\Gamma}\cC_*$ is an acyclic chain complex
of free left $\cN\Gamma^\Delta$-modules of finite length. Thus it has a contraction $\delta$ \cite[page 47]{Cohen}.
Say, $\delta_j$ is represented by the matrix $g_j\in M_{d_j\times d_{j+1}}(\cN\Gamma^\Delta)$. Then
\begin{align} \label{E-weakly acyclic}
g_jf_{j+1}+f_jg_{j-1}=I_{d_j}
\end{align}
whenever $d_j>0$, where $I_{d_j}$ denotes the $d_j\times d_j$ identity matrix. Suppose that $\ell^2(\Gamma) \otimes_{\Zb \Gamma} \cC_*$ fails to be weakly acyclic. Then there exists
some $j\in \Zb$ with $d_j>0$ such that the closed linear subspace $V:=\{y\in (\ell^2(\Gamma))^{1\times d_j}: yf_j=yf_{j+1}^*=0\}$ of $(\ell^2(\Gamma))^{1\times d_j}$ is nonzero.
We may write $g_j$ and $g_{j-1}$ as $a^{-1}h_j$ and $h_{j-1}b^{-1}$ respectively for some $h_j\in M_{d_j\times d_{j+1}}(\cN\Gamma)$, $h_{j-1}\in M_{d_{j-1}\times d_j}(\cN\Gamma)$, and $a, b\in \cN\Gamma$ satisfying $\ddet_{\cN\Gamma}a, \ddet_{\cN\Gamma}b>0$ \cite[page 301]{Lam}.

We claim that there exists some nonzero $y\in (\ell^2(\Gamma))^{1\times d_j}$ with $ya\in V$.
An argument similar to that in Section~\ref{SS-determinant} shows that the orthogonal projection from $(\ell^2(\Gamma))^{1\times d_j}$ onto $V$ is given by some projection $q\in M_d(\cN\Gamma)$, i.e. $P(x)=xq$ for all $x\in (\ell^2(\Gamma))^{1\times d_j}$.
Consider the polar decomposition of the operator $T\in B((\ell^2(\Gamma))^{1\times d_j})$ sending $x$ to $xa(I_{d_j}-q)$: there
exist unique $U, S\in B((\ell^2(\Gamma))^{1\times d_j})$ satisfying that $S\ge 0$, $\ker U=\ker S=\ker T$, $U$ is an isometry from the orthogonal complement
of $\ker T$ onto the closure of $\im T$, and $T=US$ \cite[Theorem 6.1.2.]{KR2}. Another argument similar to that in Section~\ref{SS-determinant} shows that there is some $u\in M_{d_j}(\cN\Gamma)$  such that $U(x)=xu$  for all $x\in (\ell^2(\Gamma))^{1\times d_j}$.
Suppose that $xa\not \in V$ for every nonzero $x\in (\ell^2(\Gamma))^{1\times d_j}$. Then $T$ is injective, and hence $uu^*=I_{d_j}$.
Note that both $I_{d_j}-q$ and $u^*u$ are projections and $u^*u\le I_{d_j}-q$. Thus
$$d_j=\tr_{\cN\Gamma}(uu^*)=\tr_{\cN\Gamma}(u^*u)\le \tr_{\cN\Gamma}(I_{d_j}-q)=d_j-\tr_{\cN\Gamma}q,$$
and therefore $\tr_{\cN\Gamma}q=0$. Since $\tr_{\cN\Gamma}$ is faithful, we get $q=0$, which contradicts that $V$ is nonzero.
Therefore there is some nonzero $y\in (\ell^2(\Gamma))^{1\times d_j}$ with $ya\in V$.

From \eqref{E-weakly acyclic} we have
$$ h_jf_{j+1}b+af_jh_{j-1}=abI_{d_j}.$$
Thus
$$ yh_jf_{j+1}b=y(h_jf_{j+1}b+af_jh_{j-1})=yab.$$
Since $\ddet_{\cN\Gamma}(b^*)=\ddet_{\cN\Gamma}b>0$ (resp.~$\ddet_{\cN\Gamma}(a^*)=\ddet_{\cN\Gamma}a>0$), by Theorem~\ref{T-FK} the linear map $\ell^2(\Gamma)\rightarrow \ell^2(\Gamma)$ sending $z$ to $b^*z$ (resp.~$a^*z$) is injective. Taking adjoints we find that  the linear map $\ell^2(\Gamma)\rightarrow \ell^2(\Gamma)$ sending $z$ to $zb$ (resp.~$za$) is also injective.
Thus $yh_jf_{j+1}=ya\in V$. Then
$$ \|yh_jf_{j+1}\|_2^2=\left<yh_jf_{j+1}, yh_jf_{j+1}\right>=\left<yh_jf_{j+1}f_{j+1}^*, yh_j\right>=0,$$
and hence $ya=yh_jf_{j+1}=0$. Therefore $y=0$, which contradicts our choice of $y$. Thus $\ell^2(\Gamma)\otimes_{\Zb\Gamma}\cC_*$ is weakly acyclic.

Next we prove the ``if'' part. Assume that $\ell^2(\Gamma) \otimes_{\Zb \Gamma} \cC_*$ is weakly acyclic.
If $y\in (\ell^2(\Gamma))^{1\times d_j}$ and $y\Delta_j=0$, then $yf_{j+1}^*=0$ and $yf_j=0$, i.e. $y\in \ker(1\otimes \partial_{j+1})^*\cap \ker(1\otimes \partial_j)$. Since $\ell^2(\Gamma) \otimes_{\Zb \Gamma} \cC_*$ is weakly acyclic, $\ker(1\otimes \partial_j)=\overline{\im(1\otimes \partial_{j+1})}$ and hence
$y\in \ker(1\otimes \partial_{j+1})^*\cap \overline{\im(1\otimes \partial_{j+1})}=\{0\}$.
Thus the linear map $(\ell^2(\Gamma))^{1\times d_j}\rightarrow (\ell^2(\Gamma))^{1\times d_j}$ sending $y$ to $y\Delta_j$ is injective.
Taking adjoints, we find that the linear map $(\ell^2(\Gamma))^{d_j\times 1}\rightarrow (\ell^2(\Gamma))^{d_j\times 1}$ sending $z$ to $\Delta_jz$ is also injective. Since $\Gamma$ satisfies the determinant condition, we have $\ddet_{\cN\Gamma}\Delta_j\ge 1$, whenever $d_j>0$.
Since $\Delta_j$ is a self-adjoint element of $M_{d_j}(\cN\Gamma)$, Kadison showed that there exists
$U_j\in M_{d_j}(\cN\Gamma)$ satisfying $U^*_jU_j=U_jU_j^*=I_{d_j}$ such that
$U_j\Delta_jU_j^{-1}$ is diagonal \cite[Theorem 3.19]{Kadison}.
Note that
$\ddet_{\cN\Gamma}U_j=(\ddet_{\cN\Gamma}(U_j^*U_j))^{1/2}=1$. Thus $\ddet_{\cN\Gamma}(U_j\Delta_jU_j^{-1})=\ddet_{\cN\Gamma}\Delta_j\ge 1$.
Since $U_j\Delta_jU_j^{-1}$ is diagonal, its Fuglede-Kadsion determinant is the product of the Fuglede-Kadison determinants of the diagonal entries.
It follows that $U_j\Delta_jU_j^{-1}$ is invertible in $M_{d_j}(\cN\Gamma^\Delta)$. Then so is $\Delta_j$.
Define a left $\cN\Gamma^\Delta$-module homomorphism $\delta_j:(\cN\Gamma^\Delta)^{1\times d_j}\rightarrow (\cN\Gamma^\Delta)^{1\times d_{j+1}}$ to be the one represented by the matrix $\Delta_j^{-1}f_{j+1}^*$ when $d_j>0$, and be $0$ when $d_j=0$.

We claim that the module map $(1\otimes \partial_{j+1})\delta_j+\delta_{j-1}(1\otimes \partial_j)$ is the identity map on $(\cN\Gamma^\Delta)^{1\times d_j}$.
Consider first the case  $d_j, d_{j-1}>0$.
Then $(1\otimes \partial_{j+1})\delta_j+\delta_{j-1}(1\otimes \partial_j)$ is represented by the matrix
\begin{align*}
\Delta_j^{-1}f_{j+1}^*f_{j+1}+f_j\Delta_{j-1}^{-1}f_j^*.
\end{align*}
Note that
$$ \Delta_jf_j=f_jf_j^*f_j=f_j\Delta_{j-1},$$
and hence $\Delta_j^{-1}f_j=f_j\Delta_{j-1}^{-1}$. Therefore
\begin{align*}
\Delta_j^{-1}f_{j+1}^*f_{j+1}+f_j\Delta_{j-1}^{-1}f_j^*=\Delta_j^{-1}f_{j+1}^*f_{j+1}+\Delta_j^{-1}f_jf_j^*=I_{d_j}.
\end{align*}
This proves our claim in the case $d_j, d_{j-1}>0$. The other cases can be dealt with easily. Thus $\delta$ is a contraction of the chain complex $\cN\Gamma^\Delta\otimes_{\Zb\Gamma}\cC_*$. Therefore $\cN\Gamma^\Delta\otimes_{\Zb\Gamma}\cC_*$ is acyclic, i.e. $\cC_*$ is $\Delta$-acyclic.

Finally we use the contraction $\delta$ constructed above to prove \eqref{E-defitions of torsion}. The chosen ordered basis of $\cC_j$ gives rise to
ordered basis of $(\cN\Gamma^\Delta\otimes_{\Zb\Gamma}\cC)_\odd$ and $(\cN\Gamma^\Delta\otimes_{\Zb\Gamma}\cC)_\even$. Denote by $A$ the matrix in $M_d(\cN\Gamma^\Delta)$ representing
$(\partial+\delta)_\odd$, where $d=\sum_{j\not \in 2\Zb}d_j$. Note that $\Delta_j^{-1}f_{j+1}^*f_j^*=0$ and $f_{j+2}(\Delta_j^{-1}f_{j+1}^*)^*=f_{j+2}f_{j+1}\Delta_j^{-1}=0$ whenever $d_j>0$. It follows that the matrix $AA^*$ is block-diagonal with the diagonal blocks being $\Delta_j^{-1}f_{j+1}^*f_{j+1}\Delta_j^{-1}+f_jf_j^*$ for odd $j$ with $d_j>0$.
Therefore
\begin{align*}
\tilde{\rho}^{(2)}(\cC_*)&=\log\ddet_{\cN\Gamma}A \\
&=\frac12 \log\ddet_{\cN\Gamma}(AA^*) \\
&=\frac12 \sum_{j\not\in 2\Zb, d_j>0}\log \ddet_{\cN\Gamma}(\Delta_j^{-1}f_{j+1}^*f_{j+1}\Delta_j^{-1}+f_jf_j^*)\\
&=\frac12 \sum_{j\not\in 2\Zb, d_j>0}(\log \ddet_{\cN\Gamma}(f_{j+1}^*f_{j+1}+\Delta_jf_jf_j^*\Delta_j)-2\log \ddet_{\cN\Gamma}\Delta_j)\\
&=\frac12 \sum_{j\not\in 2\Zb, d_j>0}(\log \ddet_{\cN\Gamma}(f_{j+1}^*f_{j+1}+(f_jf_j^*)^3)-2\log \ddet_{\cN\Gamma}(f_{j+1}^*f_{j+1}+f_jf_j^*)).
\end{align*}
When $d_j>0$, since $f_{j+1}^*f_{j+1}$ and $f_jf_j^*$ are self-adjoint, $f_{j+1}^*f_{j+1}\cdot f_jf_j^*=f_jf_j^*\cdot f_{j+1}^*f_{j+1}=0$, and
$\ker(f_{j+1}^*f_{j+1}+f_jf_j^*)=\{0\}$, we have $ f_{j+1}^*f_{j+1}+(f_jf_j^*)^m=(f_{j+1}^*f_{j+1}+q_{f_{j+1}})(f_jf_j^*+q_{f_j^*})^m$ for all $m\in \Nb$, and hence
\begin{eqnarray*}
& &\log \ddet_{\cN\Gamma}(f_{j+1}^*f_{j+1}+(f_jf_j^*)^3)-2\log \ddet_{\cN\Gamma}(f_{j+1}^*f_{j+1}+f_jf_j^*)\\
&=&-\log \ddet_{\cN\Gamma}(f_{j+1}^*f_{j+1}+q_{f_{j+1}})+\log \ddet_{\cN\Gamma}(f_jf_j^*+q_{f_j^*}).
\end{eqnarray*}
Therefore
\begin{align*}
\tilde{\rho}^{(2)}(\cC_*)&=\frac12 \sum_{j\not\in 2\Zb, d_j>0}(-\log \ddet_{\cN\Gamma}(f_{j+1}^*f_{j+1}+q_{f_{j+1}})+\log \ddet_{\cN\Gamma}(f_jf_j^*+q_{f_j^*}))\\
&\overset{\eqref{E-symmetry}}=\frac12 \sum_{j=1}^\infty(-1)^{j+1}\log \ddet_{\cN\Gamma}(f_j^*f_j+q_{f_j})\\
&= \rho^{(2)}(\cC_*).
\end{align*}
\end{proof}

From now on we shall write $\tilde{\rho}^{(2)}(\cC_*)$ as $\rho^{(2)}(\cC_*)$.
We may consider the trivial left $\Zb\Gamma$-module $\Zb$ corresponding to the trivial action of $\Gamma$ on $\Zb$.
If the trivial left $\Zb \Gamma$-module $\Zb$ is of type FL and $\Delta$-acyclic, we may consider the $L^2$-torsion of $\Zb$ and call it the {\it $L^2$-torsion of $\Gamma$}, which we shall denote by $\rho^{(2)}(\Gamma)$.

When the trivial left $\Zb\Gamma$-module $\Zb$ is of type FL, we define the {\it Euler characteristic} of $\Gamma$, denoted by $\chi(\Gamma)$, to be the Euler characteristic of $\Zb$ as a left $\Zb \Gamma$-module.

For a subgroup $\Lambda$ of $\Gamma$, note that $\cN\Lambda$ is naturally a subalgebra of $\cN\Gamma$, and $\ddet_{\cN\Gamma}g=\ddet_{\cN\Lambda}g$ for all $d\in \Nb$ and $g\in M_d(\cN\Lambda)$.
It follows that if $\Gamma$ satisfies the determinant condition, then so does $\Lambda$. Furthermore, $\cN\Lambda^\Delta$ is a subalgebra of $\cN\Gamma^\Delta$, and $\ddet_{\cN\Gamma}g=\ddet_{\cN\Lambda}g$ for all $d\in \Nb$ and $g\in M_d(\cN\Lambda^\Delta)$.

\begin{lemma} \label{L-euler}
Let
$$1 \to \Lambda \to \Gamma \to \Gamma/\Lambda \to 1$$
be a short exact sequence of groups. Assume that $\Gamma$  satisfies the determinant condition, that
the trivial left  $\Zb\Lambda$-module $\Zb$ is of type FL and $\Delta$-acyclic, and that
the trivial left $\Zb[\Gamma/\Lambda]$-module $\Zb$ is of type FL. Then the trivial left $\Zb\Gamma$-module $\Zb$ is of type FL and $\Delta$-acyclic, and
$$\rho^{(2)}(\Gamma) = \chi(\Gamma/\Lambda) \cdot \rho^{(2)}(\Lambda).$$
\end{lemma}
\begin{proof}
Note that $\Zb\Gamma$ is a free right $\Zb\Lambda$-module. Thus the functor $\Zb \Gamma \otimes_{\Zb \Lambda} ?$ from the category of left $\Zb\Lambda$-modules to the category of left $\Zb\Gamma$-modules is exact.
Take a resolution $\cC_*\rightarrow \Zb$ of $\Zb$ by finitely generated free left $\Zb\Lambda$-modules of finite length as in \eqref{E-resolution 2}.
Then $\Zb\Gamma\otimes_{\Zb\Lambda}\cC_*\rightarrow \Zb[\Gamma/\Lambda]$ is a resolution of $\Zb\Gamma\otimes_{\Zb\Lambda}\Zb=\Zb[\Gamma/\Lambda]$ by finitely generated free left $\Zb\Gamma$-modules of finite length. Moreover, we have
$$\cN\Gamma^{\Delta}\otimes_{\Zb\Gamma}(\Zb\Gamma\otimes_{\Zb\Lambda}\cC_*)=\cN\Gamma^\Delta\otimes_{\Zb\Lambda}\cC_*=\cN\Gamma^\Delta\otimes_{\cN\Lambda^\Delta}(\cN\Lambda^\Delta\otimes_{\Zb\Lambda}\cC_*).$$
Thus, an argument similar to that in the proof of \eqref{E-zero torsion} shows that $\Zb\Gamma\otimes_{\Zb\Lambda}\cC_*$ is $\Delta$-acyclic and
$$\rho^{(2)}(\Zb\Gamma\otimes_{\Zb\Lambda}\cC_*)=\rho^{(2)}(\cC_*).$$
Therefore $\rho^{(2)}(\Zb[\Gamma/\Lambda])=\rho^{(2)}(\Lambda)$.

Take a resolution of the trivial left  $\Zb[\Gamma/\Lambda]$-module $\Zb$ by finitely generated free left $\Zb[\Gamma/\Lambda]$-modules
\begin{align*}
0 \rightarrow (\Zb [\Gamma/\Lambda])^{d_k}\rightarrow \cdots \rightarrow (\Zb[\Gamma/\Lambda])^{d_1}\rightarrow (\Zb[\Gamma/\Lambda])^{d_0}\rightarrow \Zb \rightarrow 0.
\end{align*}
Treat the above exact sequence as a sequence of left $\Zb\Gamma$-modules.
By Lemma~\ref{additive} the left $\Zb\Gamma$-module $(\Zb[\Gamma/\Lambda])^d$ is of type FL and $\Delta$-acyclic, and
$\rho^{(2)}((\Zb[\Gamma/\Lambda])^d)=d \cdot \rho^{(2)}(\Zb[\Gamma/\Lambda])$ for every natural number $d$.
For each $1\le j<k$ denote by $\cM_{j-1}$ the image of the homomorphism $(\Zb [\Gamma/\Lambda])^{d_j}\rightarrow (\Zb [\Gamma/\Lambda])^{d_{j-1}}$.
From the short exact sequence
$$0\rightarrow (\Zb [\Gamma/\Lambda])^{d_k}\rightarrow (\Zb [\Gamma/\Lambda])^{d_{k-1}}\rightarrow \cM_{k-2}\rightarrow 0,$$
by Lemma~\ref{additive} we know that $\cM_{k-2}$ is of type FL and $\Delta$-acyclic, and
$$\rho^{(2)}((\Zb [\Gamma/\Lambda])^{d_{k-1}})=\rho^{(2)}((\Zb [\Gamma/\Lambda])^{d_k})+\rho^{(2)}(\cM_{k-2}).$$
Similarly, by induction we conclude that
$\cM_{k-3}, \cM_{k-4}, \dots, \cM_0$ and the trivial left $\Zb\Gamma$-module $\Zb$ are all of type FL and $\Delta$-acyclic, and
$$ \rho^{(2)}((\Zb [\Gamma/\Lambda])^{d_j})=\rho^{(2)}(\cM_j)+\rho^{(2)}(\cM_{j-1})$$
for all $j=k-2, k-1, \dots, 1$, and
$$ \rho^{(2)}((\Zb [\Gamma/\Lambda])^{d_0})=\rho^{(2)}(\cM_0)+\rho^{(2)}(\Zb).$$
Therefore
\begin{align*}
\rho^{(2)}(\Gamma) &=\rho^{(2)}(\Zb)\\
&=\sum_{j=0}^{k} (-1)^j \rho^{(2)}((\Zb[\Gamma/\Lambda])^{d_j})\\
&=\left(\sum_{j=0}^{k} (-1)^j d_j \right)\cdot \rho^{(2)}(\Zb[\Gamma/\Lambda]) = \chi(\Gamma/\Lambda) \cdot \rho^{(2)}(\Lambda).
\end{align*}
\end{proof}

\begin{theorem} \label{T-vanishing torsion}
Let $\Gamma$ be a countable discrete group which satisfies the determinant condition and admits a sequence of subgroups
$$\Gamma_0 \subseteq \Gamma_1 \subseteq \cdots \subseteq \Gamma_{n+1}=\Gamma,$$
such that $\Gamma_i$ is normal in $\Gamma_{i+1}$ for all $0 \leq i \leq n$, the trivial left $\Zb\Gamma_0$-module $\Zb$ and the trivial left $\Zb[\Gamma_{i+1}/\Gamma_i]$-module $\Zb$ are of type FL for all $0 \leq i \leq n$, and $\Gamma_0$ is non-trivial and amenable. Then, the trivial left $\Zb\Gamma$-module $\Zb$ is of type FL and $\Delta$-acyclic, and  $\rho^{(2)}(\Gamma)=0$.
\end{theorem}
\begin{proof}
This follows from a straightforward induction argument using Lemma~\ref{L-euler} and Theorem~\ref{T-Luck conjecture}.
\end{proof}

\section*{Acknowledgements} This work started while we visited the Institut Henri Poincar\'{e} in June 2011. We thank the institute for nice environment and financial support.

H.L. was partially supported by NSF Grants DMS-0701414 and DMS-1001625. Part of this work was carried out while H.L. visited  the math department of University of Science and Technology of China in Summer  2011 and the Erwin Schr\"{o}dinger International Institute for Mathematical Physics in Fall 2011. He thanks the Institute for financial support, and Wen Huang, Klaus Schmidt, Song Shao, and Xiangdong Ye for warm hospitality.

A.T. was supported by the ERC Starting Grant 277728. He thanks Roman Sauer for inspiring discussions.

We both thank Christopher Deninger, Russell Lyons, Varghese Mathai, Thomas Ward and the referee for helpful comments.


\begin{thebibliography}{99}
\Small

\bibitem{AF}
F. W. Anderson and K. R. Fuller. {\it Rings and Categories of Modules}. Second edition. Graduate Texts in Mathematics, 13. Springer-Verlag, New York, 1992.


\bibitem{BB}
M. Bestvina and N. Brady.
Morse theory and finiteness properties of groups.
{\it Invent. Math.} {\bf 129} (1997), no. 3, 445--470.

\bibitem{Billingsley}
P. Billingsley. {\it Convergence of Probability Measures}. Second edition.
John Wiley \& Sons, Inc., New York, 1999.

\bibitem{Boca}
F.-P. Boca. {\it Rotation $C^*$-algebras and Almost Mathieu Operators}. Theta Series in Advanced Mathematics, 1. The Theta Foundation, Bucharest, 2001.

\bibitem{Bowen10}
L. Bowen. Measure conjugacy invariants for actions of countable sofic groups.
{\it J.\ Amer.\ Math.\ Soc.} {\bf 23}  (2010), 217--245.

\bibitem{Bowen11}
L. Bowen. Entropy for expansive algebraic actions of residually finite groups. {\it Ergod.\ Th.\ Dynam. Sys.} {\bf 31} (2011), no. 3, 703--718.

\bibitem{BL}
L. Bowen and H. Li. Harmonic models and spanning forests of residually finite groups.  {\it J.\ Funct.\ Anal.} {\bf 263} (2012), no. 7, 1769--1808.


\bibitem{BCFM}
M. Braverman, A. L. Carey, M. Farber, and V. Mathai. $L^2$ torsion without the determinant class condition and extended $L^2$ cohomology. {\it Commun. Contemp. Math.} {\bf 7} (2005), no. 4, 421--462.

\bibitem{Brown}
K. S. Brown. {\it Cohomology of Groups}. Corrected reprint of the 1982 original. Graduate Texts in Mathematics, 87. Springer-Verlag, New York, 1994.

\bibitem{CFM}
A. L. Carey, M. Farber, and V. Mathai. Determinant lines, von Neumann algebras and $L^2$ torsion. {\it J. Reine Angew. Math.} {\bf 484} (1997), 153--181.

\bibitem{carmat}
A. L. Carey and V. Mathai.
$L^2$-acyclicity and $L^2$-torsion invariants. In: {\it Geometric and Topological Invariants of Elliptic Operators (Brunswick, ME, 1988)}, pp. 91--118,
Contemp. Math., 105, Amer. Math. Soc., Providence, RI, 1990.

\bibitem{carmat2}
A. L. Carey and V. Mathai.
$L^2$-torsion invariants.
{\it J. Funct. Anal.} {\bf 110} (1992), no. 2, 377--409.

\bibitem{CC}
T. Ceccherini-Silberstein and M. Coornaert. {\it Cellular Automata and Groups}. Springer Monographs in Mathematics. Springer-Verlag, Berlin, 2010.


\bibitem{CL}
N.-P. Chung and H. Li. Homoclinic groups, IE groups, and expansive algebraic actions. Preprint, 2011.

\bibitem{Danilenko}
A. I. Danilenko. Entropy theory from the orbital point of view. {\it Monatsh. Math.} {\bf 134} (2001), no. 2, 121--141.

\bibitem{Cohen}
M. M. Cohen.
{\it A Course in Simple-Homotopy Theory}.
Graduate Texts in Mathematics, Vol. 10. Springer-Verlag, New York-Berlin, 1973.


\bibitem{mixedmot}
C. Deninger. Deligne periods of mixed motives, K-theory and the entropy of certain $\Zb^n$-actions.
{\it J. Amer. Math. Soc.} {\bf 10} (1997), no. 2, 259--281.

\bibitem{Deninger06}
C. Deninger. Fuglede-Kadison determinants and entropy
for actions of discrete amenable groups.
{\it J.\ Amer.\ Math.\ Soc.} {\bf 19} (2006), no. 3, 737--758.

\bibitem{Deninger09}
C. Deninger. Mahler measures and Fuglede--Kadison determinants.  {\it M\"{u}nster J. Math.}  {\bf 2}  (2009), 45--63.

\bibitem{deningerregulator}
C. Deninger. Regulators, entropy and infinite determinants.  In: {\it Regulators}, pp. 117--134, Contemp. Math., 571, Amer. Math. Soc., Providence, RI, 2012.

\bibitem{DS}
C. Deninger and K. Schmidt. Expansive algebraic actions of discrete
residually finite amenable groups and their entropy.
{\it Ergod.\ Th.\ Dynam.\ Sys.} {\bf 27} (2007), no. 3, 769--786.

\bibitem{DLMSY}
J. Dodziuk, P. Linnell, V. Mathai, T. Schick, and S. Yates.
Approximating $L^2$-invariants and the Atiyah conjecture.
Dedicated to the memory of J\"{u}rgen K. Moser.
{\it Comm. Pure Appl. Math.} {\bf 56} (2003), no. 7, 839--873.

\bibitem{DM}
J. Dodziuk and V. Mathai. Approximating $L^2$-invariants of amenable covering spaces: a combinatorial approach. {\it J. Funct. Anal.} {\bf 154} (1998), no. 2, 359--378.

\bibitem{EG}
S. Eilenberg and T. Ganea.
On the Lusternik-Schnirelmann category of abstract groups.
{\it Ann. of Math. (2)} {\bf 65} (1957), 517--518.



\bibitem{Elek03}
G. Elek. On the analytic zero divisor conjecture of Linnell.
{\it Bull. London Math. Soc.}  {\bf 35}  (2003),  no. 2, 236--238.

\bibitem{ES}
G. Elek and E. Szab\'o. Hyperlinearity, essentially free actions and $L^2$-invariants. The sofic property.
{\it Math. Ann.} {\bf 332} (2005), no. 2, 421--441.

\bibitem{FK}
B. Fuglede and R. V. Kadison. Determinant theory in finite factors.
{\it Ann. of Math. (2)}  {\bf 55} (1952), 520--530.

\bibitem{GK}
F. P. Gantmacher and M. G. Krein. {\it Oscillation Matrices and Kernels and Small Vibrations of Mechanical Systems}. Revised edition. Translation based on the 1941 Russian original. AMS Chelsea Publishing, Providence, RI, 2002.

\bibitem{Geoghegan}
R. Geoghegan. {\it Topological Methods in Group Theory}. Graduate Texts in Mathematics, 243. Springer, New York, 2008.


\bibitem{HS}
U. Haagerup and H. Schultz.
Brown measures of unbounded operators affiliated with a finite von Neumann algebra.
{\it Math. Scand.} {\bf 100} (2007), no. 2, 209--263.

\bibitem{hellow}
H. Helson and D. Lowdenslager.
Prediction theory and Fourier series in several variables.
{\it Acta Math.} {\bf 99}  (1958), 165--202.

\bibitem{JB}
C. R. Johnson and W. W. Barrett. Spanning-tree extensions of the Hadamard-Fischer inequalities. {\it Linear Algebra Appl.} {\bf 66} (1985), 177--193.

\bibitem{Kadison}
R. V. Kadison. Diagonalizing matrices. {\it Amer. J. Math.} {\bf 106} (1984), no. 6, 1451--1468.


\bibitem{KR2}
R. V. Kadison and J. R. Ringrose. {\it Fundamentals of the Theory of Operator Algebras. Vol. II. Advanced Theory.}
Graduate Studies in Mathematics, 16. American Mathematical Society, Providence, RI, 1997.


\bibitem{KR1}
R. V. Kadison and J. R. Ringrose. {\it Fundamentals of the Theory of Operator Algebras. Vol. I. Elementary Theory.}
Graduate Studies in Mathematics, 15. American Mathematical Society, Providence, RI, 1997.

\bibitem{KL11}
D. Kerr and H. Li. Entropy and the variational principle for actions of sofic groups. {\it Invent. Math.} {\bf 186} (2011), no. 3, 501--558.


\bibitem{KK}
A. K. Kelmans and B. N. Kimelfeld.
Multiplicative submodularity of a matrix's principal minor as a function of the set of its rows, and some combinatorial applications.
{\it Discrete Math.} {\bf 44} (1983), no. 1, 113--116.



\bibitem{krop}
P. H. Kropholler, C. Martinez-P\'erez, and B. E. A. Nucinkis. Cohomological finiteness conditions for elementary amenable groups.
{\it J. Reine Angew. Math.} {\bf 637} (2009), 49--62.

\bibitem{Lam}
T. Y. Lam. {\it Lectures on Modules and Rings}.
Graduate Texts in Mathematics, 189. Springer-Verlag, New York, 1999.

\bibitem{Lang}
S. Lang. {\it Algebra}. Revised third edition. Graduate Texts in Mathematics, 211. Springer-Verlag, New York, 2002.

\bibitem{Li}
H. Li. Compact group automorphisms, addition formulas and Fuglede-Kadison determinants. {\it Ann. of Math. (2)} {\bf 176} (2012), no. 1, 303--347.



\bibitem{LSV10}
D. Lind, K. Schmidt, and E. Verbitskiy.
Entropy and growth rate of periodic points of algebraic $\Zb^d$-actions. In: {\it Dynamical Numbers: Interplay between Dynamical Systems and Number Theory}, pp. 195--211, Contemp. Math., 532, Amer. Math. Soc., Providence, RI, 2010.

\bibitem{LSV11}
D. Lind, K. Schmidt, and E. Verbitskiy.
Homoclinic points, atoral polynomials, and periodic points of algebraic $\Zb^d$-actions. {\it Ergod.\ Th.\ Dynam. Sys.} to appear.

\bibitem{LSW}
D. Lind, K. Schmidt, and T. Ward. Mahler measure and
entropy for commuting automorphisms of compact groups.
{\it Invent.\ Math.}  {\bf 101}  (1990), 593--629.

\bibitem{LW}
E. Lindenstrauss and B. Weiss. Mean topological dimension.
{\it Israel J.\ Math.} {\bf 115} (2000), 1--24.

\bibitem{linnik}
I. Ju. Linnik. A multidimensional analogue of G. Szeg\H{o}'s limit theorem. (Russian)
{\it Izv. Akad. Nauk SSSR Ser. Mat.} {\bf 39} (1975), no. 6, 1393--1403, 1439.

\bibitem{Luck94}
W. L\"{u}ck. Approximating $L^2$-invariants by their finite-dimensional analogues.
{\it Geom. Funct. Anal.}  {\bf 4}  (1994),  no. 4, 455--481.

\bibitem{Luck02}
W. L\"{u}ck. $L^2$-invariants of regular coverings of compact manifolds and CW-complexes. In: {\it Handbook of Geometric Topology}, pp. 735--817, North-Holland, Amsterdam, 2002.

\bibitem{Luck}
W. L\"{u}ck. {\it $L^2$-Invariants: Theory and Applications to Geometry and $K$-theory}.
Springer-Verlag, Berlin, 2002.

\bibitem{luckrordam}
W. L\"{u}ck and M. R{\o}rdam.
Algebraic K-theory of von Neumann algebras.
{\it K-Theory} {\bf 7} (1993), no. 6, 517--536.

\bibitem{LR}
W. L\"{u}ck and M. Rothenberg.
Reidemeister torsion and the K-theory of von Neumann algebras.
{\it K-Theory} {\bf 5} (1991), no. 3, 213--264.

\bibitem{lsweg}
W. L\"{u}ck, R. Sauer, and C. Wegner.
$L^2$-torsion, the measure-theoretic determinant conjecture, and uniform measure equivalence.
{\it J. Topol. Anal.} {\bf 2} (2010), no. 2, 145--171.

\bibitem{Lyons05}
R. Lyons. Asymptotic enumeration of spanning trees.  {\it Combin. Probab. Comput.} {\bf 14} (2005), no. 4, 491--522.

\bibitem{Lyons10}
R. Lyons. Identities and inequalities for tree entropy. {\it Combin. Probab. Comput.} {\bf 19} (2010), no. 2, 303--313.

\bibitem{Milnor66}
J. Milnor. Whitehead torsion.
{\it Bull. Amer. Math. Soc.} {\bf 72} (1966), 358--426.

\bibitem{Milnor68}
J. Milnor. Infinite cyclic coverings. In: {\it Conference on the Topology of Manifolds (Michigan State Univ., E. Lansing, Mich., 1967)}, pp. 115--133. Prindle, Weber \& Schmidt, Boston, Mass., 1968.

\bibitem{JMO}
J. Moulin Ollagnier. {\it Ergodic Theory and Statistical Mechanics.}
Lecture Notes in Math., 1115. Springer, Berlin, 1985.

\bibitem{OW}
D. S. Ornstein and B. Weiss. Entropy and isomorphism theorems for
actions of amenable groups.  {\it J. Analyse Math.} {\bf 48} (1987),
1--141.

\bibitem{Osborne}
M. S. Osborne. {\it Basic Homological Algebra}. Graduate Texts in Mathematics, 196. Springer-Verlag, New York, 2000.


\bibitem{Passman}
D. S. Passman. {\it The Algebraic Structure of Group Rings}.
Pure and Applied Mathematics. Wiley-Interscience [John Wiley \& Sons], New York-London-Sydney, 1977.

\bibitem{Peters}
J. Peters. Entropy on discrete abelian groups.
{\it Adv. in Math.} {\bf 33} (1979), no. 1, 1--13.

\bibitem{PT}
J. Peterson and A. Thom.
Group cocycles and the ring of affiliated operators.
{\it Invent.\ Math.} {\bf 185} (2011), no. 3, 561--592.

\bibitem{Rosenberg}
J. Rosenberg. {\it Algebraic K-theory and its Applications}. Graduate Texts in Mathematics, 147. Springer-Verlag, New York, 1994.


\bibitem{Schick}
T. Schick. $L^2$-determinant class and approximation of $L^2$-Betti numbers.
{\it Trans. Amer. Math. Soc.}  {\bf 353}  (2001),  no. 8, 3247--3265.

\bibitem{Schmidt}
K. Schmidt. {\it Dynamical Systems of Algebraic Origin}.
Progress in Mathematics, 128. Birkh\"{a}user Verlag, Basel, 1995.



\bibitem{Simon}
B. Simon. {\it Orthogonal Polynomials on the Unit Circle. Part 1. Classical theory.} American Mathematical Society Colloquium Publications, 54, Part 1. American Mathematical Society, Providence, RI, 2005.

\bibitem{Solomyak}
R. Solomyak. On coincidence of entropies for two classes of dynamical systems. {\it Ergod.\ Th.\ Dynam.\ Sys.}
 {\bf 18} (1998), no. 3, 731--738.


\bibitem{Szego}
G. Szeg\H{o}. Ein Grenzwertsatz \"{u}ber die Toeplitzschen Determinanten einer reellen positiven Funktion.  {\it Math. Ann.} {\bf 76} (1915), no. 4, 490--503.

\bibitem{Takesaki}
M. Takesaki. {\it Theory of Operator Algebras. I.} Encyclopaedia of Mathematical Sciences, 124. Operator Algebras and Non-commutative Geometry, 5. Springer-Verlag, Berlin, 2002.






\bibitem{Turaev}
V. Turaev. Reidemeister torsion in knot theory.
{\it Uspekhi Mat. Nauk}, {\bf 41} (1986), no. 1, 97--147.
Translated in {\it Russian Math. Surveys} {\bf 41} (1986), no. 1, 119--182.

\bibitem{turaevbook}
V. Turaev. {\it Torsions of $3$-dimensional Manifolds.}
Progress in Mathematics, 208. Birkh\"auser Verlag, Basel, 2002.

\bibitem{Wall65}
C. T. C. Wall. Finiteness conditions for CW-complexes. {\it Ann. of Math. (2)} {\bf 81} (1965), 56--69.

\bibitem{Wall66}
C. T. C. Wall. Finiteness conditions for CW complexes. II. {\it Proc. Roy. Soc. Ser. A} {\bf 295} (1966), 129--139.

\bibitem{Walters}
P. Walters. {\it An Introduction to Ergodic Theory.} Graduate Texts in
Mathematics, 79. Springer-Verlag, New York, Berlin, 1982.

\bibitem{weg}
C. Wegner.
$L^2$-invariants of finite aspherical CW-complexes.
{\it Manuscripta Math.} {\bf 128} (2009), no. 4, 469--481.

\bibitem{yup}
S. A. Yuzvinski\u{\i}. Computing the entropy of a group of endomorphisms. (Russian)  {\it Sibirsk. Mat. \^{Z}.}  {\bf 8}  (1967), 230--239.
Translated in {\it Siberian Math. J.} {\bf 8} (1967), 172--178.

\end{thebibliography}
\end{document}